\documentclass[11pt]{article}

\newcommand{\insieme}[1]{\left\{ #1 \right\}}
\usepackage{fullpage}
\usepackage{subfig}

\usepackage{enumerate}
\usepackage{amsmath,amsfonts,amssymb,amsthm, bbold}
\usepackage{svg}
\usepackage{graphics,esint}
\usepackage{epsfig}
\usepackage{xcolor}
\usepackage{xspace}
\definecolor{pingreen}{rgb}{0,39,14}

\usepackage{xfrac}
\usepackage{comment}
\usepackage{hyperref}
\usepackage{cleveref}
\newcommand\scalare[1]{{\left\langle #1 \right\rangle}}

\crefname{section}{§}{§§}
\Crefname{section}{§}{§§}

\def\vol{\mathrm{vol}}
\def\lok{\mathrm{loc}}
\def\Int{\mathrm{Int}}

\def\Acal{\mathcal{A}}
\def\Bcal{\mathcal{B}}

\def\Hcal{\mathcal{H}}

\def\rhs{r.h.s.\xspace}
\def\lhs{l.h.s.\xspace}
\def\st{s.t.\xspace}

\newcommand{\hausd}{\mathcal H} 

\makeatletter
\newtheorem*{rep@theorem}{\rep@title}
\newcommand{\newreptheorem}[2]{%
\newenvironment{rep#1}[1]{%
 \def\rep@title{#2 \ref{##1}}%
 \begin{rep@theorem}}%
 {\end{rep@theorem}}}
\makeatother

\newtheorem{theorem}{Theorem}[section]

\newtheorem{lemma}[theorem]{Lemma}
\newtheorem{definition}[theorem]{Definition}

\newtheorem{remark}[theorem]{Remark}

\newtheorem{proposition}[theorem]{Proposition}

\newreptheorem{theorem}{Theorem}

\newcommand{\R}{\mathbb{R}}
\newcommand{\Z}{\mathbb{Z}}

\newcommand{\Fcal}{\mathcal{F}}
\newcommand{\Gcal}{\mathcal{G}}

\def\discrete{\mathrm{dsc}}
\def\discreteToCont{{\mathrm{d}\to\mathrm{c}}}

\def\Fj{F_{J,L}}

\def\Scal{\mathcal S}
\def\Scalper{\mathcal S^{\mathrm{per}}}

\def\T{\mathbb{T}}

\def\N{\mathbb N}
\def\eps{\varepsilon}
\def\Fj{\tilde{\mathcal F}_{J,L}}

\def\per{\mathrm{Per}}

\def\eps{\varepsilon}

\def\d {\,\mathrm {d}}
\def\dx{\,\mathrm {d}x}
\def\dz{\,\mathrm {d}z}
\def\ds{\,\mathrm {d}s}
\def\du{\,\mathrm {d}u}
\def\dv{\,\mathrm {d}v}
\def\dt{\,\mathrm {d}t}
\def\dy{\,\mathrm {d}y}

\def\GtL{\mathcal G_{\tau,L}}
\def\FtL{\mathcal F_{\tau,L}}

\numberwithin{equation}{section}

\usepackage{authblk}

\author[1]{Sara Daneri\thanks{daneri@math.fau.de}}
\author[2]{Eris Runa\thanks{eris.runa@mis.mpg.de}}
\affil[1]{FAU Erlangen--N\"urnberg }
\affil[2]{Max Planck Institut f\"ur Mathematik in den Naturwissenschaften, Leipzig}

  \title{Exact periodic stripes for minimizers of a local/nonlocal interaction functional in general dimension}

\date{}

\begin{document}

\maketitle

\begin{center}
  {\em Dedicated to Prof. Stephan Luckhaus on the occasion of his 65th birthday.  }
\end{center}
  \vskip 7mm

\begin{abstract}
We study the functional considered in~\cite{2011PhRvB..84f4205G,2014CMaPh.tmp..127G,GiuSeirGS} and a continuous version of it, analogous to the one considered in~\cite{GR}. 
The functionals consist of a perimeter term and a nonlocal term which are in competition.  
For both the continuous and discrete problem, we show that the global minimizers are exact periodic stripes. 
One striking feature of the functionals is that the minimizers are invariant under a smaller group of symmetries than the functional itself. 
In the continuous setting, to our knowledge this is the first example of a model with local/nonlocal terms in competition such that the functional is invariant under permutation of coordinates and the minimizers display a pattern formation which is one-dimensional. Such behaviour for a smaller range of exponents in the discrete setting was already shown in~\cite{GiuSeirGS}.

\end{abstract}

\section{Introduction}
\label{sec:introduction}


%

In this paper we study a discrete local/nonlocal functional considered in a series of papers by Giuliani, Lebowitz, Lieb and Seiringer (cf.~\cite{2011PhRvB..84f4205G,2014CMaPh.tmp..127G,GiuSeirGS}) and a continuous version of it. 

The discrete functional is the following: given $E\subset \Z^d$, $d\geq1$, $L\in \N$ 
\begin{equation} 
   \label{eq:giuliani}
   \begin{split}
      \tilde{\mathcal F}_{J,L}^{\discrete} (E) := \frac{1}{L^d}\Big(J \sum_{x\in [0,L)^d\cap\Z^d}\sum_{y\sim x}|\chi_{E}(x) -\chi_{E}(y) |  - \sum_{\underset{x\neq y}{x\in [0,L)^d\cap \Z^d,\,y\in \Z^{d}}} \frac{|\chi_{E}(x)-\chi_{E}(y) |}{|x-y|^{p}} \Big ),
   \end{split}
\end{equation} 
where $J$ is a positive constant, $p\geq d+2$, $y\sim x$ if $x$ and $y$ are neighbouring points in the lattice, and 
\begin{equation*} 
   \begin{split}
      \chi_{E}(x) := 
      \begin{cases}
         1 & \text{if $x\in E$,}\\
         0 & \text{otherwise. }
      \end{cases}
   \end{split}
\end{equation*}

In the continuous setting we consider the following functional: for $E\subset\R^d$, $d\geq1$, $L>0$ 
\begin{equation}\label{E:F}
\tilde{\mathcal F}_{J,L}(E)=\frac{1}{L^d}\Big(J\per_{1}(E,[0,L)^d)-\int_{[0,L)^d}\int_{\R^d}{|\chi_E(x)-\chi_E(y)|}{K_1(x-y)}\dy\dx\Big),
\end{equation}
where $J$ is a positive constant,
\begin{equation*} 
   \begin{split}
      \per_{1}(E,[0,L)^d):=\int_{\partial E\cap [0,L)^d}|\nu^E(x)|_1\, d\mathcal H^{d-1}(x),\quad \text{$|z|_1=\sum_{i=1}^d|z_i|$},
   \end{split}
\end{equation*} 
with $\nu^E(x)$ exterior normal to $E$ in $x$, is the $1$-perimeter of $E$ and $K_1$ is a kernel. 
The general assumptions on $K_1$ will be specified in \eqref{eq:Kbound}-\eqref{eq:laplpos}.
One particular kernel in this class is 
\begin{equation}\label{e:k_1} 
   \begin{split}
     K_{1} (\zeta) = \frac{1}{(|\zeta|_{1} + 1)^{p}}.
   \end{split}
\end{equation} 

Both of the discrete and the continuous models describe systems of particles in which there is a short-range attracting force (the perimeter) and a repulsive long-range force (nonlocal term).

%
%

The behaviour of the functional is very similar to the one considered in~\cite{GR}, namely the kernel has the same scaling,  $\frac{1}{C}\frac{1}{(|\zeta|+1)^p}\leq K_1(\zeta)\leq C\frac{1}{(|\zeta|+1)^p}$, and retains the same symmetries as the functional considered in~\cite{GR}, but \eqref{e:k_1} has the advantage that, due to the positivity of the inverse Laplace transform of $\zeta_d  \mapsto \int_{\R^{d-1}}K_1(\zeta_1,\ldots,\zeta_d) \d\zeta_1\cdots\d\zeta_{d-1}$, the reflection positivity technique can be applied (see Section \ref{sec:1D_problem}).

The aim of this paper  is to study the structure of the minimizers of the discrete and continuous functionals.

In order to make our problem well-posed we impose periodic boundary conditions, namely we restrict the functional to $[0,L)^d$-periodic sets.  
However, as we will see from the statement of Theorems~\ref{T:1.3} and \ref{T:1.6}, due to the fact that our result is  independent  of  $L$, this will not be a restriction.

   For both the discrete and continuous problem, there exists a critical constant $J^\discrete_{c}$ ($J_{c}$ respectively) such that if $J>J^{\discrete}_{c}$ (respectively $J> J_c$), then the global minimizers are trivial, namely either empty or the whole domain.

The critical constants $J^{\discrete}_c$ and $J_c$ are (as proven in \cite{2011PhRvB..84f4205G} for the discrete setting and in \cite{GR} for the continuous case)
\begin{equation} \label{eq:jc}
   \begin{split}
      J_{c}^{\discrete}:=\sum_{\substack {y_1>0,\\(y_2,\ldots,y_d)\in\Z^{d-1}}}\frac{y_1}{(y_1^2+ \cdots + y_d^2)^{p\slash2}}
      \qquad \text{and} \qquad
      J_{c}:= \int_{\R^d}{|\zeta_1|}K_{1}(\zeta)\d\zeta.
   \end{split}
\end{equation}

When $J=J^{\discrete}_c-\tau$ (resp. $J=J_c-\tau$) with $0<\tau\leq \bar{\tau}$ for some $\bar{\tau}>0$ small enough, it has been conjectured that the minimizers should be periodic unions of stripes of optimal period.

In order to define what we mean by unions of stripes, let us fix a canonical basis $\{ e_i\}_{i=1}^d$. 
A union of stripes in the continuous setting is a $[0,L)^d$-periodic set which is, up to Lebesgue null sets, of the form $V_i^\perp+\hat Ee_i$ for some $i\in\{1,\dots,d\}$, where $V_i^\perp$ is the $(d-1)$-dimensional subspace orthogonal to $e_i$ and  $\hat{E}\subset \R$ with $\hat E\cap [0,L)=\cup_{k=1}^N(s_i,t_i)$. 
A union of stripes is periodic if $\exists\, h>0$, $\nu\in\R$ s.t. $\hat E\cap [0,L)=\cup_{k=0}^N(2kh+\nu,(2k+1)h+\nu)$.
In the following, we will also sometimes  call unions of stripes simply stripes. 

   The family of unions  of stripes will be denoted by $\Scal$, and the family of  unions of stripes which are $[0,L)^d$-periodic will be denoted by $\Scal^{\mathrm{per}}_L$.

   In the discrete setting the concept of union of stripes is the same, up to intersecting with the discrete lattice and considering $h,L\in \N$.
As for periodic stripes, the width and distance of the intersection of the continuous stripes with the discrete lattice must be the same. 
In particular, if the period $L$ is a prime number there can be no periodic stripes of period less than $L$, thus the only $[0,L)^d$-periodic stripes in the discrete setting are $\Z^d$ and $\emptyset$.

We will show that the conjecture on the structure of minimizers holds both for the discrete and continuous setting.  For the discrete setting, a different proof was already given in \cite{GiuSeirGS} for the smaller range of exponents $p > 2d$. 

We would like to point out that our method applies if, instead of periodic boundary conditions, we imposed optimal periodic Dirichlet boundary conditions, namely asking that the sets $E$ are, outside $[0,L)^d$, periodic unions of stripes of optimal period. 
We prefer periodic boundary conditions since the problem is invariant under coordinate exchange, and thus we are not preselecting a particular direction.  

If one optimizes among periodic unions of stripes, then unions of stripes with optimal energy for $\tilde{\mathcal F}_{J,L}^{\discrete}$ (respectively $\tilde{\mathcal F}_{J,L}$) have width and distance of order $\tau^{-1\slash{(p-d-1)}}$ and energy of order $\tau^{(p-d)\slash(p-d-1)}$. 
Thus it is natural to rescale the functional in such a way that  the width and the energy of the stripes are of order $O(1)$ for $\tau$ small.  
Letting $\beta:= p-d-1$, the rescaling is the following:
\begin{equation}\label{eq:cv1}
   x:=\tau^{-1/\beta}\tilde{x}, \quad L:=\tau^{-1/\beta}\tilde{L} \quad \textrm{and} \quad \Fj(E):= \tau^{(p-d)/\beta}  \mathcal{F}_{\tau,\tilde L}( \tilde{E}).
\end{equation}

   In particular, notice that since the minimizers of the rescaled functional $\Fcal_{\tau,L}$ are unions of stripes, also the minima of $\tilde\Fcal_{J,L}$ with $J=J_c-\tau$ are unions of stripes.

Making the substitutions in \eqref{eq:cv1}, letting also $\zeta=\tau^{-1\slash\beta}\tilde{\zeta}$ and in the end dropping the tildes, one has that the rescaled functional in the continuous setting is given by
\begin{equation}\label{eq:ftauel0}
\begin{split}
\mathcal F_{\tau,L}(E)=\frac{1}{L^d}\Big(-\per_1(E,[0,L)^d)&+\int_{\R^d} K_\tau(\zeta) \Big[\int_{\partial E \cap [0,L)^d} \sum_{i=1}^d|\nu^E_i(x)| |\zeta_i|\d\mathcal H^{d-1}(x)\\
&-\int_{[0,L)^d}|\chi_E(x)-\chi_E(x+\zeta)|\dx\Big]\d\zeta\Big),
\end{split}
\end{equation}
where   $K_\tau(\zeta)=\tau^{-p/\beta}K_1(\zeta\tau^{-1/\beta})$.

  In the discrete setting, for $E\subset\kappa\Z^d$ $[0,L)^d$-periodic, where  $\kappa=\tau^{1/\beta}$ and $L$ is a multiple of $\kappa$, let 
   \begin{equation} 
      \label{eq:definizioneTildakappa}
      \begin{split}
         \tilde{E}^{\kappa}:=\bigcup _{i\in E} \Big(i + [-\kappa/2,\kappa/2)^{d}\Big).
         \qquad
         \per_{1,\kappa}(E,[0,L)^d) := \sum_{x\in [0,L)^d\cap \kappa \Z^{d}} \sum_{y\sim x} |\chi_{E}(x)-\chi_{E}(y)| \kappa^{d-1}.
      \end{split}
   \end{equation} 
   Then, the rescaled functional has the form
   \begin{align}
      \Fcal_{\tau,L}^{\discrete}(E)=\frac{1}{L^d}\Big(-\per_{1,\kappa}(E,[0,L)^d)&+\sum_{\zeta\in\kappa\Z^{d}} K^\discrete_{\kappa}(\zeta) \Big[\int_{\partial \tilde{E}^{\kappa} \cap [0,L)^d} \sum_{i=1}^d|\nu^{\tilde{E}^{\kappa}}_i(x)| |\zeta_i|\d\mathcal H^{d-1}(x)\notag \\
      &-\int_{[0,L)^d}|\chi_{\tilde{E}^{\kappa}}(x)-\chi_{\tilde{E}^{\kappa}}(x+\zeta)|\dx\Big]\Big)\label{eq:ftaudiscintro},
   \end{align}
   where $K^\discrete_{\kappa}(\zeta)=\frac{\kappa^d}{|\zeta|^p}$.


   For fixed $\tau > 0$, consider first for all $L > 0$ the minimal value obtained by $\Fcal_{\tau,L}$ on $[0,L)^d$-periodic stripes (denoted above by $\Scalper_L$)
  and then the minimal among these values as $L$ varies in $(0,+\infty)$. We will denote this value by $C^*_\tau$, namely

  \begin{equation*}
     \begin{split}
        C^*_{\tau} := \inf_{L>0} \  \inf_{E\in \Scalper_L} \Fcal_{\tau,L}(E)  
     \end{split}
  \end{equation*}
  
  By the reflection positivity technique (see~Section~\ref{sec:1D_problem}), this value is attained on periodic stripes of width and distance $h^*_\tau > 0$.

  As in the continuum, in the discrete setting one can define similarly $C^{*,\discrete}_{\tau}$ and $h^{*,\discrete}_\tau$.


 


In general, both for the discrete and the continuous setting, one does not expect $h^*_\tau$ and $h^{*,\discrete}_\tau$ to be unique (for the discrete setting see~\cite{GiuSeirGS}). 


For the continuous setting, our main theorems are the following.

In our first theorem, we show that for  kernels  satisfying assumptions \eqref{eq:Kbound}-\eqref{eq:laplpos},  the period $h^*_{\tau}$ is unique, provided $0<\tau\leq\hat{\tau}$ with $\hat{\tau}>0$ depending on the chosen family of kernels.

\begin{theorem}\label{T:main0}
Let $d\geq1$, $p\geq d+2$ and a family of kernels $\{K_{\tau}\}$ satisfying \eqref{eq:Kmon}-\eqref{eq:laplpos}. Then there exists $\hat{\tau}>0$ s.t. whenever $0< \tau < \hat{\tau}$, $h^*_\tau$ is unique. 
\end{theorem}

In the next two theorems, we deal with the  occurrence of  pattern formation for $\Fcal_{\tau,L}$.

\begin{theorem}\label{T:main}
Let $d\geq1$, $p\geq d+2$, $L>0$. Then there exists $\bar{\tau}>0$ such that $\forall\,0<\tau\leq\bar\tau$ there exists $h_{\tau,L}$  such that  the minimizers of $\mathcal F_{\tau,L}$ are periodic stripes of width and distance $h_{\tau,L}$.
\end{theorem}

The next theorem shows that $h_{\tau,L}$ is close to   $h^*_\tau$ whenever $L$ is large.

\begin{theorem}\label{T:main2}
There exists a constant $C$ such that for every $0<\tau\leq\bar{\tau}$, one has that the width $h_{\tau,L}$ of a minimizer of $\mathcal F_{\tau,L}$ satisfies
\begin{equation}
|h_\tau^*-h_{\tau,L}|\leq \frac CL.
\end{equation}

  \end{theorem} 

     In Theorem~\ref{T:main} the constant $\bar \tau$ depends on $L$.  One expects $\bar{\tau}$ to be independent on $L$.  In this respect, when $L$ is of the form $L= 2kh^{*}_{\tau}$, the independence is shown in Theorem~\ref{T:1.3}, namely $\tau_0$ does not depend on $L$ if $L=2kh^{*}_{\tau}$. 

  \begin{theorem}\label{T:1.3}
        Let $d\geq1$, $p\geq d+2$ and $h^{*}_{\tau}$ be the optimal stripes' width for fixed $\tau$. 
    Then there exists $\tau_{0}$, such that for every $\tau< \tau_{0}$, one has that for every $k\in \N$ and  $L = 2k h_{\tau}^{*}$,   the minimizers $E_{\tau}$ of $\FtL$ are optimal stripes of width $h_{\tau}^{*}$. 
  \end{theorem}

         Notice that  the periodic boundary conditions were imposed in order to give sense to the functional which is otherwise not well-defined. If one is interested to show that optimal periodic stripes of width and distance $h^*_\tau$ are "optimal" if one varies also the periodicity, then it is not difficult  to see that Theorem~1.3 is sufficient. This corresponds to the "thermodynamic limit" and is relevant in physics.


For the discrete setting, choosing now $h^{*,\discrete}_\tau$ equal to one of the admissible optimal periodic widths for $\tau>0$ we prove equivalently:

\begin{theorem}\label{T:1.6}
   Let $d\geq1,\,p\geq d+2$ and $h^{*,\discrete}_{\tau}$ be an optimal width for the optimal periodic stripes. Then there exists $\tau_0>0$ s.t. $\forall\,0<\tau<\tau_0$ and $L=2kh_\tau^{*,\discrete}$, $k\in\N$, the minimizers of $\mathcal F^{\discrete}_{\tau,L}$ are periodic stripes of width $h_\tau^{*,\discrete}$.
\end{theorem}

As already noticed, in the discrete for a union of stripes $E$ to be $[0,L)^d$-periodic of width and distance $h$  one has that $L/h \in 2\N$. Given that $h\in \tau^{1/\beta}\Z$  this is not always possible. Therefore, there can not be an analogous statement to Theorem~\ref{T:main} and Theorem~\ref{T:main2}.

\subsection{Scientific Context}

The competition between short-range and long-range forces is  at the base of pattern formation in many areas of physics and biology (see e.g.~\cite{de2000dipolar,harrison2000mechanisms,kohn1992branching,chakrabarty2011modulation,pasta}).

In dimension $d=1$, there are many instances in which  pattern formation is rigorously shown (see~e.g.~\cite{muller1993singular,chenOshPeriodicity,glllRP1}).

However, in dimension $d\geq2$, showing pattern formation is a rather difficult problem which is rigorously solved in very few models: in the discrete case, to our knowledge, only in~\cite{GiuSeirGS,theil2006proof,bourne2012optimality}, and, in the continuous setting, in~\cite{MR2864796}. On the general issue of crystallization see~\cite{BlancLewin}.

In the continuous setting, the closest and most famous model is the sharp interface version of the Ohta-Kawasaki~\cite{OhtaKawasaki} model. This functional is well-studied (e.g.~\cite{MR2338353,CicSpa,KnMu,ACO,pasta,muller1993singular,choksi2010small, goldman2013gamma,knupfer2016low,MoriniSternberg}). Even though periodic pattern formation is expected due to physical experiments and numerical simulations (e.g.~\cite{Seul476,MR2338353}), the problem is still open.

  One of the main difficulties is that the minimizers are invariant under a smaller group of symmetries than the functional itself. This is sometimes called breaking of symmetry.  
  In the continuous setting, to our knowledge this is the first example of a model with local/nonlocal terms in competition such that the functional is invariant under permutation of coordinates and the minimizers display a pattern formation which is one-dimensional. Such behaviour for a smaller range of exponents in the discrete setting was already shown in~\cite{GiuSeirGS}.

As already stated in the beginning of the introduction, in this paper we are  considering questions that have already been studied in a series of papers in~\cite{glllRP1,glllRP2,2011PhRvB..84f4205G,2014CMaPh.tmp..127G,MR2864796,GiuSeirGS}. For this reason we would like discuss some of the similarities and differences to the most recent paper~\cite{GiuSeirGS}.
   In~\cite{GiuSeirGS}, the discrete setting is considered. 
   It can be shown that their setting is equivalent to ours. 

      In the smaller range of exponents $p> 2d$, they prove a result similar to Theorem~\ref{T:1.6}.  We improve their result to the range of exponents $p\geq d+2$.  Our results  can be viewed as progress towards the aim of proving pattern formation for the more  ``physical'' exponents. Among them we recall the case of thin magnetic films ($p=d+1$  see~e.g. \cite{Seul476,glllRP1}), 3D micromagnetics ($p=d$~see~e.g.~\cite{landaulif,hueberschafer,KnMuMicroMagnetics})  and diblock copolymers ($p=d-2$ ~see~e.g.~\cite{OhtaKawasaki}).

Very broadly speaking, the general strategy in this paper has some similarities to~\cite{GiuSeirGS}. Both use  a two-scale approach in order to identify regions in which the set resembles a union stripes (what in \cite{GiuSeirGS} is called ``good'' regions/tiles) and regions in which such resemblance does not hold (``bad'' regions/tiles).
However the criteria on which ``good'' and ``bad'' regions are chosen is different in the two papers.
In \cite{GiuSeirGS}, this corresponds to a localization in terms of droplets and in our setting this corresponds to an averaging argument (e.g.~\eqref{eq:gstr14}). 

In both papers, the goal is to show that not only ``bad'' regions are never convenient but also that in a ``good'' region it is convenient to be ``flat''.  
The way to achieve this is different in the two papers. 


In \cite{GiuSeirGS}, the deviation from being a stripe is measured in terms of ``angles'' and ``holes''. Then, a lower bound in terms of ``angles'' and ``holes'' is shown.

In the continuum one would like to find a  characterization of nonoptimality in terms of geometric quantities. However, the discrete concepts namely ``angle'' and ``hole'' in the continuum are ill-posed. 
Moreover, a characterization of the geometry cannot reduce to just ``angles'' and ``holes'' due to $E\subset \R^d$ and not $E\subset \Z^d$. For this reason one needs to find a decomposition of the functional into terms that measure in a certain sense how much a minimizer deviates from being a union of stripes. 
Such quantities have been introduced in \cite{GR} (see Section \ref{sec:setting_and_preliminary_results}).
The penalization for not being a union of stripes is expressed through the  rigidity estimate (see Proposition~\ref{prop:rigidity}).
The two-scale approach, although widely used in applied analysis, in this context appeared for the first time in~\cite{MR2864796}.

   A second common point to \cite{GiuSeirGS} is the use of the  technique of reflection positivity, which was introduced in the context of quantum field theory in  \cite{osterNatale} and applied  for the first time to statistical mechanics in \cite{FroSim}. For the first appearance of the technique in models with short-range and Coulomb type interactions see \cite{fro}. 
     For further generalizations to one-dimensional models see \cite{glllRP1,glllRP3,glllRP3bis} and for applications to two-dimensional models see \cite{glllRP2,MR2864796}.
   Such technique allows to show that minimal stripes must be periodic.

   \vskip 2mm
   In \cite{GR}, for a smaller range of exponents ($p > 2d$ instead of $p \geq d+2$) a rigidity estimate was shown, leading to prove that minimizers of $\Fcal_{\tau,L}$ converge in $L^1$ to periodic stripes as $\tau\downarrow 0$. 
In the present paper, we show that pattern formation really appears not only for $\tau$ tending to $0$ (as was done  in \cite{GR}) but for a positive fixed $\tau$, in the range $p\geq d+2$. For this we need a new rigidity argument and a stability result (namely, Lemma~\ref{lemma:new_lemma_pre1}). 
In the rigidity estimate we use a result of~\cite{Bre} (see Section~\ref{sec:d+2}).  
Moreover, we show that in case $L$ is an even multiple of the optimal period $h^*_\tau$,  such $\tau$ does not depend on how big $L$ is.


M. Goldman,  B. Merlet and  V. Millot communicated to us that they have an alternative proof of Proposition~\ref{prop:rigidity}.

\subsection{Structure of the paper}
   This paper is organized as follows: 
   in Section~\ref{sec:setting_and_preliminary_results}, we explain the setting and some preliminary result that will be used in the following; 
   in Section \ref{sec:d+2} we improve the key estimate of \cite{GR}, namely the rigidity estimate,  to exponents $p\geq d+2$ (See~Theorem~\ref{thm:gammaconv} and Theorem~\ref{T:gammaconvNew}); 
   in Section~\ref{sec:1D_problem} we show Theorem \ref{T:main0} and that for $\tau > 0$ if one minimizes $\Fcal^{\discrete}_{\tau,L}$ or $\Fcal_{\tau,L}$ among sets which are union  of stripes then the minimizers are periodic stripes (this is done by the so-called reflection positivity); 
   in Section~\ref{sec:discrete}, we show analogous results to~\cite{GR} for the discrete setting using our technique; 
   in Section~\ref{sec:structure_of_minimizers},  we  prove Theorems \ref{T:main} and \ref{T:main2};
   in Section \ref{sec:taul} we prove Theorem~\ref{T:1.3} and Theorem~\ref{T:1.6}, namely that the parameter $\tau>0$ such that minimizers are periodic stripes can be chosen independently of $L$, provided $L$ is an even multiple of $h^*_\tau$.

\section{Setting and Preliminary results}
\label{sec:setting_and_preliminary_results}

In this section, we set the notation and recall some preliminary results on the functional \eqref{E:F} proven in~\cite{GR} which will be used in the proof of our main theorems. 

\subsection{Notation and preliminary definitions} 

In the following, we let $\N=\{1,2,\dots\}$, $d\geq 1$. On $\R^d$, we let $\langle\cdot,\cdot\rangle$ be the Euclidean scalar product and $|\cdot|$ be the Euclidean norm.  We let $(e_1,\dots,e_d)$ be the canonical basis in $\R^d$ and  for $\zeta\in\R^d$ we let $\zeta_i=\langle\zeta,e_i\rangle e_i$ and $\zeta_i^\perp:=\zeta-\zeta_i$. For $z\in\R^d$, let 
$|z|_1=\sum_{i=1}^d|z_i|$ be its $1$-norm  and $|z|_\infty=\max_i|z_i|$ its $\infty$-norm.

Given a measurable set $A\subset\R^d$, let us denote by $\mathcal H^{d-1}(A)$ its $(d-1)$-dimensional Hausdorff measure and $|A|$ its Lebesgue measure.

Given a measurable function $f:\R^d\to\R$, $Df$ denotes its distributional derivative.

For a measure $\mu$ on $\R^d$, we denote by $|\mu|$ its total variation. 


 We are now ready to recall the definition of set of locally finite perimeter (see \cite{AFPBV}).  Such a property is fundamental because the functional~\eqref{E:F} is finite only on $[0,L)^d$-periodic sets of locally finite perimeter. 

\begin{definition}
A set $E\subset\R^d$ is of (locally) finite perimeter if the distributional derivative of $\chi_E$ is a (locally) finite measure. We let $\partial E$ be the reduced boundary of $E$, namely the set of points $x\in\mathrm{spt}(D\chi_E)$ such that the limit
\[
\nu^E(x):=-\underset{r\downarrow0}{\lim}\frac{D\chi_E(B(x,r))}{|D\chi_E|(B(x,r))}
\] 
exists and satisfies $|\nu^E(x)|=1$. We call $\nu^E$ the exterior normal to $E$. In particular, $D\chi_E=-\nu^E\mathcal H^{d-1}\llcorner\partial E$.
\end{definition} 

Notice that a $[0,L)^d$-periodic set (as those considered in the paper) can not be of finite perimeter.  For this reason we need to introduce the sets of locally finite perimeter.

We define now (up to multiplying by a positive constant $J$) the first term of the functional \eqref{E:F}, namely
\[
   \per_1(E,[0,L)^d):=\int_{\partial E\cap [0,L)^d}|\nu^E(x)|_1\d\mathcal H^{d-1}(x)
\]
and, for $i\in\{1,\dots,d\}$ 
\begin{equation}
   \label{eq:perI}
   \per_{1i}(E,[0,L)^d)=\int_{\partial E\cap [0,L)^d}|\nu^E_i(x)|\d\mathcal H^{d-1}(x),
\end{equation}
thus $\per_1(E,[0,L)^d)=\sum_{i=1}^d\per_{1i}(E,[0,L)^d)$. Notice that in the definition of $\per_1$ the norm applied to the exterior normal $\nu_E$ is not isotropic. For more general reference on anisotropic surface energies see \cite{Maggi}.

Because of  periodicity, w.l.o.g.  we always assume that $|D\chi_E|(\partial [0,L)^d)=0$.

Now, let us look at the definition and the assumptions on the second term of \eqref{E:F}, namely the nonlocal term. Given a function $K_1:\R^d\to\R$, one defines it as

\[
-\int_{ \R^d}\int_{[0,L)^d}K_1(\zeta)|\chi_E(x)-\chi_E(x+\zeta)|\dx\d\zeta.
\]

Now we set our assumptions on $K_1$. They will be expressed in terms of the rescaled kernels 
\begin{equation}
   \label{eq:Krescale}
   \begin{split}
K_\tau(\zeta)=\tau^{-p/\beta}K_1(\zeta\tau^{-1/\beta}), \quad\text{with $\tau>0$ small.}
   \end{split}
\end{equation}
 The reason for this is that, as anticipated in the Introduction and explained in Subsection~\ref{subsec:preliminary}, we will not consider directly the functional in \eqref{E:F} but rescalings of it depending on the parameter $\tau:=J_c-J$ with $J_c$ defined in \eqref{eq:jc} and $J<J_c$. For such rescalings the kernel of the nonlocal term has the form $K_\tau$ depending on $K_1$ as above. Rescaling the functional in \eqref{E:F} will have the advantage that the width and distance of the  periodic optimal stripes as well as their energy will be of order $O(1)$.

We assume that $K_1$ is such that the rescaled kernels $K_\tau$ satisfy the following 

\begin{align}
&\exists\,C:\quad\frac{1}{C}\frac{1}{(|\zeta|+\tau^{1/\beta})^p}\leq K_\tau(\zeta)\leq C\frac{1}{(|\zeta|+\tau^{1/\beta})^p},\quad p\geq d+2,\label{eq:Kbound}\\
&K_\tau(\zeta) \text{{ converges monotonically increasing for $\tau\downarrow 0$ either to }}\label{eq:Kmon} \\ & \hspace{4cm}\frac{1}{|\zeta|_1^p } \quad \text{ or to } \quad \frac{1}{|\zeta|^p }, \notag\\ 
&K_{\tau}\text{ is symmetric under exchange of coordinates, i.e. $K_\tau(P\zeta)=K_\tau(\zeta)$ }\notag\\
&\text{for all permutations $P$ on $d$ indices},\label{eq:Ksymm}\\
&\widehat K_\tau \text{ is the Laplace transform of a nonnegative function}\label{eq:laplpos},
\end{align}
where 
\[
\widehat K_\tau(\zeta_i)=\int_{\R^{d-1}}K_\tau(\zeta_i^\perp,\zeta_i)\d\zeta_i^\perp,
\]
which is independent of $i$ thanks to \eqref{eq:Ksymm}.

An example of such family of kernels is given by the rescaling of   $K_1(\zeta)=\frac{1}{(|\zeta|_1+1)^p}$, namely
\begin{equation}\label{eq:k11}
K_\tau(\zeta)=\frac{1}{(|\zeta|_1+\tau^{1\slash(p-d-1)})^p}.
\end{equation}

Indeed, the first two properties are trivial, and there exists a constant $C_{q}$ such that
\begin{equation*}
\widehat K_\tau(z)=C_q\frac{1}{(|z|+\tau^{1/\beta})^q},\quad q=p-d+1.
\end{equation*}

$\widehat K_\tau$ is the Laplace transform of a nonnegative function since, for $s>0$,

\[
\frac{1}{s^q}=\frac{1}{\Gamma(q)}\int_0^{+\infty} \alpha^{q-1}e^{-\alpha s}\d\alpha
\]
where $\Gamma(q)$ is the Euler's Gamma function.  Thus
\[
\frac{1}{(s+\tau^{1/\beta})^q}=\frac{1}{\Gamma(q)}\int_0^{+\infty} \alpha^{q-1}e^{-\alpha\tau^{1/\beta}}e^{-\alpha s}\d\alpha.
\]

Property \eqref{eq:laplpos} will be used in the Section~\ref{sec:1D_problem} and in particular in all one-dimensional optimizations.
Property \eqref{eq:Kbound} will be the main source of inequalities in Section~\ref{sec:structure_of_minimizers} and Section~\ref{sec:taul}. 
Property \eqref{eq:Kmon} is used in order to obtain the $\Gamma$-limit. 

From now on, we fix a kernel satisfying properties \eqref{eq:Kbound}-\eqref{eq:laplpos}.

\begin{remark}
   \label{rmk:theObvious}
   Given that all our analysis depends only on properties \eqref{eq:Kbound}-\eqref{eq:laplpos}, if one is not interested in  relating the structure of minimizers of $\Fcal_{\tau,L}$ with the structure of minimizers of $\tilde{\Fcal}_{J,L}$, then it is not necessary to assume that $K_\tau$ is obtained via rescaling $K_1$ as in \eqref{eq:Krescale}.
\end{remark}

Sets of locally finite perimeter are defined up to Lebesgue null sets, therefore in statements regarding a set of locally finite perimeter $E$ we will always mean they hold up to null sets or for a precise representative.  Given that both the perimeter and the nonlocal quantities defining \eqref{E:F} are invariant under modifications on   null sets, we can always assume to have a precise representative.
For example, the main theorems asserting that in some regimes minimizers of \eqref{E:F} are unions of periodic stripes hold neglecting a Lebesgue-null set.

Slicing (used in different contexts e.g. in  \cite{AFPBV,AFR,MR2096672,BianchiniDaneri,MR2883679,MR2805441,CarDan}) will be the main tool of our analysis. 

For $i\in\{1,\dots,d\}$, let $x_i^\perp$ be a point  in the subspace orthogonal to $e_i$.  We define the one-dimensional slices of $E\subset\R^d$ by

\[
E_{x_i^\perp}:=\bigl\{t\in[0,L):\,te_i+ x_i^\perp\in E\bigr\}.
\]
Notice that in the above definition there is an abuse of notation as the information on the direction of the slice is contained in the index $x^\perp_i$. 
As it would be always clear from the context which is the direction of the slicing, we hope this will not cause confusion to the reader.

Given a set of locally finite perimeter $E$, for a.e. $x_i^\perp$ its slice $E_{x_i^\perp}$ is a set of locally finite perimeter in $\R$ and the following slicing formula (see~\cite{Maggi}) holds for every $i\in\{1,\dots,d\}$
\[
\per_{1i}(E,[0,L)^d)=\int_{\partial E\cap [0,L)^d}|\nu^E_i(x)|\d\mathcal H^{d-1}(x)=\int_{[0,L)^{d-1}}\per_1(E_{x_i^\perp},[0,L))\dx_i^\perp.
\]

Whenever $d=1$, a set $E$ of  locally finite perimeter is up to Lebesgue-null sets a locally finite union of intervals (see~\cite{AFPBV,Maggi}). Therefore, w.l.o.g.  we will write $E=\underset{i\in\Z}{\cup} (s_i,t_i)$ with $t_i<s_{i+1}$. Moreover, one has that the reduced boundary $\partial E$ coincides with the topological boundary of $\underset{i\in\Z}{\cup} (s_i,t_i)$.

Thus, when $d=1$ one can define 
\[
\per_1(E, [0,L))  =\per(E,[0,L))= \#(\partial E \cap [0,L) ),
\]
where $\partial E$ is the reduced boundary of $E$.

While writing slicing formulas, with a slight abuse of notation we will sometimes identify $x_i\in[0,L)^d$ with its coordinate in $\R$ w.r.t. $e_i$ and $\{x_i^\perp:\,x\in[0,L)^d\}$ with $[0,L)^{d-1}\subset\R^{d-1}$. 

In Sections \ref{sec:structure_of_minimizers} and \ref{sec:taul} we will have to apply slicing on smaller cubes around a point. Therefore we need to introduce the following notation.
For $r> 0$ and $x^{\perp}_i$ we let $Q_{r}^{\perp}(x^\perp_{i}) = \{z^\perp_{i}:\, |x^{\perp}_{i} - z^{\perp}_{i} |_\infty \leq r  \}$ or we think of $x_i^\perp\in[0,L)^{d-1}$ and $Q_r^\perp(x_i^\perp)$ as a subset of $\R^{d-1}$. 
Since the subscript $i$ will be always present in the centre (namely $x_i^\perp$) of such $(d-1)$-dimensional cube, the implicit dependence on $i$ of $Q_r^\perp(x_i^\perp)$ should be clear. 
We denote also by $Q^i_r(t_i)\subset\R$ the interval of length $r$ centred in $t_i$.

 In Section \ref{sec:taul}, instead of integrals on $[0,L)^d$ one will often consider integrals on smaller cubes centred at other points of $[0,L)^d$. 
Therefore, for $z\in[0,L)^d$ and $r>0$, we define $Q_r(z)=\{x\in\R^d:\,|x-z|_\infty\leq r\}$.

\vskip 3mm

While doing estimates on slices, we will consider  $E\subset\R$  a set of locally finite perimeter and $s\in\partial E$ a point in the relative boundary of $E$. We will denote by 
\begin{equation}
\label{eq:s+s-}
\begin{split}
s^+ &:= \inf\{ t' \in \partial E, \text{with } t' > s  \} \\ s^- &:= \sup\{ t' \in \partial E, \text{with } t' < s  \}. 
\end{split}
\end{equation}

Estimates will be obtained (see the following subsection) through the function $\eta:\partial E\times \R\to\R$ defined as

\begin{equation}\label{eq:etafunct}
\eta(s,z):=\min(z_+,s-s^-)+\min(z_-,s^+-s),
\end{equation}
where  $z_+=\max\{z,0\}$ and $z_-=-\min\{z,0\}$.
In particular, given a $[0,L)^d$-periodic set $E$ of locally finite perimeter, the  functions $\eta_{x_i^\perp}:\partial E_{x_i^\perp}\times\R\to\R$ are defined as above for the slices $E_{x_i^\perp}$.

\vskip 3mm

In the paper, we will denote constants which depend on $L>0$ and on the dimension $d$ with the symbol $C_{d,L}$ and the constants which depend only on the dimension with $C_d$. In Section \ref{sec:taul}, where the constants do not depend on $L$, in order to simplify notation we will use $A\lesssim B$, whenever there exists a constant ${C}_{d}$ depending only on the dimension $d$ such that $A\leq {C}_d B$.  Notice that, since a kernel has been fixed, then the constants depend also implicitly on the chosen kernel. 

In presence of multiple integrals, we use the convention
\[
\int_{A_1}\dots\int_{A_n} f(x_1,\dots x_n)\dx_n\dots \dx_1=\int_{A_1}\Big(\dots\Big(\int_{A_n}f(x_1,\dots x_n)\dx_n\Big)\dots\Big)\dx_1.
\]

\subsection{Preliminary results}
\label{subsec:preliminary}

Let $E=E_h$ be a periodic union of stripes of width and distance $h$ in the direction $e_i$. Up to relabeling coordinates we can assume that $i=1$, thus $E_h=\hat{E_h}\times[0,L)^{d-1}$, and
 \[
 \Fj(E_h)=-\frac{\tau}{h}+\int_{\R^d}K_1(\zeta)\Big(\frac{|\zeta_1|}{h}-\frac{1}{L^d}\int_{[0,L)^d}|\chi_{E_h}(x)-\chi_{E_h}(x+\zeta_1)|\dx\Big)\d\zeta.
 \] 
 
    As in~\cite{GR}, it is possible to compute the energy $\Fj(E_h)$ to get 
    \[
       \Fj(E_h)\simeq -\frac{\tau}{h}+ h^{-(p-d)}.
    \]
Optimizing in $h$, one finds that the optimal stripes have a width of order $\tau^{-1/(p-d-1)}$ and energy of order $-\tau^{(p-d)/(p-d-1)}$. Letting $\beta:= p-d-1$, this motivates the rescaling
\begin{equation}\label{eq:changevar}
   x:=\tau^{-1/\beta}\tilde{x}, \quad L:=\tau^{-1/\beta}\tilde{L} \quad \textrm{and} \quad \Fj(E):= \tau^{(p-d)/\beta}  \mathcal{F}_{\tau,\tilde L}( \tilde{E}).
\end{equation}

In these variables, the optimal stripes have width of order $O(1)$. 

Making the substitutions in \eqref{eq:changevar} letting also $\zeta=\tau^{-1\slash\beta}\tilde{\zeta}$ and in the end dropping the tildes, one has (see Lemma 3.6 in \cite{GR})

\begin{equation}\label{eq:ftauel}
\begin{split}
\mathcal F_{\tau,L}(E)=\frac{1}{L^d}\Big(-\per_1(E,[0,L)^d)&+\int_{\R^d} K_\tau(\zeta) \Big[\int_{\partial E \cap [0,L)^d} \sum_{i=1}^d|\nu^E_i(x)| |\zeta_i|\d\mathcal H^{d-1}(x)\\
&-\int_{[0,L)^d}|\chi_E(x)-\chi_E(x+\zeta)|\dx\Big]\d\zeta\Big),
\end{split}
\end{equation}
where   $K_\tau(\zeta)=\tau^{-p/\beta}K_1(\zeta\tau^{-1/\beta})$.

   Let us now state an important estimate from below for $\mathcal F_{\tau,L}$ (see~\cite[Lemma 3.2]{GR}).

\begin{align}
|\chi_E(x)-\chi_E(x+\zeta)|=|\chi_E(x)-\chi_E(x+&\zeta_i)|+|\chi_E(x+\zeta_i)-\chi_E(x+\zeta)|\notag\\
&-2|\chi_E(x)-\chi_E(x+\zeta_i)||\chi_E(x+\zeta_i)-\chi_E(x+\zeta)|.\label{E:prod}
\end{align}
 
If the kernel $K_{\tau}$ is symmetric (namely, $K_{\tau}(\zeta_1,\ldots,\zeta_i,\ldots,\zeta_d) = K_{\tau}(\zeta_1,\ldots,-\zeta_i,\ldots,\zeta_d)$ for every $i=1,\ldots,d$),  one has that 
\begin{equation*}
   \begin{split} 
      \int_{[0,L)^d} \int_{\R^d} |\chi_{E}(x+ \zeta) - \chi_{E}(x+\zeta_1)| |\chi_{E}(x+ \zeta_1) - \chi_{E}(x)| K_{\tau}(\zeta) \d\zeta\dx  \\ = 
      \int_{[0,L)^d} \int_{\R^d} |\chi_{E}(x+ \zeta^\perp_1) - \chi_{E}(x)| |\chi_{E}(x+ \zeta_1) - \chi_{E}(x)| K_{\tau}(\zeta) \d\zeta\dx.
   \end{split}
\end{equation*}
For the general case, by using property \eqref{eq:Kbound}, we have that 
\begin{equation}
   \label{eq:simmetria_non_serve}
   \begin{split} 
      \int_{[0,L)^d} \int_{\R^d} &|\chi_{E}(x+ \zeta) - \chi_{E}(x+\zeta_1)| |\chi_{E}(x+ \zeta_1) - \chi_{E}(x)| K_{\tau}(\zeta) \d\zeta\dx \\ 
      & \leq C \int_{[0,L)^d} \int_{\R^d} \frac {|\chi_{E}(x+ \zeta) - \chi_{E}(x+\zeta_1)| |\chi_{E}(x+ \zeta_1) - \chi_{E}(x)|}{ (|\zeta |  + \tau^{1\slash \beta})^p } \d\zeta\dx \\ 
      & \leq C \int_{[0,L)^d} \int_{\R^d} \frac {|\chi_{E}(x+ \zeta_1) - \chi_{E}(x)| |\chi_{E}(x+ \zeta^\perp_1) - \chi_{E}(x)|}{ (|\zeta |  + \tau^{1\slash \beta})^p } \d\zeta\dx \\ 
      & \leq C^2 \int_{[0,L)^d} \int_{\R^d} |\chi_{E}(x+ \zeta^\perp_1) - \chi_{E}(x)| |\chi_{E}(x+ \zeta_1) - \chi_{E}(x)| K_{\tau}(\zeta)\d\zeta\dx,
   \end{split}
\end{equation}
   where $C$ is the  constant appearing in \eqref{eq:Kbound}.

   In the same way as in \cite[Lemma~3.2]{GR}, by  using \eqref{E:prod} and in addition   \eqref{eq:simmetria_non_serve}, one has that
   \begin{align}
      \int_{[0,L)^d}\int_{\R^d} K_\tau(\zeta) |\chi_E(x)&-\chi_E(x+\zeta)|\d\zeta\dx\leq \int_{[0,L)^d}\int_{\R^d} K_\tau(\zeta)\sum_{i=1}^d |\chi_E(x)-\chi_E(x+\zeta_i)|\d\zeta\dx\notag\\
      &- \frac{2C^2}{d}\int_{[0,L)^d}\int_{\R^d} K_\tau(\zeta) \sum_{i=1}^d|\chi_E(x)-\chi_{E}(x+\zeta_i)||\chi_E(x)-\chi_E(x+\zeta_i^\perp)|\d\zeta\dx,\label{E:3.2}.
   \end{align}

   As it will be clear from the proof, the result does not depend on the particular value of $C$, without loss of generality we may assume that $C=1$. 

   Notice also that \eqref{eq:simmetria_non_serve} is an equality if and only if  $\chi_E$ represents  unions of stripes.

Define then, for $i\in \{1,\ldots,d\}$,
\begin{equation*}
   \Gcal_{\tau,L}^i(E):=\int_{\R} \widehat{K}_\tau(\zeta_i)\Big[\int_{\partial E \cap [0,L)^d} |\nu^E_i(x)| |\zeta_i|\d\mathcal H^{d-1}(x)-\int_{[0,L)^d} |\chi_E(x)-\chi_E(x+\zeta_i)|\dx\Big]\d\zeta_i,
\end{equation*}
where $\widehat{K}_\tau(z)=\int_{\R^{d-1}}K_\tau(z,\zeta')\d\zeta' $ and
\begin{equation}\label{eq:I}
   \begin{split}
      I_{\tau,L}^i(E)&:=\frac{2}{d} \int_{[0,L)^d} \int_{\R^d} K_\tau(\zeta) |\chi_E(x)-\chi_{E}(x+\zeta_i)||\chi_E(x)-\chi_E(x+\zeta_i^\perp)|\d\zeta\dx,\\
      I_{\tau,L}(E) &:= \sum_{i=1}^{d} I_{\tau,L}^i(E).
   \end{split}
\end{equation}

Estimate \eqref{E:3.2} implies (see Lemma 3.6 in \cite{GR})
\begin{equation}
   \label{E:fbelow}
      \Fcal_{\tau,L}(E)\ge \frac{1}{L^d}\Big( -\per_1(E,[0,L)^d)+\sum_{i=1}^d \Gcal_{\tau,L}^i(E)+\sum_{i=1}^d I_{\tau,L}^i(E)\Big).
\end{equation}

The estimate we need now is the following (see Lemma 3.4 of~\cite{GR}): for every one-dimensional $L$-periodic set $E\subset \R$ of locally finite perimeter and every $z\in\R$, 

\begin{equation}\label{toprovebandesim}
      \int_{0}^L |\chi_{E}(x)-\chi_{E}(x+z)| \dx\leq \sum_{x\in \partial E\cap [0,L)} \eta(x,z),
   \end{equation}
   where $\eta$ is the function defined in \eqref{eq:etafunct}
   
   

      For every  $\tau\ge 0$, $\zeta\in\R^{d}$,  $[0,L)^d$-periodic set $E \subset \R^d$ of locally finite perimeter and $i\in \{1,\ldots,d\}$, by using \eqref{toprovebandesim} with a slicing argument (see Lemma 3.4 in \cite{GR}) one has that 
      \begin{equation}\label{eq:bandestim0}
         \int_{[0,L)^d} \int_{\R^d} K_\tau(\zeta) |\chi_{E}(x)-\chi_{E}(x+\zeta_i)|\d\zeta\dx\le \int_{\partial E\cap [0,L)^d} |\nu^E_i(x)| \int_{\R^d} K_\tau(\zeta) \eta_{x_i^\perp}(x_i,\zeta_i)\d\zeta\d\mathcal H^{d-1}(x).
      \end{equation}

   We will use the following slicing formula
 \begin{equation}
    \label{eq:gstrf1}
    \begin{split}
 \Gcal_{\tau,L}^i(E)=\int_{[0,L)^{d-1}} \Gcal_{\tau,L}^{1d}(E_{x_i^\perp}) \dx_i^\perp,
    \end{split}
 \end{equation}
 where 
 \begin{equation}
    \label{eq:gstrf2}
    \begin{split}
    \Gcal_{\tau,L}^{1d}(E_{x_i^\perp}):=\int_{\R} \widehat{K}_\tau(z) \Big(\per(E_{x_i^\perp},[0,L))|z|-\int_0^L |\chi_{E_{x_i^\perp}}(x)-\chi_{E_{x_i^\perp}}(x+z)|\dx\Big)\dz.
    \end{split}
 \end{equation}
As a consequence of \eqref{eq:bandestim0} and the fact that $|z|\geq\eta_{x_i^\perp}(x,z)$, $\Gcal_{\tau,L}^i(E)\geq0$.

More precisely, one has the following estimate from below for $\Gcal_{\tau,L}^{1d}$ (see~\cite[Lemma 3.7]{GR}): for every one-dimensional $L$-periodic set $E$ of locally finite perimeter  (recall that $\beta=p-d-1$),

\begin{equation}\label{E:boundg}
 \Gcal_{\tau,L}^{1d}(E) \geq C_{d,L}   \sum_{x\in \partial E\cap [0,L)}  \min((x^+-x)^{-\beta},\tau^{-1})+\min((x-x^-)^{-\beta},\tau^{-1}),
\end{equation}
where $x^+$ and $x^-$ are defined as in \eqref{eq:s+s-}.

   Moreover, for every $\delta\geq\tau^{1\slash\beta}$, one has that (see~\cite[Lemma~3.7]{GR})
   \begin{equation}\label{E:perg}
      \per(E,[0,L))-1\leq C_{d,L} ( L\delta^{-1}+ \delta^{\beta}\Gcal_{\tau,L}^{1d}(E)).
   \end{equation}
   
Although in \cite{GR} the estimates above hold for a slightly different kernel (namely, $\frac{1}{|\zeta|^p+1}$), they continue to hold for the type of kernels considered here since the only thing which is used in the proof is assumption \eqref{eq:Kbound}.

Optimizing in $\delta$ in \eqref{E:perg}, one has that 
\[
\per(E,[0,L))-1\leq C_{d,L}\max\bigl(\tau\Gcal_{\tau,L}^{1d}(E),\Gcal_{\tau,L}^{1d}(E)^{1\slash{p-d}}\bigr).
\]
Integrating it for $E_{x_i^\perp}$ w.r.t. $x_i^\perp\in[0,L)^{d-1}$ one obtains
\begin{equation}\label{E:estbelow}
\Fcal_{\tau,L}(E)\geq C_{d,L}\Big[-1+ \frac{1}{L^d}\Big(\sum_{i=1}^d\Gcal_{\tau,L}^i(E)+\sum_{i=1}^d I_{\tau,L}^i(E)\Big)\Big]
\end{equation}
and 
\begin{equation}
   \label{eq:gr1}
   \per_1(E,[0,L)^d)\leq C_{d,L} L^d \max(1,\Fcal_{\tau,L}(E)).
\end{equation}
For details see Lemma 3.9 in \cite{GR}.

The main result of \cite{GR} is the following.

\begin{theorem}[{\cite[Theorem~1.1]{GR}}]\label{thm:gammaconv}
   Let $p>2d$ and $L>0$. Then one has that $\Fcal_{\tau,L}$ $\Gamma$-converge in the $L^1$-topology as $\tau\to0$  to a functional $\Fcal_{0,L}$ which is invariant under permutation of coordinates and finite on sets which are (up to permutation of coordinates) of the form $E=F\times\R^{d-1}$ where $F\subset \R$ is $L$-periodic with $\#\{\partial F\cap[0,L)\}<\infty$.

      Moreover, let $\{ E_{\tau}\}$ be a family of $[0,L)^d$-periodic subsets of $\R^d$ such that there exists $M$ such that for every $\tau$ one has that $\mathcal F_{\tau,L}(E_\tau)<M$. Then, up to a permutation of coordinates, one has that there is a subsequence which converges in $L^1$  to some set of the form $E=F\times\R^{d-1}$ with $\#\{\partial F\cap[0,L)\}<\infty$.
\end{theorem}


\section{From $p>2d$ to $p\geq d+2$}
\label{sec:d+2}

In this section we will prove the following theorem:

\begin{theorem}[{\cite[Theorem~1.1]{GR}} improved]\label{T:gammaconvNew}
   Let $p\geq d+2$ and $L>0$. Then one has that $\Fcal_{\tau,L}$ $\Gamma$-converge in the $L^1$-topology as $\tau\to0$  to a functional $\Fcal_{0,L}$ which is invariant under permutation of coordinates and finite on sets (up to permutation of coordinates) of the form $E=F\times\R^{d-1}$, where $F\subset \R$ is $L$-periodic with $\#\{\partial F\cap[0,L)\}<\infty$. 

   On sets of the form $E=F\times\R^{d-1}$ the functional is defined by
\begin{equation}
   \label{def:F0}
\mathcal F_{0,L}(E)=\frac{1}{L}\Big(-\#\{\partial F\cap[0,L)\}+\int_{\R^d}\frac{1}{|\zeta|^{p}}\Big[\sum_{x\in\partial F\cap[0,L)}|\zeta_1|-\int_0^L|\chi_{F}(x)- \chi_{F}(x+\zeta_1)|\dx\Big]\d\zeta\Big).
\end{equation}

Moreover, let $\{ E_{\tau}\}$ be a family of $[0,L)^d$-periodic subsets of $\R^d$ such that there exists $M$ such that for every $\tau$ one has that $\mathcal F_{\tau,L}(E_\tau)<M$,
 then, up to a permutation of coordinates, one has that there is a subsequence which converges in $L^1([0,L)^d)$  to some set $E=F\times\R^{d-1}$ with $\#\{\partial F\cap[0,L)\}<\infty$.
\end{theorem}

The proof of Theorem~\ref{T:gammaconvNew} consists of two main parts: a part in which compactness of sets of equibounded energy  and lower semicontinuity of the functionals $\Gcal^i_{\tau,L}$ and $I^{i}_{\tau,L}$ is proved and a part in which a rigidity estimate is proved.  The rigidity estimate roughly says that in the limit as $\tau\downarrow 0$, sets of equibounded energy must converge to stripes in $L^1([0,L)^d)$.  This in turn says that the limiting problem is one-dimensional. 

For the first part we refer to \cite{GR}, where the quantities and the estimates defined in Section~\ref{sec:setting_and_preliminary_results} are used. 

The second part  is based on different arguments. 
Indeed, the rigidity result  is the core of Theorem~\ref{T:gammaconvNew}. 
Such result is contained in the following proposition, which  substitutes Proposition~4.3 of \cite{GR}. 

\begin{proposition}[Rigidity]
   \label{prop:rigidity}
   Let $p\geq d+2$ and let $E$ be a $[0,L)^d$-periodic set of locally finite perimeter such that $\sum_{i=1}^d\Gcal_{0,L}^i(E)+I_{0,L}(E)<+\infty$. Then, $E$ is one-dimensional, i.e. up to permutation of the coordinates, $E=\widehat{E}\times \R^{d-1}$ for some $L-$periodic set $\widehat{E}$.
\end{proposition}


In the proof of Proposition~\ref{prop:rigidity}, we will apply the following result (see~\cite[Proposition~1]{Bre}). It will be applied to the function $r^{i}_{\lambda}(u,\cdot):\R^{d-1}\to\R$ (for the definition see~\eqref{eq:defri_local} below).

\begin{theorem}
   \label{thm:brezis}
   Let $\Omega\subset \R^{d-1}$ be an open and connected set in $\R^{d-1}$ and  $f:\R^{d-1}\to \R$ be a measurable function such that
   \begin{equation} 
      \label{eq:thm_brezis}
      \begin{split}
         \int_{\Omega\times \Omega}\frac{|f(x) - f(y) |}{|x-y|^{d}} < +\infty. 
      \end{split}
   \end{equation} 
   Then $f$ is constant almost everywhere, namely there is constant function $\tilde{f}$  such that $f = \tilde{f}$ up to a set of null Lebesgue measure. 
\end{theorem}

Notice that the Theorem~\ref{thm:brezis} is trivial whenever $f$ is smooth.  
In order to obtain the nonsmooth case, a regularization step is needed  (see~\cite{Bre} for the details).

Let us also recall our notation: given $t\in \R^d$, we will denote by $t_{i} = \scalare{t, e_i} e_{i} $, $t_{i}^\perp =  t - t_{i}$ and we set
   \begin{equation}\label{eq:fE}
      \begin{split}
         f_{E}(t^{\perp}_i,t_i,t'^\perp_{i},t'_i):=|\chi_{E}(t_{i}^\perp +t_i+ t'_{i}) - \chi_{E}(t_i + t^{\perp}_{i} ) | | \chi_{E}(t_{i}^\perp +t_i+ t'^{\perp}_{i}) - \chi_{E}(t_i + t^{\perp}_{i} )  |.
      \end{split}
   \end{equation}

   In order to be able to use Theorem~\ref{thm:brezis}, we need to change variables and have $K_{0}(t' -t)$ instead of $K_0(t')$ in our formulas. For this reason we will make the change of variables $\tilde{t} = t' + t$. Thus we have that
   \begin{equation*}
      \begin{split}
         \frac{d}{2}I^{i}_{0,L} (E)  & = \int_{[0,L)^{d}} \int_{\R^d} |\chi_{E}(t+ t'^{\perp}_{i}) - \chi_{E}(t)| |\chi_{E}(t+ t'_{i}) - \chi_{E}(t)| K_{0}(t')\dt'\dt  \\ &= 
          \int_{[0,L)^{d}} \int_{\R^d} |\chi_{E}( t_i+ \tilde{t}^{\perp}_{i}) - \chi_{E}(t)| |\chi_{E}(t^{\perp}_i+ \tilde{t}_{i}) - \chi_{E}(t)| K_{0}( \tilde{t}-t)\d\tilde{t} \dt  
          \\ &= \int_{[0,L)^{d}} \int_{\R^d} |\chi_{E}(t_i+ {t'}^{\perp}_{i}) - \chi_{E}(t)| |\chi_{E}(t^{\perp}_i+ {t'}_{i}) - \chi_{E}(t)| K_{0}( {t'}-t)\dt' \dt
          \\ &\geq  \int_{[0,L)^{d}} \int_{[0,L)^{d}} |\chi_{E}(t_i+ {t'}^{\perp}_{i}) - \chi_{E}(t)| |\chi_{E}(t^{\perp}_i+ {t'}_{i}) - \chi_{E}(t)| K_{0}(t' - t) \dt' \dt
      \end{split}
   \end{equation*}
   
   Since we will make a slicing argument, we further rewrite the above as 

\begin{equation}
   \label{eq:gstr11}
   \begin{split}
      \frac{d}{2}I^{i}_{0,L}(E) \geq   \int_{[0,L)^{d-1}}\int_{[0,L)^{d-1}} \Int(t^{\perp}_{i},t'^{\perp}_{i}) \dt^{\perp}_{i} \dt'^{\perp}_{i},
   \end{split}
\end{equation}
where 
\begin{equation}
   \label{eq:definizioneInteraction}
   \begin{split}
      \Int(t^{\perp}_{i},t'^{\perp}_{i})   := 
      \int_{0}^L \int_{0}^L f_{E}(t^{\perp}_i,t_i,t'^\perp_{i}-t^{\perp}_i, t'_i-t_i) K_{0}(t- t')\dt_{i} \dt'_{i} .
   \end{split}
\end{equation}

It is useful to think of $\Int(t^{\perp}_i,t'^{\perp}_i)$ as the interaction between two different slices.

In this section we will use the following notation: given $\lambda\in( 0,\frac L2)$, $u\in (\lambda,L - \lambda)$ and $E$ a set of locally finite perimeter and $t^\perp_{i}\in[0,L)^{d-1}$,  we will denote by  
\begin{equation}
   \label{eq:defri_local}
   \begin{split}
      r^i_{\lambda}(u, t_{i}^\perp) &:= \min \big\{ \inf \{|u-s |:\ s\in \partial E_{t^\perp_i} \text{ and } s\in (\lambda,L-\lambda)  \}, |u-\lambda |, |L-\lambda -u | \big\} \\
    r^i_o(t_{i}^\perp) &:= \inf_{s\in \partial E_{t^\perp_i}\cap [0,L]} \min(s^+-s,s-s^-),
   \end{split}
\end{equation}
where $s^+,s^-$ are defined in \eqref{eq:s+s-}. 

   Notice that  for a set of finite perimeter the followings hold:
   \begin{itemize}
      \item $E_{t^\perp_i}$ is a set  of finite perimeter for almost every $t^\perp_i$,
      \item for every set of finite perimeter in $F\subset \R$, there exist a finite number of intervals $\{ J_i\}_{i=1}^N$ such that $F\cap [0,L)= \bigcup_{i=1}^N J_i$, where the equality is intended in the measure theoretic sense, namely up to a set of null measure (see~\cite{AFPBV}).
   \end{itemize}
   Thus the above map is well-defined for almost every $t^\perp_i$ and measurable.

      \begin{remark}
         \label{rmk:PeriodicitaNonServe}
         When $\lambda = 0$, term $r^i_{\lambda}(u,t^\perp_i)$ measures the distance of $u$ from  jump points on $\partial E_{t_i^\perp}$ which are close. 
         The role of $\lambda > 0$ is technical: it is used to handle the situation in which the next jump point $s\in \partial E_{t^\perp_i}$ is in a $\lambda$-neighborhood of $\{ 0, L\}$. In Proposition~\ref{prop:rigidity}, this technical point  would be unnecessary since $E$ is $[0,L)^d$-periodic. However, we introduce it because it will be needed in the proof of the local rigidity lemma, namely Lemma~\ref{lemma:local_rigidity} (when instead of $[0,L)^d$ we will consider $[0,l)^d$ with $l<L$ and therefore $E$ is not $[0,l)^d$-periodic). We prefer to give here already a more general proof instead of repeating twice a similar argument.
      \end{remark}
   
     Suppose that, for every $u$, one has that $r^{i}_{\lambda}(u,\cdot)$ is constant almost everywhere: if this holds for every $i$, then it is not difficult to see that  $E$  is (up to null sets) either  a union of stripes or a  checkerboards, where by checkerboards we mean any set whose boundary is the union of affine subspace orthogonal to  coordinate axes, and there are at least two of these directions. 
     
     Via an energetic argument one can rule out checkerboards
      (see the comment at the end of the proof of Proposition~\ref{prop:rigidity}).

In order to obtain that $r^i_\lambda(u,\cdot)$ is constant almost everywhere, we will apply Theorem~\ref{thm:brezis}.

The main use of the term $r^{i}_{o}(t_{i})$ is the following: since $\Gcal_{0,L}^{i}(E) < +\infty$ and inequality \eqref{E:boundg} holds for $\tau = 0$, one has that

\begin{equation}
   \label{eq:stimaro}
   \begin{split}
      \int_{[0,L)^{d-1}} (r^{i}_{0}(t_{i}^{\perp}))^{-\beta} \dt^{\perp}_i \leq \Gcal_{0,L}^{i}(E) < +\infty. 
   \end{split}
\end{equation}

The next lemma gives a lower bound for the interaction term of close slices.

   \begin{lemma}
      \label{lemma:p2d_stimaSlice}
      Let $\lambda \in (0,L/2)$ and let $t'^{\perp}_{i},t^{\perp}_{i}\in [0,L)^{d-1}$, $t_i^\perp \neq t'^\perp_i$ be such that  $\min(r^{i}_o(t^{\perp}_i), r^{i}_o(t'^{\perp}_i)) > |t'^\perp_{i} - t^{\perp}_{i}|$ and $|t'^\perp_i - t^\perp_i | \leq \lambda$.
      Then for every $u\in (\lambda, L -\lambda)$ it holds
      \begin{equation} 
         \label{eq:lemma_AB1}
         \Int(t'^{\perp}_{i},t^{\perp}_{i}) \geq C_{d,L} \frac{|r^{i}_\lambda(u,t'^{\perp}_{i}) - r^{i}_{\lambda}(u,t^{\perp}_{i}) |}{|t'^{\perp}_{i} - t^{\perp}_{i}|^{d}}.
      \end{equation}
   \end{lemma}

   \begin{proof}
      W.l.o.g. let us assume that $r^{i}_\lambda(u, t^\perp_{i}) < r^i_\lambda(u, t'^{\perp}_i)$.  In particular this implies that $r^i_\lambda(u,t^\perp_i) < \min(|u-\lambda|, | L - \lambda -u |)$, and hence there exists a point  $s_o \in (\lambda,L-\lambda)$  
      such that 
      \begin{equation*}
         \begin{split}
            |u - s_o| = \inf \{ |u-s|:\ s\in \partial E_{t^\perp_i}, s\in (\lambda, L-\lambda)  \}.
         \end{split}
      \end{equation*}
      
      For simplicity of notation denote $\delta = |t'^\perp_i -  t^\perp_i|$ and $r = |r^i_\lambda(u,t^\perp_i) - r^i_\lambda(u,t'^{\perp}_{i})|$.  
            Since $r_{o}(t^{\perp}_{i}) > \delta$, one has that
            \begin{equation*}
               \begin{split}
                  (s_{o} -\delta , s_{o} +\delta) \cap E_{t^{\perp}_i} = (s_{o} , s_{o} +\delta )  \qquad\text{or} \qquad 
                  (s_{o} -\delta , s_{o} +\delta) \cap E_{t^{\perp}_i} = (s_{o} -\delta, s_{o}  )  .
               \end{split}
            \end{equation*}
            Notice that since $\lambda \geq \delta$, we have that $(s_o - \delta, s_o + \delta ) \subset [0,L)$.
            We will assume
            \begin{equation}
               \label{eq:gstr33}
               \begin{split}
                  (s_{o} -\delta , s_{o} +\delta) \cap E_{t^{\perp}_i} = (s_{o} , s_{o} +\delta ) 
               \end{split}
            \end{equation}
            The other case is analogous.

      We will distinguish two subcases: 
      
      \begin{enumerate}[(i)]
         \item  Suppose $r  > \delta/2$. 
            From the definition of $\delta$ and $r$, then for the slice in $t'^\perp_i$ one has that
      \begin{equation*}
         \begin{split}
      (s_{o} - \delta/2 , s_{o} + \delta/2) \cap E_{t'^{\perp}_i} = (s_{o} - \delta/2 , s_{o} + \delta/2 ) 
      \qquad \text{or}\qquad
      (s_{o} - \delta/2 , s_{o} + \delta/2) \cap E_{t'^{\perp}_i} = \emptyset. 
         \end{split}
      \end{equation*}
      Indeed on the slice $E_{t'^\perp_i}$, the closest jump point to $s_o$  is at least $r$ distant and $r>\delta/2$. 
      We will assume the first alternative above. The other case is analogous.  

            Then for every $a \in (s_o-\delta/2, s_o)$ and $a' \in (s_o, s_o+\delta/2)$, one has that 
            \begin{equation*}
               \begin{split}
                  f_{E}(t^\perp_i, a, t'^\perp_i -  t^\perp_i, a' - a) = 1.
               \end{split}
            \end{equation*} 
            Thus given that $r^i_{\lambda}(u,t^\perp_i) \leq L $, we have that
            \begin{equation*}
               \begin{split}
                  \Int(t^{\perp}_{i},t'^{\perp}_{i})   & = 
      \int_{0}^L \int_{0}^L f_{E}(t^{\perp}_i,t_i,t'^\perp_{i}-t^{\perp}_i, t'_i-t_i) K_{0}(t'- t)\dt_{i} \dt'_{i} 
      \\ &\geq \int_{s_o-\delta /2}^{s_o} \int_{s_o}^{s_o + \delta /2} f_{E}(t^{\perp}_i,t_i,t'^\perp_{i}-t^{\perp}_i, t'_i-t_i) K_{0}(t'- t)\dt'_{i} \dt_{i} 
      \\ &\geq \int_{s_o-\delta /2}^{s_o} \int_{s_o}^{s_o + \delta /2}  K_{0}(t'- t)\dt'_{i} \dt_{i} \geq C_d\frac{\delta ^2}{\delta ^p} \geq C_d \frac{|r^i_\lambda(u,t^\perp_i) - r^i_\lambda(u,t'^\perp_i)|}{2L}\frac{1}{\delta ^{p-2}}
      \\ &  \geq C_{d,L} \frac{|r^i_\lambda(u,t^\perp_i) - r^i_\lambda(u,t'^\perp_i)|}{|t'^\perp_{i} - t^{\perp}_{i}|^{d}},
               \end{split}
            \end{equation*}
            where in the second last line above we have used that for every $t_i \in (s_o- \delta /2, s_o)$ and $t'_i \in (s_o, s_o+ \delta /2)$ one has that  
            \begin{equation}
               \label{eq:gstr31}
               \begin{split}
                  K_0(t' -t ) \geq \frac{C_d}{|t'^\perp_i - t^\perp_i|^{p}}.
               \end{split}
            \end{equation}

\begin{figure}
	\centering
	\def\svgwidth{14cm}
	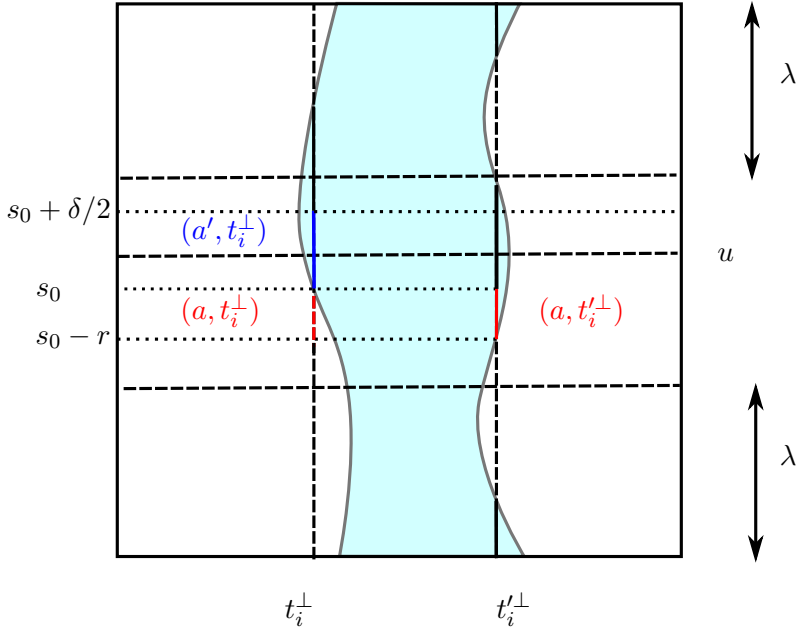
\caption{
{The points in the proof of Lemma \ref{lemma:p2d_stimaSlice} case (ii) are depicted. Recall that $\delta=|t^\perp_i-t'^\perp_i|\leq\lambda$, $r=|r_\lambda^i(u,t'^\perp_i)-r_\lambda^i(u,t^\perp_i)|$ is less than or equal to $\delta/2$. In the estimate for the interaction term $\Int(t'^{\perp}_{i},t^\perp_{i})$ one takes points \textcolor{red}{$(a,t_i^\perp)$}, \textcolor{red}{$(a,t'^\perp_i)$} with $a\in(s_o-r,s_o)$ and   \textcolor{blue}{$(a',t_i^\perp)$} with $a'\in(s_o,s_o+\frac\delta 2)$.}
}
\label{fig:1}
\end{figure}

         \item  Let us assume now that $r \leq  \delta/2$. 
            Since $r_{o}(t'^{\perp}_{i}),r_{o}(t^{\perp}_{i}) > \delta$, one has that
             either
            \begin{equation*}
               \begin{split}
                  &(s_{o} - r , s_{o} + \delta/2) \cap E_{t'^{\perp}_i} = (s_{o} - r , s_{o} + \delta/2 ) \quad\text{or}\quad (s_{o} - r , s_{o} + \delta/2) \cap E_{t'^{\perp}_i} =\emptyset\\
                  \text{or}\quad&(s_{o} - \delta/2 , s_{o} +r) \cap E_{t'^{\perp}_i} = (s_{o} - \delta/2 , s_{o} + r )\quad\text{or}\quad(s_{o} - \delta/2 , s_{o} +r) \cap E_{t'^{\perp}_i} =\emptyset.
               \end{split}
            \end{equation*}

            Indeed if none of the above were true we would have that $\#(\partial E_{t'^{\perp}_{i}} \cap (s_o -\delta/2 , s_o +\delta/2)) \geq 2 $, which contradicts $r_o(t'^{\perp}_i) > \delta$. 

            W.l.o.g. (see Figure \ref{fig:1}) we will assume 
            \begin{equation*}
               \begin{split}
                  (s_{o} - r , s_{o} + \delta/2) \cap E_{t'^{\perp}_i} = (s_{o} - r , s_{o} + \delta/2 ). 
               \end{split}
            \end{equation*}
            The other cases are analogous.

            Then for every $ a \in (s_o - r, s_o)$ and $a' \in (s_o, s_o+ \delta/2)$, one has that $f_{E}(t^\perp_i, a, t'^\perp_i -  t^\perp_i, a' - a) = 1$.  Thus 
            \begin{equation*}
               \begin{split}
                  \Int(t^{\perp}_{i},t'^{\perp}_{i})  & = 
      \int_{0}^L \int_{0}^L f_{E}(t^{\perp}_i,t_i,t'^\perp_{i}-t^{\perp}_i, t'_i-t_i) K_{0}(t'- t)\dt_{i} \dt'_{i} 
      \\ &\geq \int_{s_o - r}^{s_o} \int_{s_o}^{s_o + \delta/2} f_{E}(t^{\perp}_i,t_i,t'^\perp_{i}-t^{\perp}_i, t'_i-t_i) K_{0}(t'- t)\dt'_{i} \dt_{i} 
      \\ &\geq \int_{s_o - r}^{s_o} \int_{s_o}^{s_o + \delta/2}  K_{0}(t'- t)\dt'_{i} \dt_{i} \geq C_d \frac{\delta r}{\delta^p} \geq C_{d,L}\frac{|r^i_\lambda(u,t^\perp_i) - r^i_\lambda(u,t'^\perp_i)|}{|t'^\perp_{i} - t^{\perp}_{i}|^{d+1}},
               \end{split}
            \end{equation*}
      where in the last line we have used  \eqref{eq:gstr31}.
      \end{enumerate}

   \end{proof}

\begin{lemma}
   \label{lemma:ri_constantNew}
   Assume that $E\subset \R^d$ is a set of locally finite perimeter such that $(r_o^i)^{-\beta}\in L^1([0,L)^d)$, $p \geq d+2$ and $I_{0,L}(E)<+\infty$.  Let $r^i_\lambda(u,\cdot)$ be as defined in \eqref{eq:defri_local}.  Then, we have that $r^{i}_\lambda(u,\cdot)$ is constant almost everywhere. 
\end{lemma}

\begin{proof}
   We will apply Theorem~\ref{thm:brezis} to the function $r^i_\lambda(u,\cdot)$. 
   Notice that because $[0,L)^{d-1} \supset  (0,L)^{d-1}$, by showing \eqref{eq:thm_brezis} with $[0,L)^{d-1}$ instead of $(0,L)^{d-1}$ we have that $r^i_{\lambda}(u,\cdot)$ is constant  almost everywhere. 

   \begin{equation}
      \label{eq:gstr29}
      \begin{split}
         \int_{[0,L)^{d-1}} \int_{[0,L)^{d-1}} \frac{|r^i_\lambda(u, t^{\perp}_{i}) - r^i_\lambda(u, t'^{\perp}_{i})|}{|t'^{\perp}_{_{i}} - t^{\perp}_{_{i}}|^d} \dt'^{\perp}_{i} \dt^{\perp}_{i} \leq C_{d,L} ( \mathcal G^i_{0,L}(E) + I^i_{0,L} (E)) < + \infty. 
      \end{split}
   \end{equation}
   and since $I^i_{0,L}(E)$ is finite we have that $r^i_\lambda(u,\cdot)$ is constant almost everywhere. 

   We have that 
   \begin{equation}
      \label{eq:split1}
      \begin{split}
         \int_{[0,L)^{d-1}} \int_{[0,L)^{d-1}} &\frac{|r^i_\lambda(u, t^{\perp}_{i}) - r^i_\lambda(u, t'^{\perp}_{i})|}{|t'^{\perp}_{_{i}} - t^{\perp}_{_{i}}|^d} \dt^{\perp}_{i}\dt'^{\perp}_{i}            
         \\ \leq&  \int_{[0,L)^{d-1}} \int_{[0,L)^{d-1}} 
         \chi_{A}(t^\perp_{i},t'^\perp_{i})
         \frac{|r^i_\lambda(u, t'^\perp_{i}) - r^i_\lambda(u, t^\perp_{i})|}{|t^\perp_{i}  - t'^\perp_{i} |^d}\dt'^{\perp}_{i}\dt^{\perp}_{i}
         \\ &+  \int_{[0,L)^{d-1}} \int_{[0,L)^{d-1}} 
         \chi_{B}(t^\perp_{i},t'^\perp_{i})
         \frac{|r^i_\lambda(u, t'^\perp_{i}) - r^i_\lambda(u, t^\perp_{i})|}{|t^\perp_{i}  - t'^\perp_{i} |^d}\dt'^{\perp}_{i}\dt^{\perp}_{i}
        \\& +   \int_{[0,L)^{d-1}} \int_{[0,L)^{d-1}} 
         \chi_{C}(t^\perp_{i},t'^\perp_{i})
         \frac{|r^i_\lambda(u, t'^\perp_{i}) - r^i_\lambda(u, t^\perp_{i})|}{|t^\perp_{i}  - t'^\perp_{i} |^d}\dt'^{\perp}_{i}\dt^{\perp}_{i}
      \end{split}
   \end{equation}

   where 
   \begin{equation*}
      \begin{split}
         A&: = \insieme{ (t'^\perp_{i},t^\perp_{i})\in [0,L)^{d-1}\times[0,L)^{d-1} :  \min(r_{o}( t'^\perp_{i}) , r_o(t^\perp_{i})) \leq  |t^{\perp}_{i} - t'^\perp_{i} |, \text{ and } |t'^\perp_i -t^\perp_i| \leq \lambda}\\
         B&: = \insieme{ (t'^\perp_{i},t^\perp_{i})\in [0,L)^{d-1}\times[0,L)^{d-1} :  \min(r_{o}( t'^\perp_{i}) , r_o(t^\perp_{i})) > |t^{\perp}_{i} - t'^\perp_{i} |,\text{ and } |t'^\perp_i -t^\perp_i| \leq  \lambda}\\
         C&: = \insieme{ (t'^\perp_{i},t^\perp_{i})\in [0,L)^{d-1}\times[0,L)^{d-1} : |t'^\perp_i -t^\perp_i | > \lambda}.
      \end{split}
   \end{equation*}
   The last term on the \rhs of \eqref{eq:split1} is trivially bounded. 
  
   Because of \eqref{eq:lemma_AB1}, we have that
   \begin{equation}
   \label{eq:stima1I}
      \begin{split}
          \int_{[0,L)^{d-1}} \int_{[0,L)^{d-1}} 
         \chi_{B}(t^\perp_{i},t'^\perp_{i})
         \frac{|r^i_\lambda(u, t'^\perp_{i}) - r^i_\lambda(u, t^\perp_{i})|}{|t^\perp_{i}  - t'^\perp_{i} |^d}\dt'^{\perp}_{i}\dt^{\perp}_{i}& \leq C_{d,L}
          \int_{[0,L)^{d-1}} \int_{[0,L)^{d-1}}  \Int(t'^{\perp}_{i},t^\perp_{i})\\
          & \leq C_{d,L}\frac{d}{2}I^i_{0,L}(E)
      \end{split}
   \end{equation}

   Since
   \begin{equation*}
      \begin{split}
         A= \{(t^\perp_i,t'^\perp_i): r_{o}(t^\perp_i) \leq |t^{\perp}_{i} - t'^\perp_{i} |\} \cup \{(t^\perp_i,t'^\perp_i): r_{o}(t'^\perp_i) \leq |t^{\perp}_{i} - t'^\perp_{i} |\},
      \end{split}
   \end{equation*}
   it is sufficient to estimate the contribution of $\{ (t^\perp_i, t'^\perp_i):\ r_{o}(t^\perp_i) \leq |t'^\perp_i - t^\perp_i |\}$. 
   Thus, one has that
   \begin{equation}
   \label{eq:stima2G}
      \begin{split}
          \iint_{[0,L)^{d-1}\times [0,L)^{d-1}}& 
          \chi_{A}(t'^{\perp}_{i},t^{\perp}_{i}) \frac{|r^{i}_\lambda(u,t^{\perp}_{i}) - r^{i}_\lambda(u,t'^{\perp}_{i})|}{|t^{\perp}_{i} - t'^\perp_{i} |^{d}} \dt'^{\perp}_{i}\dt^{\perp}_{i} \\ 
          &\leq2 
          \int_{[0,L)^{d-1}}
          \int_{\{ r_o(t^\perp_i) < |t'^\perp_i - t^\perp_i |\}} 
           \frac{|r^{i}_\lambda(u,t^{\perp}_{i}) - r^{i}_\lambda(u,t'^{\perp}_{i})|}{|t^{\perp}_{i} - t'^\perp_{i} |^{d}} \dt'^{\perp}_{i}\dt^{\perp}_{i} \\ 
          &\leq2 
          \int_{[0,L)^{d-1}}
          \int_{\{ r_o(t^\perp_i) < |t'^\perp_i - t^\perp_i |\}} 
           \frac{2L}{|t^{\perp}_{i} - t'^\perp_{i} |^{d}} \dt'^{\perp}_{i}\dt^{\perp}_{i} \\ 
         & \leq C_d
          \int_{[0,L)^{d-1}}
          \int_{r_o(t^\perp_i) }^L 
           \frac{2L \rho^{d-2}}{\rho^d} \d\rho \dt^\perp_i \\
           & \leq C_{d,L} \int_{[0,L)^{d-1}} r_o(t^\perp_{i})^{-1} \dt^{\perp}_{i}\\
           &\leq C_{d,L} \int_{[0,L)^{d-1}} r_o(t^\perp_{i})^{-\beta} \dt^{\perp}_{i}\\
           &\leq C_{d,L}\mathcal{G}_{0,L}^i(E)
      \end{split}
   \end{equation}
   where from the third to the fourth line we have used a change to polar coordinates, from the fifth to the sixth line the fact that $\beta = p-d -1 \geq 1$ and in the last line estimate  \eqref{eq:stimaro}.

   By combining estimates \eqref{eq:stima1I} and \eqref{eq:stima2G} with the boundedness of the last term in \eqref{eq:split1} we have~\eqref{eq:gstr29}.

\end{proof}

\begin{proof}[Proof of Proposition~\ref{prop:rigidity}]
   From the Lemma~\ref{lemma:ri_constantNew} $r^i_\lambda(u,\cdot)$ is constant almost everywhere for every $u$ and for every $i$. 
   Fix $u$ and $\lambda$ sufficiently small such that $r^i_\lambda(u,\cdot) \neq \min(|u-\lambda | , |L-\lambda -u|)$. If this is not possible, then $E\cap [0,L)^d$ is either $[0,L)^d$ or $\emptyset$.  Therefore for every $t^\perp_i$, the minimizers of $|s-u|$ for $s\in \partial E_{t^\perp_i}$ are either $u+ r^{i}_{\lambda}(u,t^\perp_i)$ or $u- r^{i}_{\lambda}(u,t^\perp_i)$. The fact that $r^i_{\lambda}(u', t^\perp_i)$ is also constant almost everywhere for $u'\in (u-\varepsilon, u+\varepsilon)$ implies that one of the following three cases holds
   \begin{enumerate}[(a)]
      \item $u+ r^{i}_{\lambda}(u,t^{\perp}_{i})\in\partial E_{t^\perp_i}$ for all $t^{\perp}_{i} \in [0,L)^{d-1}$
      \item $u- r^{i}_{\lambda}(u,t^{\perp}_{i})\in\partial E_{t^\perp_i}$ for all $t^{\perp}_{i} \in [0,L)^{d-1}$
      \item $u+ r^{i}_{\lambda}(u,t^{\perp}_{i})\in\partial E_{t^\perp_i}$ for all $t^{\perp}_{i} \in [0,L)^{d-1}$ and  $u- r^{i}_{\lambda}(u,t^{\perp}_{i})\in\partial E_{t^\perp_i}$ for all $t^{\perp}_{i} \in [0,L)^{d-1}$.
   \end{enumerate}
   Thus, since this holds for every $i$, we have that $E$ must be a checkerboard or a union of stripes. We recall that by a checkerboard we mean  any set whose boundary is the union of affine hyperplanes orthogonal to  coordinate axes, and there are at least two of these directions. 
   
   However, one can rule out checkerboard.
   To see this we consider the contribution to the energy given in a neighbourhood of an edge.  W.l.o.g. we may assume that around this edge the set $E$ is of the following form $-\varepsilon \leq x_1\leq 0$ and $ -\varepsilon \leq x_2 \leq 0$ and $x_i \in (-\varepsilon, \varepsilon)$ for $i\neq 1,2$. 
   Notice that for every $\zeta$ such that $\zeta_1  + x_1 > 0$, $\zeta_2 + x_2 > 0$ and $\zeta_i\in (-\varepsilon,\varepsilon)$ for $i\neq 1,2$,  the integrand in $I^1_{0,L}(E)$ is equal to $1/|\zeta|_1^p$. 
   Then  by setting  $A_{x_1,x_2}:= (-x_1, -x_1+ \varepsilon)\times (-x_2, -x_2 + \varepsilon) \times (-\varepsilon, \varepsilon)^{d-2}$, one has that
   \begin{equation*}
      \begin{split}
         \frac d2 I^1_{0,L}(E) \geq  \int_{(-\varepsilon,0)^2 \times (-\varepsilon, \varepsilon)^{d-2}}    \int_{A_{x_1,x_2}}  \frac{1}{|\zeta|_1^p} \d\zeta \dx \geq  \int_{(0, \varepsilon)^{2d}} \frac{1}{(x_1 + \zeta_1 + x_2 + \zeta_2 + \sum_{i=3}^{d} \zeta_i)^p} \d\zeta \dx 
      \end{split}
   \end{equation*}

   The above can be further bounded from below (up to a constant depending on the dimension) by  
   \begin{equation*}
      \begin{split}
         \varepsilon^{d-2}  \int_{(0,\varepsilon)^{d+2}} \frac{1}{|t|^p} \dt
      \end{split}
   \end{equation*}
   which, by making a change of variables  to polar coordinates,  diverges. 
   A similar calculation was made also in  \cite{2011PhRvB..84f4205G}.

\end{proof}

\section{The one-dimensional problem}
\label{sec:1D_problem}
Let us consider the following one-dimensional functional:  on an $L$-periodic set $E\subset\R$ of locally finite perimeter
 \[
 \Fcal_{\tau,L}^1(E)=\frac1L\Big(-\per(E,[0,L))+\int_{\R}\widehat K_\tau(z)\Big[\per(E, [0,L))|z|-\int_0^L|\chi_E(x)-\chi_E(x+z)|\dx\Big]\dz\Big),
 \]
 where
  \begin{equation*} 
 \begin{split}
 \widehat{K}_{\tau} (z) = \int_{\R^{d-1}} K_{\tau}( z,z' ) \dz'.
 \end{split}
 \end{equation*}

 We will consider as before kernels which satisfy conditions \eqref{eq:Kbound}-\eqref{eq:laplpos}. In particular,
 \begin{equation}
    \label{eq:hatKInequality}
 \frac{C_q}{C(\tau^{1\slash\beta} + |z|)^q}\leq \widehat{K}_{\tau}(z)\leq \frac{CC_q}{(\tau^{1\slash\beta} + |z|)^q},\quad q:=p-d+1,
 \end{equation}

 where $C$ is the constant appearing in \eqref{eq:Kbound} and $C_q$  appears in 
 \begin{equation*}
    \begin{split}
       \int_{\R^{d-1}} \frac{\d\zeta_1 \cdots\d\zeta_{d-1}}{(\tau^{1/\beta} + |\zeta_1| + \cdots + |\zeta_d| )^p }   =  \frac{C_q}{ (\tau^{1/\beta} + |\zeta_1|)^q}.
    \end{split}
 \end{equation*}

 The above functional corresponds to $\Fcal_{\tau,L}(E)$ when the set $E$ is a union of stripes. 

 The main tool of this section is the  reflection positivity technique, introduced in the context of quantum field theory \cite{osterNatale} and then applied for the first time in statistical mechanics in \cite{FroSim}. For the first appearance of the technique in models with short-range and Coulomb type interactions see \cite{fro}. See \cite{glllRP1}, \cite{glllRP3}, \cite{glllRP3bis} for further generalizations to one-dimensional models and \cite{glllRP2}, \cite{MR2864796} for applications to two-dimensional models. 
 
    Only for notational reasons, we follow~\cite{GR}.
    Except for Theorem~\ref{T:main0}, the rest of this section can be indeed  obtained by simple modifications of well-known results.


The main statements that will be shown in this section are Theorem~\ref{T:main0} and the following 


 \begin{theorem}\label{T:1d} 
There exists $C>0$ and $\hat{\tau}>0$ such that for every $0\leq\tau<\hat{\tau}$, $\exists\,h_\tau^*>0$ s.t., for every $L>0$, the minimizers of $\Fcal_{\tau,L}^1$ are periodic stripes of period $h_\tau$ for some $h_{\tau}>0$ satisfying
  \[
     |h_\tau-h_\tau^*|\leq \frac {C} L.
  \]
  $h^*_\tau$ is the period of the stripes giving the optimal energy density among all $[0,L)^d$-periodic sets, as $L$ varies.  
  \end{theorem}

 

  In order to prove Theorem~\ref{T:main0} and Theorem~\ref{T:1d}, we proceed using the method of reflection positivity.  
  Where the proofs are the same as in \cite{GR}, we refer directly to that paper for the details.
  
  For $h>0$, recall that $E_h:=\cup_{k\in \Z} [(2k)h,(2k+1)h]$. Then, define
  \[e_{\infty,\tau}(h):=\Fcal_{\tau,2h}^1(E_h)=\lim_{L\to+\infty} \Fcal_{\tau,L}^1(E_h).\]
  
  
  Analogously to~\cite[Lemma~6.1]{GR}, one can see that
  \begin{equation}\label{eq:einftytau}
  e_{\infty, \tau}(h) = -\frac 1h+ A_\tau(h),
  \end{equation}
  where 
  \begin{equation*}
    \begin{split}
     A_\tau(h) := \int_h^{+\infty} (z-h)\widehat K _\tau (z) \dz + \sum_{k\in \N} \int_0^h \int_{2kh }^{(2k+1)h} \widehat K_\tau(x-y) \dy\dx.
    \end{split}
  \end{equation*}

     Thus one can define $h^*_\tau$ equivalently, as the minimizer of $e_{\infty,\tau}$. 

     For convenience of notation, let us denote by $\bar{e}_{\infty,\tau}$, the corresponding $e_{\infty,\tau}$ relative to the specific kernel defined in \eqref{eq:k11}, namely
\begin{equation*}
K_\tau(\zeta)=\frac{1}{(|\zeta|_1+\tau^{1\slash(p-d-1)})^p}.
\end{equation*}


  \begin{proof}[Proof of Theorem~\ref{T:main0}]

  We will show that for $\tau$ small enough the function $e_{\infty,\tau}$ has a unique minimizer. 

  We will divide the proof of into several steps.

  {\bf Step 0:}  It is not difficult to see that for the specific kernel defined in \eqref{eq:k11} the minimizer of $\bar e_{\infty,0}$ is unique. Indeed, performing calculations analogous to those in Lemma 6.1 of \cite{GR}, one has that for that kernel
  \begin{equation}\label{eq:bareinfty}
  \bar e_{\infty,0}=-\frac 1h+A_0(h)=-\frac 1h+\bar C_{q}\frac{1}{h^{q-1}},
  \end{equation}
  which has the unique minimizer 
  \[
  \bar h^*=((q-1)\bar C_{q})^{1/(q-2)}.
  \]

 {\bf Step 1:} $\widehat K_\tau $ is $C^{\infty}$ and convex. This comes from the definition of reflection positivity. Indeed, if $\widehat K _\tau(z) = \int_0^{+\infty} e^{ - \lambda z } \d\mu(\lambda)$, for a positive measure $\mu$, then  $\widehat K'' _\tau(z) = \int_0^{+\infty} \lambda^2 e^{ - \lambda z } \d\mu(\lambda) \geq 0$.

{\bf Step 2:}  Let $f:\R\to\R$ be a  $C^1$ and convex function and let $f_n:\R\to\R$ be a family of $C^1$ and convex functions such that for every $x\in \R$ one has that $f_n(x) \to f(x)$. One can show that $f'_n(x)  \to f'(x)$ for every $x$.
    Indeed, let $x\in \R$ and $\varepsilon > 0$.
    By definition,  there exists $h>0$ such that $\frac {f(x+h)-f(x)}h<f'(x)+\varepsilon$. Moreover, since $f_n$ converges to $f$ pointwise, there exists $n\in \N$ such that $\frac {f_k(x+h)-f_k(x)}h<f'(x)+\varepsilon$ for all $k\geq n$.
    By using the  convexity and differentiability of $f_k$, one  can  observe that $ f'_k(x)\leq \frac {f_k(x+h)-f_k(x)}h$, hence $f'_k(x)<f'(x)+\varepsilon$ for all $k\geq n$. The  inequality $f'(x) -\varepsilon\leq f'_k(x)$ can be obtained via a similar reasoning.

 {\bf Step 3:} We now show that for every $\delta > 0$ there exists $\bar\tau >0 $ such that  $|\widehat K' _{\tau} (z)| \leq \frac 1 {z^{q+1}}$ for every $z \geq \delta$ and $\tau\leq\bar{\tau}$.  Indeed, let $z \geq \delta$. Then, by assumption, one has that $\widehat K_\tau(z) \geq \frac{C_q} {C(\tau^{q \slash \beta}  + z^q)}$ and $\widehat K _\tau (z/2) \leq \frac {CC_q}{\tau^{q \slash \beta}  + (z/2)^q} $. 
    Thus, from the convexity of $\widehat K_\tau $ proved in Step 1 one has that
    \begin{equation}
      |\widehat K' _\tau(z)| \leq  \frac{|\widehat K_\tau(z) -\widehat K_\tau(z/2)|}{z/2}  \leq  \frac{\Big|   \frac{C_q}{{C}(\tau^{q/\beta} + z ^q)} - \frac{CC_q} {\tau^{q/\beta} + (z/2) ^q} \Big|  }    {z/2}.
      \end{equation}
      Given that $\tau \ll \delta$, there exist $\hat c$, $\hat C$ and $\bar C$, such that
    \begin{equation*}
      |\widehat K' _\tau(z)| \leq  \frac{\Big|   \frac{C_q}{C(\tau^{q/\beta} + z ^q)} - \frac {C_q} {\tau^{q/\beta} + (z/2) ^q} \Big|  }    {z/2} \leq
      \frac{\Big|   \frac{\hat c}{ z ^q} - \frac {\hat C} { (z/2) ^q} \Big|  }    {z/2} \leq \frac{\bar C}{z^{q+1}}.
    \end{equation*}
  {\bf Step 4: } From Steps 2 and Step 3, one has that the convergence of $ \widehat K_\tau(z)$ to $\frac1 {z^{q}}$ in property \eqref{eq:Kmon} can be upgraded to uniform convergence on compact intervals of $(0, + \infty)$. Indeed, let $[a,b]\subset(0,+\infty)$ and let $y \in [a,b]$. Then $\widehat K_\tau(y) =\widehat K_\tau(a) +  \int_a^y \widehat K'_\tau(t) \dt $. On one side, $\widehat K_\tau(a)$ converges to $\frac1{a^q}$. On the other side, the pointwise convergence of $\widehat K'_\tau$ from Step 2 turns in a convergence of the integral $\int_a^y \widehat K'_\tau(t) \dt$ thanks to the bound in Step 3 for $\tau$ small enough depending on $a$ and dominated convergence.
  
  {\bf Step 5: }  From the definition of $A_\tau(h)$ one has that
  \begin{equation*}
  \begin{split}
  A'_\tau(h)  &= \int_h ^{+\infty} \widehat K_\tau (z) \dz+ \sum_{k\in \N}\int_{2kh}^{(2k + 1 )h} \widehat K _\tau ( h - y ) \dy \\ &+ 
  \sum_{k\in \N}(2k+1)\int_0^h \widehat K _\tau (x - (2k+1)h) \dx
  			  - \sum_{k\in \N}2k \int_0^h \widehat K_\tau (x - 2kh) \dx
  \end{split}
  \end{equation*}
  
  and 
  \begin{equation*}
  \begin{split}
  A''_\tau(h) = &- \widehat K _\tau (h) + \sum_{k\in \N} (2k + 1) \widehat K _\tau (-2kh ) - \sum_{k\in \N}2k  \widehat K _\tau ( (1- 2k) h )    
  +\sum_{k\in \N} \int_{2kh}^{(2k+1)h}  \widehat K'_\tau (h-y) \dy
  \\
  &- \sum_{k\in \N}(2k+1)^2 \int_0^h \widehat K' _\tau (x - (2k+1)h ) \dx
  +\sum_{k\in \N} (2k+1) \widehat K _\tau (-2kh)
  \\&+ \sum_{k\in \N}(2k)^2 \int_0^h \widehat K' _\tau (x - (2k)h ) \dx
  - \sum_{k\in \N}2k \widehat K _\tau ((1-2k)h).
  \end{split}
  \end{equation*}
  
  In particular, for $\delta>0$, on $[\delta,+\infty)$ $A'_\tau$ and $A''_\tau$ converge uniformly to $A'_0$ respectively $A''_0$,  where $A_0$ is defined in \eqref{eq:bareinfty}. This follows from Step 4, the decay property \eqref{eq:Kbound}, pointwise convergence of $\widehat K'_\tau$ given by Step 2 and the bound for $\widehat K'_\tau$ given in Step 3.

  {\bf Step 6: }As a consequence of the convergence of $A'_\tau$ in Step 5, one has that 
  \[
  \lim_{\tau\to0} \sup \{ |h-\bar h^*|:\,e'_{\infty,\tau}(h)=0 \} = 0 
  \]
  
  {\bf Step 7: } The function $\bar e_{\infty,0}$ defined in \eqref{eq:bareinfty} is strictly convex in a neighborhood of the unique minimizer. Indeed,
  \[
  \bar e''_{\infty,0}(h)=\frac 1{h^3}\Big[\bar C_q(q-1)q\frac{1}{h^{q-2}}-1\Big]
  \]
  which on the minimizer $\bar h^*=((q-1)\bar C_{q})^{1/(q-2)}$ gives $\frac{1}{{\bar h}^{*3}}(q-1)>0$.
  
  {\bf Conclusion: } By Step 6, critical points of $e_{\infty,\tau}$ for $\tau$ small enough have to be contained in a connected neighborhood of $\bar h^*$ where $\bar e''_{\infty,0}\geq\alpha>0$. Since by Step 5 $A''_\tau$ converges uniformly to $A''_0=\bar e''_{\infty,0}$, then also $e''_{\infty,\tau}>0$ where $e'_{\infty,\tau}=0$ for $\tau$ small. Therefore the critical points of $e'_{\infty,\tau}$ are local minima, and since the neighborhood of $\bar h^*$ where they are contained is connected, then they are also unique.

\end{proof}

  
  The main ingredient to prove Theorem \ref{T:1d} is the following estimate, which recalls~\cite[Lemma~6.3]{GR}.

     \begin{lemma}\label{L:chessboard}
        For every $L-$periodic set $E$ of locally finite perimeter, it holds
        \begin{equation}\label{eq:chessboardineq}
           \Fcal^1_{\tau,L}(E)\ge \frac{1}{2L} \sum_{x\in \partial E\cap [0,L)} (x^+-x) e_{\infty,\tau}(x^+-x)+(x-x^-) e_{\infty,\tau}(x-x^-),
        \end{equation}
        where $x^+,x^-$ have been defined in \eqref{eq:s+s-}.
  \end{lemma}

Theorem \ref{T:1d} follows from \eqref{eq:chessboardineq} in a few lines as in~\cite{GR}.
  
 Lemma \ref{L:chessboard} is a direct consequence of the following two Lemmas.
 
 The first, analogous to~\cite[Lemma~6.6]{GR}, rewrites the second term of the functional $\Fcal_{\tau,L}^1$ as a Laplace transform. 
 
 \begin{lemma}
 \label{lemma:laplace}
 For every $E\in[0,L)$, $\tau>0$,
 \begin{align}
 \Fcal_{\tau,L}^1(E)&=\frac 1L \per(E,[0,L))\Big(-1+\int_{\R}\widehat K_\tau(z){|z|}\dz\Big)\notag\\
 &-\int_0^{+\infty}f(\alpha)\Big(\int_{[0,L)} \int_{\R}|\chi_E(x)-\chi_E(y)|e^{-\alpha|x-y|}\dy\dx\Big)\d\alpha,\label{E:functlapl}
 \end{align}
 where $f$ is a nonnegative integrable function, the inverse Laplace transform of $\widehat K_\tau$.
 \end{lemma}

 \begin{proof}
 Notice that, in comparison with the functional $\mathcal F_{0,L}^1$ studied in~\cite{GR}, the part of $\Fcal^1_{\tau,L}$ which multiplies the perimeter of $E$ has finite energy (due to the presence of $\tau^\beta$ in the kernel) and then it does not need to be transformed through Laplace transform together with the last term.
 
 For the last term, it is sufficient as in Section \ref{sec:setting_and_preliminary_results} to notice that, due to assumption \eqref{eq:laplpos}, $f$ is nonnegative.
 \end{proof}

 \begin{lemma}[{\cite[Lemma 6.10]{GR}}]
 \label{lemma:reflpos}
  For every $\alpha>0$ and every $L-$periodic set $E$
    \begin{equation}\label{chessboardLapl}
       -\int_{[0,L]}\int_{\R} |\chi_{E}(x)-\chi_{E}(y)|e^{-\alpha |x-y|} \dy\dx\ge \frac{1}{2}\sum_{x\in \partial E\cap [0,L)} (x^+-x) e_{\alpha,\infty}(x^+-x)+(x-x^-) e_{\alpha,\infty}(x-x^-). 
    \end{equation}
 where, for $\alpha,h>0$,
 \[e_{\alpha,\infty}(h):=-\frac{1}{2h}\int_{0}^{2h} \int_{\R}|\chi_{E_h}(x)-\chi_{E_h}(y)| e^{-\alpha |x-y|}\dy \dx=\lim_{L\to +\infty} -\frac{1}{L} \int_{[0,L]}\int_{\R} |\chi_{E_h}(x) -\chi_{E_h}(y)| e^{-\alpha|x-y|}\dy \dx.\]
 \end{lemma}
 
 For the proof of this result, based on the so-called reflection positivity technique, see~\cite[Lemma~6.10]{GR}.

  \begin{proof}[Proof of Theorem~\ref{T:1d}]
   Theorem~\ref{T:1d}  follows in the same way as in \cite{GR} once Lemma~\ref{L:chessboard} is shown.
   We refer to Section 6 in \cite{GR} for the details. 
  \end{proof}

\section{Discrete $\Gamma$-convergence Result}
\label{sec:discrete}

In this section, we will show how to adapt the proof of Theorem~\ref{T:gammaconvNew} in order to obtain the analogous result for the discrete setting. In order to state the precise result we need two preliminary steps.

\ 

{\bf Continuous representation of a discrete set $E\subset \varepsilon\Z^d$. }

Given $E\subset \varepsilon \Z^{d}$ $\varepsilon > 0$, we will associate to it $\tilde{E}^\varepsilon\subset \R^{d}$ via
\begin{equation} 
   \label{eq:definizioneTilda}
   \begin{split}
      \tilde{E}^{\varepsilon}:=\bigcup _{i\in E} \Big(i + [-\varepsilon/2,\varepsilon/2)^{d}\Big).
      \qquad
         \per_{1,\varepsilon}(E,[0,L)^d) := \sum_{x\in [0,L)^d\cap \varepsilon \Z^{d}} \sum_{y\sim x} |\chi_{E}(x)-\chi_{E}(y)| \varepsilon^{d-1}.
   \end{split}
\end{equation} 
We call $\per_{1,\varepsilon}$ the $(1,\eps)$-perimeter. Notice that
\begin{equation*} 
   \begin{split}
      \per_{1,\varepsilon}(E,[0,L)^d) = \per_{1}(\tilde{E}^{\varepsilon},[0,L)^d).
   \end{split}
\end{equation*} 

Let also denote by
\begin{equation*}
   \begin{split}
      \Bcal_{\varepsilon} := \insieme{F\subset \R^d:\ \exists E\subset \varepsilon \Z^d \text{ such that } \chi_{\tilde{E}^{\varepsilon}}(x) = \chi_F(x) \text{ for a.e. $x\in \R^d$  }} .
   \end{split}
\end{equation*}

   Since for every $F\in \Bcal_{\varepsilon}$ there exists only one $E\subset \varepsilon \Z^d$ such that $\chi_{\tilde{E}^{\varepsilon} }(x) = \chi_{F}(x)$ for almost every point $x$, we will use the above relation to identify the discrete $E\subset \varepsilon\Z^d$ set with the corresponding continuous set $\tilde{E}^{\varepsilon}$.

Letting $E\subset \Z^d$ and $\tilde{E}^{1} \subset \R^d$ as in \eqref{eq:definizioneTilda} (for $\varepsilon=1$), then the discrete functional defined in \eqref{eq:giuliani}  can be rewritten as
\begin{equation*}
   \begin{split}
   \tilde{\Fcal}^{\discrete}_{J,L}(E):=\frac{1}{L^d}\Big(J \per_{1,1}(E,[0,L)^d)-\int_{[0,L)^d}\sum_{ \zeta\in \Z^{d}\setminus\{ 0 \}}K^\discrete(\zeta)|\chi_{\tilde{E}^1}(x+\zeta)-\chi_{\tilde{E}^1}(x)|\dx  \Big),
   \end{split}
\end{equation*}
where $K^\discrete(\zeta):=\frac{1}{|\zeta|^p}$ for some $p\geq d+2$.

\ 

{\bf  Scaling of the functional. }

Let us denote by $K^\discrete_{\varepsilon}(\zeta):=\frac{\varepsilon^{d}}{|\zeta|^p}$ and let $\tau:= J^{\discrete}_c-J$ for $J<J_c^{\discrete}$, where $J_c^\discrete$ has been defined in \eqref{eq:jc}. 

The factor $\varepsilon^d$ in the definition of $K^\discrete_\varepsilon$, comes from the rescaling in the spatial variables. 
Indeed $\varepsilon^d$ corresponds to the volume of  $[-\varepsilon/2,\varepsilon/2)^d$. 
It is not difficult to see that the measure $\mu_{\varepsilon} = \sum_{\zeta\in \varepsilon \Z^d } K^{\discrete}_\varepsilon(\zeta)\,\delta_\zeta$ 
(by $\delta_\zeta$ we denote the Dirac measure concentrated in $\zeta$)  converge weakly* to $K^\discrete(\zeta)\d\zeta$ in $\R^d\setminus\{0\}$.
Moreover the piecewise constant function associated to $\mu_\varepsilon$ 
\begin{equation*}
   \begin{split}
      x\to \sum_{\zeta \in \varepsilon \Z^{d}\setminus\{ 0\}} \frac{1}{|\zeta|^p}\chi_{Q_\varepsilon(\zeta)}(x)
   \end{split}
\end{equation*}
converges in $L^1_{\lok}(\R^d\setminus\{ 0\})$ to $\frac{1}{|x|^{p}}$.

   As in the continuous case, if one optimizes in the discrete setting on the family of periodic stripes, one finds that the optimal stripes have a width of order $\tau^{-1/(p-d-1)}$ and energy of order $-\tau^{(p-d)/(p-d-1)}$ (see \cite{2014CMaPh.tmp..127G}). Letting $\beta:= p-d-1$, this motivates the rescaling 
\begin{equation}
   \label{eq:changevar-discrete}
   \begin{split}
      x:=\tau^{-1/\beta}\widehat{x}, \quad L:=\tau^{-1/\beta}\widehat{L} \quad \textrm{and} \quad \tilde{\Fcal}^{\discrete}_{J,L}(E):= \tau^{(p-d)/\beta}  \mathcal{F}_{\tau,\widehat L}^{\discrete}( \widehat{E}),
   \end{split}
\end{equation}
where now $\hat{E}\subset \tau^{1/\beta}\Z^{d}$. 

In these variables, the optimal stripes have width of order one. 
For simplicity of notation we will denote by $\kappa = \tau^{1/\beta}$. 
After rescaling the functional (similarly to the continuum) we obtain a new functional defined for $[0,L)^d$-periodic  sets $E\subset \kappa\Z^{d}$ (where $L$ is a multiple of $\kappa$) via the formula
\begin{align}
   \Fcal_{\tau,L}^{\discrete}(E)=\frac{1}{L^d}\Big(-\per_{1,\kappa}(E,[0,L)^d)&+\sum_{\zeta\in\kappa\Z^{d}\setminus\{0\}} K^\discrete_{\kappa}(\zeta) \Big[\int_{\partial \tilde{E}^{\kappa} \cap [0,L)^d} \sum_{i=1}^d|\nu^{\tilde{E}^{\kappa}}_i(x)| |\zeta_i|\d\mathcal H^{d-1}(x)\notag \\
   &-\int_{[0,L)^d}|\chi_{\tilde{E}^{\kappa}}(x)-\chi_{\tilde{E}^{\kappa}}(x+\zeta)|\dx\Big]\Big)\label{eq:ftaudisc},
\end{align}
where $\nu^{\tilde{E}^\kappa}$ is the generalized normal of $\partial \tilde{E}^\kappa$.

Before stating the main result of this section, define the functional
\begin{equation*}
\Fcal_{\tau,L}^\discreteToCont(F):=\left\{\begin{aligned}
&\Fcal^\discrete_{\tau,L}(E) && &\text{if $F=\tilde E^\kappa$, for some $E\subset\kappa\Z^d$}\\
&+\infty && &\text{otherwise}
\end{aligned}\right.
\end{equation*}

Our main result is the following:

\begin{theorem}
   \label{thm:gammaconv-discrete}
   Let $p\geq d+2$ and $L>0$. Then one has that $\Fcal^\discreteToCont_{\tau,L}$ $\Gamma$-converge in the $L^1$-topology as $\tau\to0$  to a functional $\widehat{\Fcal}_{0,L}$ which is invariant under permutation of coordinates and finite on sets of the form $E=F\times\R^{d-1}$ where $F\subset \R$ is $L$-periodic with $\#\{\partial F\cap[0,L)\}<\infty$. On such sets the functional is defined by
\begin{equation*}
\widehat{\mathcal F}_{0,L}(E)=\frac{1}{L}\Big(-\#\{\partial F\cap[0,L)\}+\int_{\R^d} K^\discrete(\zeta)\Big[\sum_{x\in\partial F\cap[0,L)}|\zeta_1|-\int_0^L|\chi_{F}(x)- \chi_{F}(x+\zeta_1)|\dx\Big]\d\zeta\Big).
\end{equation*}

Moreover, if there exists $M$ and a family of sets $E_{\tau}\subset \tau^{1\slash\beta} \Z^d$ such that for every $\tau$ one has that $\mathcal F^\discreteToCont_{\tau,L}(\tilde E_\tau^\kappa)<M$, then up to a relabeling of the coordinate axes, one has that there is a subsequence which converges in $L^1$  to some set $E=F\times\R^{d-1}$ with $\#\{\partial F\cap[0,L)\}<+\infty$.
\end{theorem}

\begin{proof}

As the method used in the continuum  applies almost identically to the discrete setting, 
instead of repeating the same arguments we will give the main steps.

{\bf Step 1:} Lower bound of the functional. 

As in the continuum (see~Section~2), we find a lower bound for the discrete functionals very similar to \eqref{E:fbelow}. In order to obtain it let us define quantities analog to those in  Section~\ref{sec:setting_and_preliminary_results}.

Let 
 \begin{equation}
    \label{eq:gstrf1_discrete}
    \begin{split}
 \Gcal_{\tau,L}^{i,\discreteToCont}(\tilde{E}^{\kappa}):=\Gcal_{\tau,L}^{i,\discrete}({E}):=\int_{[0,L)^{d-1}} \Gcal_{\tau,L}^{1d,\discreteToCont}(\tilde{E}^{\kappa}_{x_i^\perp}) \dx_i^\perp,
    \end{split}
 \end{equation}
 where for every $E\subset \Z$ and  thus $\tilde{E}^{\kappa} \subset \R$, one has that
 \begin{equation}
    \label{eq:gstrf2_discrete}
    \begin{split}
    \Gcal_{\tau,L}^{1d,\discreteToCont}(\tilde{E}^{\kappa}):= \sum_{z\in \kappa\Z \setminus \{0\}} \widehat{K}^\discrete_\kappa(z) \Big(\per(\tilde{E}^{\kappa},[0,L))|z|-\int_0^L |\chi_{\tilde{E}^\kappa}(x)-\chi_{\tilde{E}^{\kappa}}(x+z)|\dx\Big),
    \end{split}
 \end{equation}
and $\widehat{K}^\discrete_\kappa(z)=\sum_{{\zeta' \in \kappa\Z^{d-1}}}K^\discrete_\kappa(z,\zeta')$.

Notice that when $z = 0$, $\widehat K^\discrete_\kappa(z) = +\infty$. 
However in \eqref{eq:gstrf2_discrete} $\widehat K^\discrete_\kappa (0)$ is multiplied by $0$. Thus, by using the standard convention in analysis $+\infty \cdot 0 = 0$, the formula would be the same also including $0$ in the sum.    
Moreover, let
\begin{equation}\label{eq:I_discrete}
   \begin{split}
      \tilde{I}_{\tau,L}^{i,\discreteToCont}(\tilde{E}^{\kappa})&:=\frac{2}{d}\sum_{\zeta \in \kappa\Z^{d}\setminus \{0\}}\int_{[0,L)^d} K^\discrete_\kappa(\zeta) |\chi_{\tilde{E}^{\kappa}}(x)-\chi_{\tilde{E}^{\kappa}}(x+\zeta_i)||\chi_{\tilde{E}^{\kappa}}(x)-\chi_{\tilde{E}^{\kappa}}(x+\zeta_i^\perp)|\dx,\\
      \tilde{I}_{\tau,L}^{\discreteToCont}(\tilde{E}^{\kappa}) &:= \sum_{i=1}^{d} I_{\tau,L}^{i,\discreteToCont}(\tilde{E}^{\kappa}).
   \end{split}
\end{equation}

We also define for $E\subset \R^d$  $[0,L)^d$-periodic 

\begin{equation*}
\label{eq:gstrf10_discrete}
\begin{split}
\Gcal_{0,L}^{i,\discreteToCont}(E)&:=\int_{[0,L)^{d-1}} \Gcal_{0,L}^{1d,\discreteToCont}(E_{x_i^\perp}) \dx_i^\perp,\\
\Gcal_{0,L}^{1d,\discreteToCont}({E})&:= \int_{\R} \widehat{K}^\discrete(z) \Big(\per({E},[0,L))|z|-\int_0^L |\chi_{E}(x)-\chi_{E}(x+z)|\dx\Big)\dz, \\
I_{0,L}^{i,\discreteToCont} (E)& :=\int_{[0,L)^d} \int_{\R^d} K^\discrete(\zeta) |\chi_{E}(x+ \zeta_{i}) - \chi_E(x) | |\chi_{E}(x+ \zeta^\perp_{i}) - \chi_E(x) | \d\zeta \dx
\end{split}
\end{equation*}

In a similar fashion  to Section~\ref{sec:setting_and_preliminary_results}, one obtains that

\begin{equation}
   \label{E:fbelow_discrete}
   \Fcal_{\tau,L}^\discreteToCont(\tilde{E}^{\kappa}) =    \Fcal^\discrete_{\tau,L}({E})\ge \frac{1}{L^d}\Big( -\per_{1}(\tilde{E}^{\kappa},[0,L)^d)+\sum_{i=1}^d \Gcal^{i,\discreteToCont}_{\tau,L}(\tilde{E}^{\kappa})+\sum_{i=1}^d I_{\tau,L}^{i,\discreteToCont}(\tilde{E}^{\kappa})\Big).
\end{equation}
and for $E\subset \kappa\Z$ 
\begin{equation}\label{E:boundg_discrete}
 \Gcal_{\tau,L}^{1d,\discreteToCont}(\tilde{E}^{\kappa}) \geq C_{d,L}    \sum_{x\in \partial \tilde{E}^{\kappa}\cap [0,L)}  \min((x^+-x)^{-\beta},\tau^{-1})+\min((x-x^-)^{-\beta},\tau^{-1}).
\end{equation}
where $x^+,x^-$ are defined in \eqref{eq:s+s-}. Inequality \eqref{E:boundg_discrete} follows also with the discrete kernel in the same way as in the continuum.

{\bf Step 2:} Estimate of the perimeter. 

   From Step~1 and by using a similar reasoning  to the one leading to \eqref{eq:gr1}, one can obtain that for $E\subset \kappa \Z^d$ it holds 
\begin{equation}
   \label{eq:gr1_discrete}
   \per_1(\tilde{E}^{\kappa},[0,L)^d)\leq C_{d,L} L^d \max(1,\Fcal^{\discreteToCont}_{\tau,L}(\tilde{E}^{\kappa})).
\end{equation}

The above in particular implies a compactness result, namely given $E_{\tau}\subset \kappa\Z^{d}$ such that $\Fcal^{\discreteToCont}_{\tau,L}(\tilde{E}^{\kappa}_\tau) < + \infty$, then there exists a set of locally finite perimeter $E\subset \R^d$ and a subsequence $\{ \tilde{E}^{\kappa_{n}}_{\tau_n}\}$ such that 
$\tilde{E}^{\kappa_{n}}_{\tau_n} \to E$  in $L^1$. 

{\bf Step 3:} Lower semicontinuity of $I_{\tau,L}^{i,\discreteToCont}$.

For  $E \subset \kappa\Z^d$, let us define
\begin{equation*}
   \begin{split}
      f^{i}_{\tilde{E}^{\kappa}}(x,\zeta) = |\chi_{\tilde{E}^\kappa}(x+\zeta_i)- \chi_{\tilde{E}^\kappa}(x) | |\chi_{\tilde{E}^\kappa}(x+\zeta^{\perp}_i)- \chi_{\tilde{E}^\kappa}(x) |
   \end{split}
\end{equation*}

Let us fix $1 \gg \delta \gg \kappa$ and let $x,\zeta\in \kappa \Z^d$ with $|\zeta| > \delta$.  Then,  for every $\zeta'\in Q_\kappa(\zeta)$, by using  $\zeta > \delta \gg \kappa$ and Taylor's remainder theorem,  one has that $\frac{1}{|\zeta|^p} \leq  \frac{1}{|\zeta'|^p} + C_d\frac{\kappa}{|\zeta|^{p+1}}$, and thus
\begin{equation}
   \label{gsmstr1}
   \begin{split}
      \int_{Q_{\kappa}(x)} \kappa^d \frac{f^i_{\tilde{E}^\kappa}(x',\zeta)}{|\zeta|^{p}} \dx' 
       &= \int_{Q_{\kappa}(x)}\int_{Q_{\kappa}(\zeta)} \frac{f^i_{\tilde{E}^\kappa}(x',\zeta)}{|\zeta|^{p}} \d\zeta' \dx' \\ &
      = \int_{Q_\kappa(x)} \int_{Q_\kappa(\zeta)} {f^i_{\tilde{E}^{\kappa}}(x',\zeta')}K^\discrete(\zeta') \d\zeta' \dx' + C_d\kappa^{d+1}/|\zeta|^{p+1}.
   \end{split}
\end{equation}

By summing over $x\in [0,L)^d$ and $\zeta\in\kappa \Z^d\setminus \{0\}$ and using the definition  of $I^{i,\discreteToCont}_{\tau,L}$ (see~\eqref{eq:I_discrete})
and using the fact that $\kappa^d \#([0,L)^d \cap \kappa \Z^d) \simeq L^d$,
one has that
\begin{equation*}
   \begin{split}
     I^{i,\discreteToCont}_{\tau,L}(\tilde{E}^{\kappa})  \geq  \int_{[0,L)^d} \int_{\R^d} \chi_{B^c_{\delta}(0)} (\zeta) {f^{i}_{\tilde{E}^{\kappa}}(x,\zeta)}K^\discrete(\zeta) \d\zeta\dx + L^d C_d \kappa/\delta^{p - d+1}.
   \end{split}
\end{equation*}
Finally, let us assume that $\tilde{E}^{\kappa}_{\tau} \subset \R^d$ is such that there exists $E\subset\R^d$ such that $\tilde{E}^{\kappa}_{\tau}$  converges to ${E}$ in $L^1_{\lok}$.  Then 
\begin{equation*}
   \begin{split}
      \liminf_{\tau\downarrow 0} I^{i,\discreteToCont}_{\tau,L}(\tilde{E}^{\kappa}_{\tau}) & \geq \liminf_{\tau\downarrow 0}  \int_{[0,L)^d} \int_{\R^d} \chi_{B^c_{\delta}(0)} (\zeta) {f^i_{\tilde{E}^{\kappa}_{\tau}}(x,\zeta)}K^\discrete(\zeta) \d\zeta\dx \\  & \geq  \int_{[0,L)^d} \int_{\R^d} \chi_{B^c_{\delta}(0)} (\zeta) {f^i_{{E}}(x,\zeta)}K^\discrete(\zeta) \d\zeta\dx ,
   \end{split}
\end{equation*}
thus from the arbitrariness of $\delta$  one has that 
\begin{equation*}
   \begin{split}
      \liminf_{\tau\downarrow 0} I^{i,\discreteToCont}_{\tau,L}(\tilde{E}^{\kappa}_{\tau}) \geq  \int_{[0,L)^d} \int_{\R^d}  {f^{i}_{{E}}(x,\zeta)}K^\discrete(\zeta) \d\zeta\dx,
   \end{split}
\end{equation*}

{\bf Step 4:} Semicontinuity of $\Gcal_{\tau,L}^{i,\discreteToCont}$. 

Let $E_{\tau}\subset \kappa\Z^d$.  Moreover, assume that $\tilde{E}^{\kappa}_{\tau}$ converges to $E$ in $L^1_{\mathrm{loc}}$. Then
\begin{equation}
   \label{eq:gstrf_discrete1}
   \begin{split}
      \lim_{\tau\downarrow 0}\sum_{\substack{z\in \kappa\Z^d\\ |z|>\delta}}\int _{[0,L)^d}  \kappa^d\frac{|\chi_{\tilde{E}^{\kappa}_{\tau}}(x+z) - \chi_{\tilde{E}^{\kappa}_{\tau}}(x) |}{|z|^{p}} \dx = \int_{{\R^d }} \int_{[0,L)^d}  \chi_{\R^d\setminus B_{\delta}(0)}(\zeta) {|\chi_{{E}}(x+\zeta) - \chi_{{E}}(x) |}K^\discrete(\zeta) \d\zeta\dx . 
   \end{split}
\end{equation}

In order to prove the above notice that given $x,\zeta\in \kappa \Z^d$, then whenever $|\zeta| > \delta \gg \kappa$ one has that (similarly to \eqref{gsmstr1})
   
\begin{equation*}
   \begin{split}
      \int_{Q_{\kappa}(x)} \kappa^d \frac{|\chi_{\tilde{E}^{\kappa}_{\tau}}(x' + \zeta) - \chi_{\tilde{E}^{\kappa}_{\tau}}(x')|}{|\zeta |^p}  \dx'& = \int_{Q_{\kappa}(x)} \int_{Q_{\kappa}(\zeta)}  {|\chi_{\tilde{E}^{\kappa}_{\tau}}(x' + \zeta') - \chi_{\tilde{E}^{\kappa}_{\tau}}(x') |}K^\discrete(\zeta') \d\zeta'\dx'\\
      &  + C_d \kappa^{d+1}\slash |\zeta|^{p+1}
   \end{split}
\end{equation*}
Thus summing over $x$ and $\zeta$ one obtains \eqref{eq:gstrf_discrete1}.

We would like to show the semicontinuity $\Gcal^{i,\discreteToCont}_{\tau,L}$.   Let $E_{\tau}\subset \kappa \Z^d$, such that $\tilde{E}^{\kappa}_{\tau}$ converges to $E\subset \R^d$   in $L^1_{\lok}$. Then
\begin{equation*}
   \begin{split}
      \liminf_{\tau\downarrow 0} \Gcal^{i,\discreteToCont}_{\tau,L}(\tilde{E}^{\kappa}_{\tau}) \geq \Gcal^{i,\discreteToCont}_{0,L}(E).
   \end{split}
\end{equation*}

To prove it, it is sufficient to show a similar statement for the $\Gcal^{1d, \discreteToCont}_{\tau,L}$.  Namely let $E_\tau\subset \kappa \Z$  such that $\tilde{E}^\kappa_\tau$ converges to $E\subset \R$ in $L^1_\lok$.  Then

\begin{equation}
   \label{gstrpd1}
   \begin{split}
      \liminf_{\tau\downarrow 0} \Gcal^{1d,\discreteToCont}_{\tau,L}(\tilde{E}^{\kappa}_{\tau}) \geq \Gcal^{1d,\discreteToCont}_{0,L}(E).
   \end{split}
\end{equation}

Let us now show \eqref{gstrpd1}. Given that (see the comment after \eqref{eq:gstrf2})
\begin{equation*}
   \begin{split}
    \per(\tilde{E}^{\kappa}_{\tau},[0,L))|z|-\int_0^L |\chi_{\tilde{E}^\kappa_\tau}(x)-\chi_{\tilde{E}^{\kappa}_{\tau}}(x+z)|\dx \geq 0,
   \end{split}
\end{equation*}
 one has that for every $\delta> 0$, it holds
\begin{equation*}
   \begin{split}
      \Gcal_{\tau,L}^{1d,\discreteToCont}(\tilde{E}^{\kappa}_\tau)\geq \sum_{\substack{z\in \kappa\Z \\ |z|\geq \delta}} \widehat{K}^\discrete_\kappa(z) \Big(\per(\tilde{E}^{\kappa}_{\tau},[0,L))|z|-\int_0^L |\chi_{\tilde{E}^\kappa_{\tau}}(x)-\chi_{\tilde{E}^{\kappa}_\tau}(x+z)|\dx\Big).
   \end{split}
\end{equation*}
Thus by passing to the liminf in the above, and using the lower semicontinuity of the perimeter and \eqref{eq:gstrf_discrete1} one has that
\begin{equation*}
   \begin{split}
      \liminf_{\tau\downarrow 0} \Gcal^{1d,\discreteToCont}_{\tau,L} (\tilde{E}^{\kappa}_{\tau}) &\geq \liminf_{\tau\downarrow 0}  \sum_{\substack{z\in \kappa\Z \\ |z|\geq \delta}} \widehat{K}^\discrete_\kappa(z) \Big(\per(\tilde{E}^{\kappa}_{\tau},[0,L))|z|-\int_0^L |\chi_{\tilde{E}^\kappa_{\tau}}(x)-\chi_{\tilde{E}^{\kappa}_\tau}(x+z)|\dx\Big)\\  &\geq 
        \liminf_{\tau\downarrow 0}  \int_{\R} \chi_{B_{\delta}^c(0)}(z)  \widehat{K}^{\discrete}(z)\Big(\per(\tilde{E}^{\kappa}_{\tau},[0,L))|z|-\int_0^L |\chi_{\tilde{E}^\kappa_{\tau}}(x)-\chi_{\tilde{E}^{\kappa}_\tau}(x+z)|\dx\Big)\dz\\
        &\geq \int_{\R} \chi_{B_{\delta}^c(0)}(z)  \widehat{K}^{\discrete}(z)\Big(\per({E},[0,L))|z|-\int_0^L |\chi_{{E}}(x)-\chi_{{E}}(x+z)|\dx\Big)\dz,
   \end{split}
\end{equation*}
where $\widehat K^\discrete (z) = \int_{\R^{d-1}} K^\discrete(\zeta',z)\d\zeta' $. 
By the arbitrarity of $\delta$ one has the desired claim.

{\bf Step 5:} Gamma limit. 

Let $E_{\tau}\subset \Z^d$ such that 
\begin{equation*}
   \begin{split}
      \sup_{\tau} \Fcal_{\tau,L}^\discreteToCont(\tilde{E}^{\kappa}_{\tau})  < +\infty.
   \end{split}
\end{equation*}

From Step 2, one has that $\tilde{E}^\kappa_{\tau}$ is a sequence of sets of locally finite perimeter, thus there exists $E\subset \R^d$  such that $E$ is a set of locally finite perimeter 
and $\tilde{E}^\kappa_\tau$ converges to $E$ in $L^1_\lok$. 
From the lower semicontinuity Steps 3 and 4, one has that for every $i\in \{ 1,\ldots,d\}$ it holds $\Gcal^{i,\discreteToCont}_{0,L}(E)$, $I^{i,\discreteToCont}_{0,L}(E) < +\infty$.

Thus from the rigidity estimate (Proposition~\ref{prop:rigidity}), one has that $E$ is a union of stripes. 

The $\Gamma$-limit result consists in two parts:  
\begin{enumerate}[(i)]
   \item $\Gamma$-liminf
   \item $\Gamma$-limsup
\end{enumerate}

The  $\Gamma$-liminf is obtained by Step 4 above. In order to obtain the $\Gamma$-limsup,  we need to find a recovery sequence $E_{\tau}$ such that $\Fcal^\discreteToCont_{\tau,L}(\tilde{E}^{\kappa}_{\tau})\to \widehat{\Fcal}_{0,L}(E)$. To do so it is sufficient to consider for every $\tau$,  a union of stripes $\tilde{E}^{\kappa}_\tau$ such that $\tilde{E}^{\kappa}_\tau$ is $[0,L)^d$-periodic and such that it is close in $L^1$ to $E$. 

\end{proof}

\section{Structure of minimizers}
\label{sec:structure_of_minimizers}

In this section we prove Theorem \ref{T:main}, asserting that minimizers of $\Fcal_{\tau,L}$ are periodic stripes of period $h_\tau$, provided $\tau$ is small enough.

   The idea of the proof of Theorem~\ref{T:main} is to  start from  \eqref{E:fbelow}, namely
   \begin{equation}
      \label{gstrF:1}
      \begin{split}
         \Fcal_{\tau,L}(E)\ge 
         \frac{1}{L^d}\Big( -\sum_{i=1}^d\per_{1i}(E,[0,L)^d)+\sum_{i=1}^d \Gcal_{\tau,L}^i(E)+\sum_{i=1}^d I_{\tau,L}^i(E)\Big).
      \end{split}
   \end{equation}
   Since $I^i_{\tau,L}(E) = 0$ if and only if $E$ is a union of  stripes, one has  that  the \lhs and \rhs of the above are equal whenever $E$ is a union of stripes. Thus,  it is sufficient to show that optimal stripes are minimizers of the \rhs of the above. 

   For the r.h.s., in order to show optimality of stripes we will initially use Theorem~\ref{T:gammaconvNew}  in order to reduce ourselves to a situation in which the minimizer of $\FtL$ are close to optimal stripes $S$. This holds for $\tau < \bar{\tau}$, where $\bar{\tau}$ depends on $L$. 
   In this situation we will show that oscillations of the characteristic function of the set $E$ in directions which are orthogonal to the direction of $S$ increase necessarily the \rhs  of \eqref{gstrF:1} (this is done in Lemma~\ref{lemma:new_lemma_pre1}). Thus the \rhs of \eqref{gstrF:1} can be further bounded from below by
   \begin{equation}
      \label{eq:gstrF:2}
      \begin{split}
         \frac{1}{L^d} \Big(-\per_{1i}(E,[0,L)^d)  + \Gcal^i_{\tau,L}(E)\Big),
      \end{split}
   \end{equation}
   where $e_i$ is the orientation of the stripes. 
   Finally, optimal stripes minimize \eqref{eq:gstrF:2}, since it corresponds to the one-dimensional problem of Section \ref{sec:1D_problem}.

\vspace{3mm}

   Informally, the next lemma says the following: suppose that $E$ is $L^1$-close to a set $S$ which is a union of periodic stripes in the direction $e_1$ (see Figure \ref{fig:stability} (a)), then the contribution of the functional from directions other than $e_1$ are nonnegative. This claim is formally expressed in \eqref{eq:lemma_pre1_1}. 
   Moreover, these contribution are equal to zero if and only if the set $E$ is  also union of stripes.

\begin{figure} [!ht]
   \centering
\def\svgwidth{5cm}
\subfloat[] {
      \fbox{\includegraphics[width=5cm]{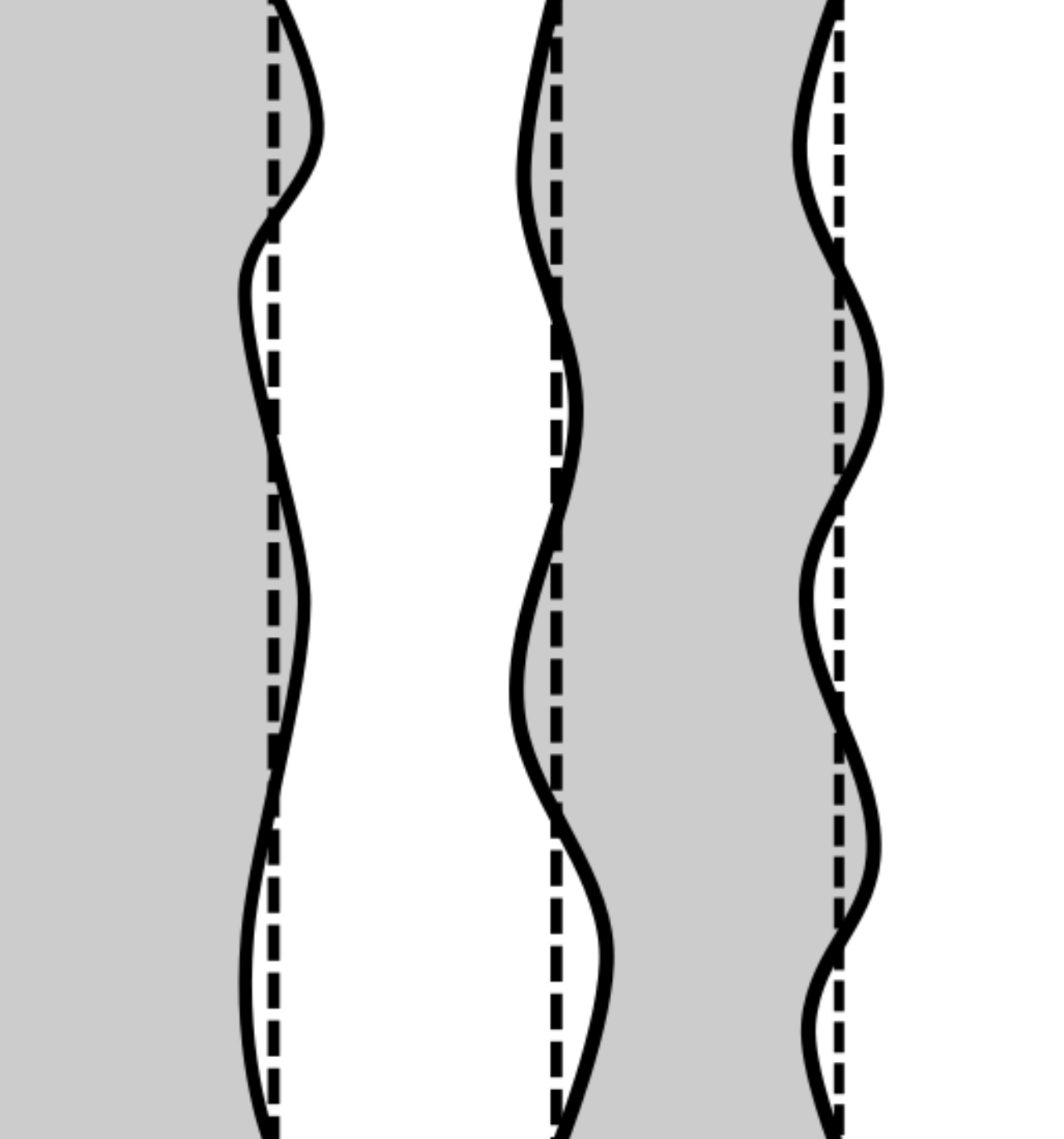}}
}
\subfloat[] {
      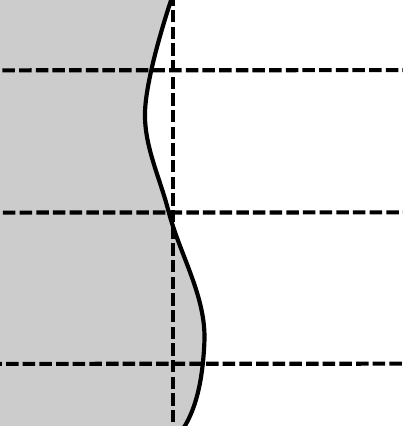
   }
      \caption{In (a) a set which is close to optimal periodic stripes is depicted.  In (b) a  zoomed-in portion where the set deviates from being a union of stripes is depicted. }
      \label{fig:stability}
\end{figure}

   \begin{lemma}[Stability]
      \label{lemma:new_lemma_pre1}
      \label{lemma:stability}
      Let $E\subset \R^d$ be a $[0,L)^d$-periodic set of locally finite perimeter, and $S$ be a set which is a union of periodic stripes, i.e.  (up to exchange of coordinates and translations) there exists $\hat E \subset \R$ such that  $S = \hat E \times \R^{d-1}$ and 
      \begin{equation}
         \hat E = \bigcup_{k\in \Z} [2kh, (2k +1 ) h),
      \end{equation}
      for a suitable $h$.  Then, there exist $\bar{\varepsilon},\bar \tau > 0 $ such that if $\|\chi_E  -\chi_S\|_{L^1} < \bar \varepsilon$  and $\tau < \bar{\tau}$, one has that, for $i\in\{2,\dots,d\}$,
      \begin{equation} 
         \label{eq:lemma_pre1_1}
         \begin{split}
            -  \per_{1i}(E,[0,L)^d) +  \Gcal^{i}_{\tau,L}(E) + I^{i}_{\tau,L}(E)\geq 0 .
         \end{split}
      \end{equation} 
      Moreover, in \eqref{eq:lemma_pre1_1} equality holds if and only if $E$ is a union of stripes in direction $e_1$. 
   \end{lemma} 

   \begin{proof}

Let $i\in\{2,\dots,d\}$. From \eqref{eq:perI} and \eqref{eq:gstrf1},
\[
-\per_{1i}(E,[0,L)^d)+\GtL^i(E)=\int_{[0,L)^{d-1}}\Big[-\#\partial E_{x_i^\perp}+\GtL^{1d}(E_{x_i^\perp})\Big]\dx_i^\perp.
\]

   From \eqref{E:boundg}, 
   \begin{align}\label{E:boundi}
      -\per_{1i}(E,[0,L)^d)+\GtL^i(E)\geq\int_{[0,L)^{d-1}}\sum_{s\in\partial E_{x_i^\perp}} [-1&+C_{d,L}\min\bigl((s^+-s^-)^{-\beta},\tau^{-1}\bigr)\\ &+C_{d,L}\min\bigl((s-s^-)^{-\beta},\tau^{-1}\bigr)]\dx_i^{\perp},
   \end{align}
   where $s^+$ and $s^-$ are defined after \eqref{eq:s+s-}.

Let $\tau_0>0$ and $\eta_0>0$ such that whenever $\tau<\tau_0$ and $\rho\leq\eta_0$
\[
-1+C_{d,L}\min(\rho^{-\beta}, \tau^{-1})>0.
\]

Thus, given that there are at most $L/\eta_{0}$ points $s\in \partial E_{x^\perp_i}$ in the slice with $s^+-s, s-s^- > \eta_0$, one has that the \rhs of \eqref{E:boundi}, can be bounded from below by $-L\slash\eta_{0}$.

Consider the following decomposition
\begin{equation*} 
   \begin{split}
      [0,L)^{d-1}=A_1^i(\eta)\cup A_2^i\cup A_3^i(\eta),
   \end{split}
\end{equation*} 
where 
\begin{align}
A_1^i(\eta)&=\{x_i^\perp\in[0,L)^{d-1}:\,\underset{s\in\partial E_{x_i^\perp}}{\min}(s^+-s,s-s^-)\geq\eta\};\\
A_2^i&=\{x_i^\perp\in[0,L)^{d-1}:\,\partial E_{x_i^\perp}=\emptyset\};\\
A_3^i(\eta)&=\{x_i^\perp\in[0,L)^{d-1}:\,\exists\,s\in\partial E_{x_i^\perp}\text{ s.t. }s^+-s<\eta\text{ or }s-s^-<\eta\}
\end{align}
where $\eta\leq\eta_0$ is such that, for $\tau\leq\bar{\tau}$ with $0<\bar\tau<\tau_0$, and for $\rho\leq\eta$

\begin{equation}
   \label{eq:gstrff1}
   \begin{split}
-1+C_{d,L}\min\bigl(\rho^{-\beta},\tau^{-1}\bigr)>\frac{L}{\eta_0}.
   \end{split}
\end{equation}

The integrand in the \rhs of \eqref{E:boundi} can be estimated as follows:
\begin{enumerate}
   \item  if  $x_i^\perp\in A_1^i(\eta)$ it can be estimated from below by $-L/\eta_0$,
   \item if $x_i^\perp\in A_2^i$ it is zero,
   \item  if  $x_i^\perp\in A_3^i(\eta)$ it is positive.  Indeed,
if $x_i^\perp\in A_3^i(\eta)$, then 
there exists a point $s\in\partial E_{x^\perp_i}$ such that for $\rho$ equal to either $s^+-s$ or $s-s^- $, \eqref{eq:gstrff1} holds. 
Since for the choice of $\eta_0$, the  sum of the negative terms in \rhs of \eqref{E:boundi} is  bigger than or equal to $-L/\eta_{0}$, one has that for every $x^\perp_i\in A^i_{3}(\eta)$ the integrand in the \rhs of \eqref{E:boundi} is positive. 
\end{enumerate}

%

Hence,
\begin{equation*}
-\per_{1i}(E,[0,L)^d)+\Gcal_{\tau,L}^i(E)\geq-\frac{L}{\eta_0}\mathcal H^{d-1}(A_1^i(\eta))+c\mathcal H^{d-1}(A_3^i(\eta)),\qquad \text{for some $c>0$.}
\end{equation*}

We will proceed now to show that  there exists $\varepsilon < \eta$ \st $I_{\tau,L}^{i}(E) > 2\frac{L}{\varepsilon}\Hcal^{d-1}(A_1^{i}(\varepsilon))$ for $\tau$ small enough. In order to do so, we will estimate via slicing the contribution in $I_{\tau,L}^{i}(E)$ for fixed $x_{i}^{\perp}\in A_{1}^{i}(\varepsilon)$ and show that it is larger than $2\frac{L}{\varepsilon}$ for a certain choice of parameters $\varepsilon$ and $\tau$.

Let us  now estimate $I^{i}_{\tau,L}(E)$. Recall that

\begin{equation*} 
   \begin{split}
     \frac{d}{2} I^{i}_{\tau,L}(E) = \int_{\R^{d-1}}\int_{\R}\int_{[0,L)^{d-1}}\int_{[0,L)} K_{\tau}(\zeta) f_{E}(x^{\perp}_i,x_i,\zeta_{i}^{\perp},\zeta_i)\,\dx_{i}\dx_{i}^{\perp}  \d\zeta_{i}\d\zeta_{i}^{\perp} ,
   \end{split}
\end{equation*} 
where $f_E$ was defined in \eqref{eq:fE}, namely
\begin{equation*} 
   \begin{split}
      f_{E}(x^{\perp}_i,x_i,\zeta_{i}^{\perp},\zeta_i)=|\chi_{E}(x_{i}^\perp +x_i+ \zeta_{i}) - \chi_{E}(x_i + x^{\perp}_{i} ) | | \chi_{E}(x_{i}^\perp +x_i+ \zeta_{i}^{\perp}) - \chi_{E}(x_i + x^{\perp}_{i} )  |.
   \end{split} 
\end{equation*} 
Choose $(16\bar{\varepsilon})^{1/d} < \eta$,  $(16\bar\varepsilon)^{1/d} < \varepsilon < \eta$ and fix $t_{i}^{\perp}\in A_{1}^{i}(\varepsilon)$. The choice of $\varepsilon,\bar{\varepsilon}$ is made in order to have \eqref{eq:defBarEps}.

Because of the definition of $A^{i}_1(\varepsilon)$, there exists $t_i\in \partial E_{t_{i}^{\perp}}$ such that one of the following holds 
\begin{enumerate}[(i)]
   \item $(t_i-\varepsilon,t_i)\subset E_{t_{i}^{\perp}}$ and $(t_i,t_i+ \varepsilon)\subset E_{t_{i}^{\perp}}^c$ 
   \item $(t_i- \varepsilon,t_i)\subset E^c_{t_{i}^{\perp}}$ and $(t_i,t_i+\varepsilon)\subset E_{t_{i}^{\perp}}$ ,
\end{enumerate}
 
W.l.o.g., we may assume that (i) above holds (see~Figure~\ref{fig:stability} (b)) and that $i=d$. 
We recall from Section \ref{sec:setting_and_preliminary_results} that for $\varepsilon> 0$ and $t^{\perp}_d\in [0,L)^{d-1}$ we let $Q_{\varepsilon}^{\perp}(t^\perp_{d}) = \{z^\perp_{d} \in [0,L)^{d-1}: |t^{\perp}_{d} - z^{\perp}_{d} | \leq \varepsilon  \}$.

Since $ \|\chi_E -\chi_S \|_{L^1([0,L)^d)} < \bar \varepsilon$,  choosing  $\bar\varepsilon$ as above, one has 
   \begin{equation} 
      \label{eq:defBarEps}
      \begin{split}
         \max \Big(\frac{| Q_{\varepsilon}^{\perp}(t^{\perp}_d)\times(t_{d}-\varepsilon,t_{d}) \cap E^{c}|} {| Q_{\varepsilon}^{\perp}(t_d^{\perp})\times(t_{d}-\varepsilon,t_{d})|} ,
         \frac{|Q_{\varepsilon}^{\perp}(t^{\perp}_{d})\times(t_{d},t_{d}+\varepsilon) \cap E|} {| Q^{\perp}_{\varepsilon}(t^{\perp}_{d})\times(t_{d}-\varepsilon,t_{d})|}  
         \Big)  \geq  \frac{7}{16}.
      \end{split}
   \end{equation} 
      For a set which is a union of stripes in the direction $e_d$ the above is trivial with $1\slash2$ instead of $7\slash16$. For a general set  $E$, we use the hypothesis that it is  close to stripes in the direction $e_1$ less than $\varepsilon^d/16$ with respect to the distance induced by $L^1([0,L)^d)$. 

Indeed, we have that 
\begin{equation*}
   \begin{split}
      \mathcal L^d \Big(\big( Q^{\perp}_{\varepsilon}(t_{d}^{\perp})\times( t_{d} - \varepsilon,t_d )\big) \cap E^c\Big) \geq  \mathcal L^d \Big(\big( Q^{\perp}_{\varepsilon}(t_{d}^{\perp})\times( t_{d} - \varepsilon,t_d )\big) \cap S^c\Big)  -\int_{ Q^{\perp}_{\varepsilon}(t_{d}^{\perp})\times( t_{d} - \varepsilon,t_d ) } |\chi_{E}(x) - \chi_{S}(x) |\dx  \\
      \geq  \mathcal L^d \Big(\big( Q^{\perp}_{\varepsilon}(t_{d}^{\perp})\times( t_{d} - \varepsilon,t_d )\big) \cap S^c\Big)  -\int_{[0,L)^d} |\chi_{E}(x) - \chi_{S}(x) |\dx \geq  \mathcal L^d \Big(\big( Q^{\perp}_{\varepsilon}(t_{d}^{\perp})\times( t_{d} - \varepsilon,t_d )\big) \cap S^c\Big)  -\bar{\varepsilon}
   \end{split}
\end{equation*}
or 
\begin{equation*}
   \begin{split}
      \mathcal L^d \Big(\big( Q^{\perp}_{\varepsilon}(t_{d}^{\perp})\times( t_{d} ,t_d+ \varepsilon )\big) \cap E\Big)  & \geq  \mathcal L^d \Big(\big( Q^{\perp}_{\varepsilon}(t_{d}^{\perp})\times( t_{d},t_d + \varepsilon )\big) \cap S\Big)  -\int_{ Q^{\perp}_{\varepsilon}(t_{d}^{\perp})\times( t_{d}, t_d +\varepsilon) } |\chi_{E}(x) - \chi_{S}(x) |\dx. 
      \\ & \geq  \mathcal L^d \Big(\big( Q^{\perp}_{\varepsilon}(t_{d}^{\perp})\times( t_{d}, t_d+\varepsilon )\big) \cap S\Big)  -\bar{\varepsilon}
   \end{split}
\end{equation*}

Given that for every  $t_{d}^\perp\in [0,L)^{d-1}$ and $S$ periodic union of stripes of period $h$ (actually it would suffice that $S$ are stripes) it holds
\begin{equation*}
   \max \Big(
   \mathcal L^d \Big(\big( Q^{\perp}_{\varepsilon}(t_{d}^{\perp})\times( t_{d},t_d + \varepsilon )\big) \cap S\Big), 
\mathcal L^d \Big(\big( Q^{\perp}_{\varepsilon}(t_{d}^{\perp})\times( t_{d} - \varepsilon,t_d )\big) \cap S^c\Big)
\Big) \geq \varepsilon^d/2,
\end{equation*}
one has the desired claim.

   Hence, from the above, we can further assume that
   \begin{equation} 
      \label{eq:caso1} 
      \begin{split}
         (t_{d}-\varepsilon,t_d) \subset E_{t_d^\perp} \quad \text{and} \quad\frac{| Q_{\varepsilon}^{\perp}(t^{\perp}_{d})\times(t_{d}-\varepsilon,t_{d}) \cap E^{c}|}{| Q^{\perp}_{\varepsilon}(t^{\perp}_{d})\times(t_{d}-\varepsilon,t_{d})|} \geq \frac{7}{16}.
      \end{split}
   \end{equation} 

   For every $s\in (t_{d}-\varepsilon,t_{d})$, $(\zeta_{d}^{\perp} , s) \not \in E$ and $\zeta_{d} +s \in (t_{d},t_{d}+\varepsilon)$ we have that $f_{E}(t^{\perp}_{d},s,\zeta^{\perp}_{d},\zeta_{d}) = 1$. Thus by integrating initially in $\zeta_{d}$ and estimating by \eqref{eq:Kbound} $K_{\tau}(\zeta)$ with $\frac{C}{\varepsilon^{p} + \tau^{p/\beta}}$, we have that 
   \begin{equation*} 
      \begin{split}
         \int_{t_{d}-\varepsilon}^{t_{d}+\varepsilon}\int_{t_{d}-s}^{t_{d}+\varepsilon-s} \int_{Q^{\perp}_{\varepsilon}(t^{\perp}_{d})} f_{E}(t^{\perp}_{d},s,\zeta^{\perp}_{d},\zeta_{d}) &K_{\tau}(\zeta)  \d\zeta^\perp_d \d\zeta_d \d s\geq\\
         &\geq \frac{C}{\varepsilon^{p} + \tau^{ p/\beta }} \varepsilon \int_{Q^{\perp}_{\varepsilon}(t^{\perp}_{d})} \int_{t_{d}-\varepsilon}^{t_{d}} |\chi_{E_{t_{d}^{\perp}}}(s) - \chi_{E_{t_{d}^{\perp}+ \zeta^{\perp}_{d}}}( s)  | \ds \d\zeta_{d}^{\perp} \\ 
         &\geq  \frac{C}{\varepsilon^{p} + \tau^{ p/\beta}} \varepsilon  \int_{Q^{\perp}_{\varepsilon}(t^{\perp}_{d})} \int_{t_{d}-\varepsilon}^{t_{d}} |1 - \chi_{E_{t_{d}^{\perp}+ \zeta^{\perp}_{d}}}( s)  | \ds \d\zeta_{d}^{\perp}  \\
         &\geq  \frac{C}{\varepsilon^{p} + \tau^{ p/\beta}} \varepsilon | Q^{\perp}_{\varepsilon}(t_{d}^{\perp})\times(t_{d}-\varepsilon, t_{d}) \cap E^c| \geq\frac{7C\slash 16 \varepsilon^{d+1}}{\varepsilon^{p} +\tau^{ p/\beta}}.
      \end{split}
   \end{equation*} 
   In order to conclude we notice that 
   \begin{equation}\label{eq:6.13}
      \begin{split}
        \frac{d}{2} I^{d}_{\tau,L}(E) \geq \int_{A_1^d (\varepsilon)}  \Big( \sum_{t_d\in \partial E_{t_d^\perp}} \int_{t_{d} -\varepsilon} ^{t_d +\varepsilon} \int_{t_d-s} ^{t_d+\varepsilon-s} \int_{Q^\perp_\varepsilon(t_d^\perp)} f_E(t_d^\perp,s,\zeta_d^\perp,\zeta_d) \d\zeta_{d}^{\perp} \d\zeta_d \ds\Big) \dt_{d}^{\perp} \\ 
         \geq \int_{A_1^d (\varepsilon)}  \Big( \sum_{t_d\in \partial E_{t_d^\perp}} \frac{C \varepsilon^{d+1}}{\varepsilon^p +\tau^{p/\beta}} \Big)\dt_{d}^{\perp}.
      \end{split}
   \end{equation}
   Finally by choosing $\varepsilon$ and $\tau$ small we have the desired
   result, namely  that  there exists $\varepsilon < \eta$ \st
   $I_{\tau,L}^{d}(E) > 2\frac{L}{\varepsilon}\Hcal^{d-1}(A_1^{d}(\varepsilon))$ for $\tau$ small enough. 
   Up to a permutation of coordinates, this naturally holds also for $i=2,\ldots,d-1$.

   \end{proof}

\begin{proof}[Proof of Theorem \ref{T:main}]
   Before proceeding to the proof, let us explain its strategy. 
   In the first step  we  show  that the minimizers are $L^1([0,L)^d)$-close to some set $S$ that consists of optimal periodic stripes. W.l.o.g. let us assume that $S = \hat{E} \times \R^{d-1}$, where
\begin{equation*} 
   \begin{split}
      \hat E = \bigcup_{k\in \Z} [2kh,(2k+1)h).
   \end{split}
\end{equation*} 

   Once in this configuration, we  do a slicing argument. Namely, we split the contributions of the functional in two parts: one of which is given by
   \begin{equation} 
      \label{eq:contribution_1}
      \begin{split}
         -\int_{[0,L)^{d-1}} \per(E_{x_1^{\perp}},[0,L)) \dx_{1}^{\perp} + \int_{[0,L)^{d-1}} \Gcal_{\tau,L}^{1d}(E_{x_{1}^{\perp}}) \dx_{1}^{\perp} + I^{1}_{\tau,L}(E)
      \end{split}
   \end{equation} 
   and one which is larger than
      \begin{equation} 
         \label{eq:contribution_2}
         \begin{split}
               - \sum_{i=2}^{d} \per_{1i}(E,[0,L)^d)  + \sum_{i=2}^{d}\Gcal_{\tau,L}^{i}(E)  + \sum_{i=2}^{d}I^{i}_{\tau,L}(E).
         \end{split}
      \end{equation} 
   Afterwards, we  notice that optimal stripes minimize the first term \eqref{eq:contribution_1} and  are such that \eqref{eq:contribution_2} is equal to zero. 
   On the other side, for anything that does not consists of stripes and is $L^1([0,L)^d)$-close to $S$, one has that the contribution given from \eqref{eq:contribution_2} is strictly positive. 
   Thus optimal stripes are minimizers for $\Fcal_{\tau,L}$. 
   To prove the positivity of \eqref{eq:contribution_2} in case of nonoptimal stripes, we  use Lemma~\ref{lemma:new_lemma_pre1}.

\textbf{\underline{Step 1}: }
{From the $\Gamma$-convergence result} (see~Theorem~\ref{T:gammaconvNew}), we have that for every $\varepsilon> 0 $, there exists a ${\tau_{0} = \tau_0(\varepsilon)} > 0$ such that, for every $0< \tau < {\tau_0}$ and for every minimizer $E_{\tau}$ of $\Fcal_{\tau,L}$, one has that $E_{\tau}$ is $\varepsilon$-close to ${S}$ in $L^1([0,L)^d)$, where $S$ is a periodic stripe of size $2h$. W.l.o.g., we may assume that $S = \hat{E} \times [0,L)^{d-1}$, where
      \begin{equation*} 
         \begin{split}
            \hat E = \bigcup_{k=0}^{j} [2kh,(2k+1)h),
         \end{split}
      \end{equation*} 
      for some $j\in \N$.

      We fix $\varepsilon=\bar{\varepsilon}$ and $\tau < \bar{\tau}$ as in Lemma~\ref{lemma:new_lemma_pre1}.

\textbf{\underline{Step 2}: } 
Let us consider the original functional $\Fcal_{\tau,L}$, for $\tau\leq\bar\tau$ as in Step~1 and set $E= E_\tau$. Recall that one has 
\begin{align}
   \Fcal_{\tau,L}(E)&=\frac{1}{L^d}\Big(-\per_1(E,[0,L)^d)+\int_{\R^d} K_\tau(\zeta) \Big[\int_{\partial E \cap [0,L)^d} \sum_{i=1}^d|\nu^E_i(x) | |\zeta_i|\d\hausd^{d-1}(x)\notag\\&-\int_{[0,L)^d}|\chi_E(x)-\chi_E(x+\zeta)|\dx \Big]\d\zeta\Big)\notag\\
   &\geq\frac{1}{L^d}\Big(-\per_{11}(E,[0,L)^d)+\int_{\R^d} K_\tau(\zeta) \Big[\int_{\partial E \cap [0,L)^d} |\nu^E_1(x)| |\zeta_1|\d\hausd^{d-1}(x)\notag\\
   &-\int_{[0,L)^d}|\chi_E(x)-\chi_E(x+\zeta_1)|\dx\Big]d\zeta\Big)\label{E:f0}
   \\ & - \sum_{i=2}^{d} \per_{1i}(E,[0,L)^d)  + \sum_{i=2}^{d}\Gcal_{\tau,L}^{i}(E)  + \sum_{i=2}^{d}I^{i}_{\tau,L}(E).\label{E:f2}
\end{align}
One notices immediately that, thanks to Theorem \ref{T:1d}, if $E=\hat{E}_\tau\times [0,L)^{d-1}$ with $\hat{E}_\tau$ one-dimensional periodic set of period $h_\tau$, then the first term \eqref{E:f0} of $\Fcal_{\tau,L}$ is minimized, while the terms in \eqref{E:f2} are equal to zero. 

On the other hand, from Lemma~\ref{lemma:new_lemma_pre1}, if a minimizer $E$ does not have such a structure, or more in general $E$ is not a stripe in direction $e_1$, then the last term \eqref{E:f2} is strictly positive.

\end{proof}

\begin{proof}[Proof of Theorems \ref{T:main2}] 
   This claim is implied by the one-dimensional result of Theorem~\ref{T:1d} once from Theorem~\ref{T:main} one knows that the minimizers are stripes. 
\end{proof}

\section{Independence of $\tau$ from $L$}
\label{sec:taul}

The main purpose of this section will be to prove  Theorem~\ref{T:1.3}.

   At the end of this section, we will briefly say how to deal with the discrete setting, namely how to prove Theorem~\ref{T:1.6}.

\subsection{Outline of the Proof}

Let us first give an idea of the proof. 

We  say that a union of stripes $S$ is oriented along  the direction $e_i$, if $S$ is invariant with respect to every translation orthogonal to $e_i$. 

As in Section~\ref{sec:structure_of_minimizers}, instead of the functional $\Fcal_{\tau,L}$ it is convenient to consider the \rhs of \eqref{gstrF:1} and show that its minimizers are stripes. 

The main ingredient in the proof of  Theorem~\ref{T:main} was to use the rigidity argument in order to show that the minimizers of $\FtL$ are close to being stripes (say oriented along $e_i$).
Once in this situation we showed that on slices $E_{t_j^\perp}$ in directions $e_j\neq e_i$ having points in $\partial E_{t_j^\perp}$ increases necessarily the energy. Thus we are left to optimize on the slices in direction $e_i$. On the slices in direction $e_i$ the energy contribution is bounded from below by the energy contribution of the periodic stripes of period $h^*_{\tau}$. 

The main  difficulty in proving Theorem~\ref{T:1.3} when compared to Theorem~\ref{T:main} lies in the fact that the rigidity result (Proposition~\ref{prop:rigidity} and Theorem~\ref{T:gammaconvNew}) can not be applied directly in order to imply the closeness of the minimizers in $L^1([0,L)^d)$ to optimal stripes or even stripes. This is due to the fact that the rigidity argument works for fixed $L$ and $\tau\downarrow 0$.

In order to overcome this issue,  the r.h.s. of \eqref{gstrF:1} on an arbitrary large cube of size $L$ is rewritten as an average of local contributions on smaller cubes of size $l$ (the functionals $\bar F_\tau(E,Q_l(z))$ defined in \eqref{eq:fbartau}), where $l$ is fixed independently of $L$ and $l<L$ ($l$ depends only on a constant depending only on the dimension which comes out of the estimates, as explained at the end of this outline). Namely, we will show that (see~\eqref{eq:gstr14})
\begin{equation}
   \label{eq:stimaDaDimostrare1}
   \begin{split}
      \Fcal_{\tau,L}(E) \geq \frac{1}{L^d}\int_{[0,L)^d} \bar F_{\tau}(E, Q_{l}(z))\dz.
   \end{split}
\end{equation}

   The aim of this section will be to show  that  the minimizers of the \rhs of the above are optimal periodic stripes and on optimal periodic stripes the above inequality is an equality. 
   In this outline, we will speak of contributions to the energy of subsets of $[0,L)^d$. For a generic subset $B \subset [0,L)^d$, such contribution is 
   \begin{equation*}
      \begin{split}
         \int_{B} \bar{F}_{\tau}(E,Q_l(z))\dz. 
      \end{split}
   \end{equation*}
   When a set $J$ is contained in a one-dimensional slice then its contribution is given by
   \begin{equation}
      \label{eq:contributionSegment}
      \begin{split}
         \int_{J} \bar{F}_{\tau}(E,Q_l(z)) \d\hausd^1(z). 
      \end{split}
   \end{equation}
   where $\hausd^1$ is the usual one-dimensional Hausdorff measure. 
   Indeed, since we use slicing arguments in the proof, often the contribution on a set $B$ will be recovered by integrating the contributions given by its slices. 

Let $E_\tau$ be a minimizer.  
The functional $\bar{F}_{\tau}(E_\tau,Q_l(z))$ will contain by construction a term of the form
\begin{equation}
   \label{eq:outline1}
   \frac{1}{l^d}\sum_{i=1}^d \int_{Q_{l}(z)}\int_{ \R^d} {|\chi_{E_\tau}(x + \zeta_i) - \chi_{E_\tau}(x)||\chi_{E_\tau}(x + \zeta^\perp_i) - \chi_{E_\tau}(x)|}K_{\tau}( \zeta )\d\zeta\dx,
\end{equation}
which is a kind of local version (on scale $l$) of the cross-interaction term $I_{\tau,L}(E_\tau)$ defined in \eqref{eq:I}.
As we prove in the local rigidity lemma, namely Lemma~\ref{lemma:local_rigidity}, such term will be large, for $E$ not close in $L^1$ to stripes in $Q_{l}(z)$ for $\tau < \tau_0(l)$ (our measure of closeness will be quantified in Definition \ref{def:defDEta}). This is the local counterpart of the rigidity Proposition \ref{prop:rigidity}.
A first consequence of this fact is that, from average arguments, one has that for ``most'' of the $z$ contained in $[0,L)^d$, it holds that $E_{\tau}\cap Q_{l}(z)$ has to be close to a set which is a union of stripes. The aim is then to prove that one is $L^1$-close to stripes in some fixed direction on the whole cube $[0,L)^d$.

 A clearer picture of what happens in $[0,L)^d$ is given by  
the following decomposition: 

$[0,L)^d =A_{-1}\cup A_0 \cup \ldots \cup A_d$ where
\begin{itemize}
   \item $A_i$ with $i > 0$ are the set of points $z$ such that there is only one direction $e_i$, such that $E_\tau\cap Q_{l}(z)$ is close to stripes oriented in  direction~$e_i$. 
   \item $A_{-1}$ is the set of points $z$ such that there exist directions $e_i$ and $e_j$ ($i\neq j$) and stripes $S_i$ (oriented in direction $e_i$)  and $S_j$ (stripes oriented in direction $e_j$) such that  $E_\tau\cap Q_{l}(z)$ is close to both $S_i\cap Q_l(z)$ and $S_j\cap Q_l(z)$. In particular, this implies that either $|E_\tau\cap Q_l(z)|\ll l^d$ or $|E_{\tau}^c\cap Q_l(z)|\ll l^d$ (see Remark \ref{rmk:lip} (ii)).
   \item $A_{0}$ is the set of points $z$ where none of the above points is true.
\end{itemize}

The aim is then to show that $A_0\cup A_{-1} = \emptyset$ and  that there exists only one $A_i$ with $i >  0$. 
Thus, by the local version of the Stability Lemma \ref{lemma:new_lemma_pre1}, namely Lemma \ref{lemma:stimaContributoVariazionePiccola}, minimizers must be stripes, which by Theorem \ref{T:main2} are periodic of period $h^*_\tau$ (with $\tau$ depending on $l$ and not on $L$).

 As we will show in the proof of Theorem \ref{T:1.3},  $A_0 \cup A_{-1}$ separates the different $A_i$, namely every continuous curve $\gamma:[0,T]\to [0,L)^d$ intersecting $A_i$ and $A_j$ has necessarily to intersect $A_0 \cup A_{-1}$.

Let us initially explain what is intuitively expected:
\begin{enumerate}[(i)]
   \item for any $z\in A_i$ with $i > 0$,  when slicing in directions orthogonal to $e_i$, alternation between regions in $E$ and regions in $E^c$ should increase the energy (similarly 
      to Lemma~\ref{lemma:new_lemma_pre1}). Thus one expects the contribution of $A_i$ to be  
      \begin{equation*}
         \begin{split}
            C^*_\tau |A_i |/L^d - C_l |\partial A_i |/L^d
         \end{split}
      \end{equation*}

      where $C_l$ is a constant depending on $l$ and $C^*_\tau$ is the energy of periodic stripes of width $h^*_\tau$.  
   \item for any $z\in A_0$ or $z\in A_{-1}$ we expect sub-optimal contributions, namely larger than $C^*_\tau$. Thus having $A_0\cup A_{-1}$ is not energetically convenient. Since $A_0 \cup A_{-1}$ separates the different $A_i$, one has that $|A_0\cup A_{-1}|$ acts like a boundary term and compensates the boundary term $|\partial A_i|$ in (i).
\end{enumerate}
We will show that by choosing $\tau$ small but independent of $L$, the contribution of $A_{0}\cup A_{-1}$ balances the contribution due to the presence of $\partial A_{i}$. 

\vspace{2mm}

Let us now give more specific technical details as a guideline in the reading of the proof. 

In a similar way to the proof of Theorem~1.1 (see~Lemma~\ref{lemma:new_lemma_pre1}), we first show that, once we are in a region $A_i$ with $i>0$, alternations between regions in $E$ and regions in $E^c$ on slices in directions perpendicular to $e_i$ increase necessarily the value of the functional (see Lemma \ref{lemma:stimaContributoVariazionePiccola}). 

{Therefore, we can ignore for regions $A_i$ contributions along $e_j$ for $j\neq i$. }

Thus we are left to bound from below the contributions of the slices of $A_i$s  in direction $e_i$, for all $i> 0$. 


Let us consider a slice of $[0,L)^d$ in the direction $e_i$. There are two cases: 
\begin{enumerate}[(i)]
   \item all the slice is contained in $A_i$; 
   \item there are points in the slice belonging to $\partial A_{i}$. 
\end{enumerate}


In the first case, we show that the contribution of the slice to the energy is bigger or equal to $C^{*}_{\tau}L$, which would be the contribution of periodic stripes of period $h^*_{\tau}$. 

In the second case, points in the slice which belong to  the boundary of $A_i$ necessarily belong to $A_0 \cup A_{-1}$ and we prove the following estimates. 
Let $I \subset A_i$ be a maximal interval on the slice in direction $e_i$. 
The optimal contribution of $I$ whenever $\partial I\cap A_0\neq\emptyset$ is bigger than $(C^*_\tau|I| - M_0l)$ where $|I|$ is the length of the interval, $C^*_\tau$ is the optimal energy density for stripes of width $h^*_\tau$ and $M_0$ is a constant not depending on $\tau$ but depending only on the dimension
(see Lemma \ref{lemma:stimaLinea}). 
The constant $M_0$ comes from the nonlocal interactions close to the boundary. 
Analogously, if $I\subset A_i$ satisfies  $\partial I\subset A_{-1}$ one has that the contribution has the better lower bound  $(C^*_\tau|I| - M_0/l)$ (due to the fact  that, close to the boundary and then to $A_{-1}$, one is close also to stripes in a direction different from $i$ and then Lemma \ref{lemma:stimaContributoVariazionePiccola} can be applied).


In order to balance the negative term ($-M_0l$ or $-M_0/l$) coming from the presence of the boundary of $I$,  we will use the fact that the adjacent regions in $A_0$ or $A_{-1}$ are ``thick'' enough. 
For $A_0$, this is a consequence of the fact that the map 
\begin{equation*}
E \mapsto \text{``$L^1$-distance of $E$ from the stripes''}.
\end{equation*}
is Lipschitz in $L^1$ (see Remark \ref{rmk:lip} (i)). For $A_{-1}$ this follows from the continuity of the $L^1$ measure w.r.t. translations.

If $\partial I \cap A_0 \neq \emptyset$, choosing  $\tau=\tau(l)$  sufficiently small but finite, the term \eqref{eq:outline1} in such a region will give a large positive contribution $M$ (see the Local Rigidity Lemma \ref{lemma:local_rigidity}), that will compensate the term $-M_0l$.

If $\partial I\cap A_{-1} \neq \emptyset$, the total negative contribution of an interval $I\subset A_i$ and the neighbouring $J\subset A_{-1}$ will be at most of the order of $C^*_{\tau}|I|-\max(M_0,1)|J|/l$, due to the fact that $A_{-1}$ is ``thick'', i.e. $|J|\geq1$.  

In the end, integrating the contribution of the slices and dividing by $L^d$, the energy will be  bounded from below by 
\begin{equation*}
C^*_\tau\sum_{i=1}^d|A_i|/{L^d}-M_0\frac{|A_0\cup A_{-1}|}{lL^d}\geq C^*_{\tau}-C^*_{\tau}\frac{|A_0\cup A_{-1}|}{L^d}-M_0\frac{|A_0\cup A_{-1}|}{lL^d},
\end{equation*}
which, provided $l$ is chosen bigger than some constant depending only on the dimension chosen at the beginning, is greater than $C^*_\tau$ and strictly greater than $C^*_\tau$ if $|A_0\cup A_{-1}|>0$.

Thus, there exists just one $A_i$, $i>0$, with $|A_i|>0$. 
Then, as a consequence of Lemma \ref{lemma:stimaContributoVariazionePiccola}, when one slices $[0,L)^d$ in directions orthogonal to $e_i$, any set that deviates from being exactly a stripe in direction $e_i$ gives a positive contribution (see \eqref{eq:gstr21}) while slicing in direction $e_i$ one has that periodic unions of stripes of period $h^*_\tau$ in direction $e_i$ are optimal. 
Therefore, as in the proof of Theorem \ref{T:main} one gets that the minimizers must be periodic unions of stripes of period $h^*_\tau$ in direction $e_i$.

\subsection{Preliminary Lemmas}

Let $t\in \R^d$. We recall that we denote by $t_{i} = \scalare{t, e_{i}}e_{i}$ and $t ^{\perp}_{i} = t - t_{i}$. We will also denote by $Q^{\perp}_{l}(z_i^\perp)$ the cube of size $l$ centered at $z$ in the subspace which is orthogonal to $e_{i}$. As explained in Section \ref{sec:setting_and_preliminary_results}, with a slight abuse of notation with might identify points $x_i$ with their $e_i$-coordinates in $\R$ and points $x_i^\perp$ with points of $\R^{d-1}$.

\emph{In order to simplify notation, we will use $A\lesssim B$, whenever there exists a constant $\bar{C}_{d}$ depending only on the dimension $d$ such that $A\leq \bar{C}_d B$.}

In the definitions \eqref{eq:ritau}, \eqref{eq:witau} and \eqref{eq:vitau} below, we introduce the different terms which, summed together, give rise to the ``local contribution'' $\bar F_\tau(E, Q_l(z))$ to the energy on a square $Q_l(z)$ (defined in \eqref{eq:fbartau}). Let us see how these terms come out naturally from the lower bound \eqref{gstrF:1} on the functional $\mathcal F_{\tau,L}$.

Let us recall that (see~\eqref{gstrF:1} and \eqref{E:fbelow}) 
\begin{align}
   \label{eq:gstr1}
      \FtL (E) &\geq 
      \frac{1}{L^d}\Big( -\sum_{i=1}^d\per_{1i}(E,[0,L)^d)+\sum_{i=1}^d \Gcal_{\tau,L}^i(E)+\sum_{i=1}^d I_{\tau,L}^i(E)\Big)\notag\\
      &=-\frac{1}{L^d}\sum_{i=1}^{d}\per_{1i}(E,[0,L)^d) + \frac{1}{L^d}\sum_{i=1}^{d} \Big[\int_{[0,L)^d\cap \partial E} \int_{\R^d} |\nu^{E}_{i} (x)| |\zeta_{i} | K_{\tau}(\zeta) \d\zeta \d\mathcal H^{d-1}(x) \notag\\ 
      &- \int_{[0,L)^d } \int_{\R^d} |\chi_{E}(x + \zeta_i)  - \chi_{E}(x) | K_{\tau}(\zeta) \d\zeta\dx \Big] \notag\\ 
               &+ \frac{2}{d} \frac{1}{L^d}\sum_{i=1}^d \int_{[0,L)^d } \int_{\R^d} |\chi_{E}(x + \zeta_{i}) -\chi_{E}(x) | | \chi_{E}(x + \zeta^{\perp}_i) - \chi_{E}(x) | K_{\tau}(\zeta)
                  \d\zeta \dx
\end{align}

and recall that in the above equality holds whenever the set $E$ is a union of stripes.  Thus, proving that optimal stripes are the minimizers of the r.h.s. of \eqref{eq:gstr1} implies that they are the minimizers for $\FtL$.

Now we want to further express the \rhs  of  \eqref{eq:gstr1} as a sum of contributions obtained first by slicing and then considering interactions with neighbouring points on the slice, namely
\begin{equation*}
   \begin{split}
      &-\frac{1}{L^d}\per_{1i}(E,[0,L)^d)+ \frac{1}{L^d}\Big[\int_{[0,L)^d\cap \partial E} \int_{\R^d} |\nu^{E}_{i} (x)| |\zeta_{i} | K_{\tau}(\zeta)\d\zeta\d\mathcal H^{d-1}(x) \\& - \int_{[0,L)^d } \int_{\R^d} |\chi_{E}(x + \zeta_i)  - \chi_{E}(x) | K_{\tau}(\zeta) \d\zeta \dx \Big]
      = \frac{1}{L^d}\int_{[0,L)^{d-1}}  \sum_{s\in  \partial E_{t^{\perp}_{i}}\cap [0,L]} r_{i,\tau}(E,t_{i}^{\perp},s)
      \dt^{\perp}_i, 
   \end{split}
\end{equation*}
where  for $s\in \partial E_{t_{i}^{\perp}}$,
\begin{equation}\label{eq:ritau}
   \begin{split}
      r_{i,\tau}(E, t_{i}^{\perp},s) := -1 + \int_{\R} |\zeta_{i}| \widehat{K}_{\tau}(\zeta_{i})\d\zeta_i &- \int_{s^-}^{s}\int_{0}^{+\infty} |\chi_{E_{t_{i}^{\perp}}}(u + \rho) - \chi_{E_{t_{i}^{\perp}}}(u) | \widehat{K}_{\tau}(\rho)\d\rho \du  \\ & - \int_{s}^{s^+}\int_{-\infty}^{0} |\chi_{E_{t_{i}^{\perp}}}(u + \rho) - \chi_{E_{t_{i}^{\perp}}}(u) |\widehat{K}_{\tau}(\rho) \d\rho \du\\ 
   \end{split}
\end{equation}
and
\begin{equation}
   \label{eq:gstr2}
   \begin{split}
      s^+ &= \inf\{ t' \in \partial E_{t_{i}^{\perp}}, \text{with } t' > s  \} \\ s^- &= \sup\{ t' \in \partial E_{t_{i}^{\perp}}, \text{with } t' < s  \}
   \end{split}
\end{equation}
are as in \eqref{eq:s+s-}.


Indeed, given that $E$ is a set of locally finite perimeter, we can use slicing arguments (see Section~\ref{sec:setting_and_preliminary_results}). Let us slice in the directions $e_i$, $i=1,\ldots,d$. One has that
\begin{equation*}
   \begin{split}
     &-\frac{1}{L^d}\per_{1i}(E,[0,L)^d)+ \frac{1}{L^d}\Big[\int_{[0,L)^d\cap \partial E} \int_{\R^d} |\nu^{E}_{i} (x)| |\zeta_{i} | K_{\tau}(\zeta)\d\zeta \d\mathcal H^{d-1}(x)\\
    & - \int_{[0,L)^d } \int_{\R^d} |\chi_{E}(x + \zeta_i)  - \chi_{E}(x) | K_{\tau}(\zeta) \d\zeta \dx  \Big]
       =  \frac{1}{L^d}\int_{[0,L)^{d-1}} \Big  ( -\per(E_{t^\perp_i},[0,L)) + \sum_{s\in \partial E_{t^\perp_i}}\int_\R |\rho| \widehat{K}_{\tau}(\rho)\d\rho    \\ 
      &- \int_{0}^{L} \int_{-\infty}^0  |\chi_{E_{t^\perp_i}}(\rho+ u) - \chi_{E_{t^\perp_i}}(u)| \widehat{K}_{\tau} (\rho)\d\rho  \du 
       - \int_{0}^{L} \int_0^{+\infty}  |\chi_{E_{t^\perp_i}}(\rho+ u) - \chi_{E_{t^\perp_i}}(u)| \widehat{K}_{\tau} (\rho)\d\rho  \du  \Big) \dt^\perp_i.
   \end{split}
\end{equation*}

Now, given a measurable  and $L$-periodic function $f$ in $\R$, it is immediate to notice that 
\begin{equation}
  \label{eq:fu}
   \begin{split}
      \int _0^L f(u) \du = \sum_{s\in \partial E_{t^\perp_i}} \int_{s^-} ^s f(u)\du = \sum_{s\in \partial E_{t^\perp_i}} \int_{s} ^{s^+} f(u)\du.
   \end{split}
 \end{equation}

 Therefore we have that

 \begin{equation*}
 \begin{split}
   \int_{0}^{L} \int_{-\infty}^0  |\chi_{E_{t^\perp_i}}(\rho+ u) - \chi_{E_{t^\perp_i}}(u)| \widehat{K}_{\tau} (\rho)\d\rho  \du  =\sum_{s\in \partial E_{t^\perp}} \int_{s} ^{s^+}
   \int_{-\infty}^0   |\chi_{E_{t^\perp_i}}(\rho+ u) - \chi_{E_{t^\perp_i}}(u)| \widehat{K}_{\tau} (\rho)\d\rho  \du 
 \end{split}
 \end{equation*}
 and analogously
 \begin{equation*}
 \begin{split}
   \int_{0}^{L} \int_0^{+\infty}  |\chi_{E_{t^\perp_i}}(\rho+ u) - \chi_{E_{t^\perp_i}}(u)| \widehat{K}_{\tau} (\rho)\d\rho  \du  =\sum_{s\in \partial E_{t^\perp}} \int_{s^-} ^{s}
   \int_0^{+\infty}   |\chi_{E_{t^\perp_i}}(\rho+ u) - \chi_{E_{t^\perp_i}}(u)| \widehat{K}_{\tau} (\rho)\d\rho  \du.
 \end{split}
 \end{equation*}

For notational reasons it is convenient to introduce  the one-dimensional analogue of \eqref{eq:ritau}. Namely, let $E\subset \R$ be a set of locally finite perimeter and let $s^-, s,s^+\in \partial E$.   We define
\begin{equation}\label{eq:rtau1D}
   \begin{split}
      r_{\tau}(E,s) := -1 & + \int_\R |\rho| \widehat{K}_{\tau}(\rho)\d\rho  -  \int_{s^-}^{s} \int_0^{+\infty}  |\chi_{E}(\rho+ u) - \chi_{E}(u)| \widehat{K}_{\tau} (\rho)\d\rho  \du \\ & - \int_{s}^{s^+} \int_{-\infty}^0  |\chi_{E}(\rho+ u) - \chi_{E}(u)| \widehat{K}_{\tau} (\rho)\d\rho  \du. 
   \end{split}
\end{equation}

The quantities defined in \eqref{eq:ritau} and \eqref{eq:rtau1D} are related via $r_{i,\tau}(E,t^\perp_i,s) = r_{\tau}(E_{t^\perp_{i}},s)$.

The next Remark is the analogue of \eqref{E:boundg}, where now instead of considering all the contributions from the points in $\partial E_{t_i^\perp}$, we restrict to the points neighbouring $s$.

\begin{remark}
   \label{rmk:stimax1}
   There exists $\eta_0 > 0$ and $\tau_{0} > 0$ such that, for $E\subset \R^d$, $s^-,s,s^+\in \partial E_{t_{i}^{\perp}}$ three consecutive points, whenever $\tau< \tau_0$ and  $\min(|s - s^- |,|s^+ -s |) <\eta_0$, then $r_{i,\tau}(E,t_{i}^{\perp},s) > 0$. 
   
   Indeed, since
   \begin{equation*} 
      \begin{split}
        \forall\,\rho\in(0,+\infty),\quad\text{it holds:}\quad \int_{s^-}^{s}  |\chi_{E_{t^\perp_i}}(u + \rho) -\chi_{E_{t^\perp_i}}(u) | \du \leq \min(\rho,|s - s^-|)\\
         \forall\,\rho\in(-\infty,0),\quad\text{it holds:}\quad \int_{s}^{s^+}  |\chi_{E_{t^\perp_i}}(u + \rho) -\chi_{E_{t^\perp_i}}(u) | \du \leq \min(-\rho,|s - s^+|),
      \end{split}
   \end{equation*} 
  (see \eqref{toprovebandesim}), thus 
   \begin{equation*} 
      \begin{split}
         r_{i,\tau}(E,t_{i}^{\perp},s) & \geq -1 + \int_{0} ^{+\infty} \big(\rho - \min(\rho,|s - s^-|) \big) \widehat{K}_{\tau}(\rho) \d\rho +  \int_{-\infty} ^{0} \big(-\rho - \min(-\rho,|s - s^+|)\big) \widehat{K}_{\tau}(\rho) \d\rho. 
          \\ & \geq -1 + \int_{2|s - s^-|} ^{+\infty} \rho \widehat{K}_{\tau}(\rho) \d\rho + \int_{-\infty}^{-2|s - s^+|} -\rho \widehat{K}_{\tau}(\rho) \d\rho. 
      \end{split}
   \end{equation*} 
   From the above formula the claim follows directly. 
   Moreover, by using the elementary inequality 
   \begin{equation*}
      \begin{split}
        \min\Big(\int_{-\infty}^{-2 \alpha } -\rho \widehat{K}_{\tau}(\rho) \d\rho,\int_{2\alpha}^{+\infty} \rho\widehat K_\tau(\rho)\d\rho\Big) \gtrsim \min (\alpha^{-\beta}, \tau^{-1}),
      \end{split}
   \end{equation*}
   one can further estimate the above as
   \begin{equation}
      \label{eq:stimamax1_eq}
      \begin{split}
         r_{i,\tau}(E,t^{\perp}_{i},s) \geq -1 + C\min(|s-s^+ |^{-\beta},\tau^{-1}) + C\min(|s-s^-|^{-\beta} , \tau^{-1})
      \end{split}
   \end{equation}
   where $C$ is a constant depending on the kernel and $\beta = p -d -1$. 
   The above is the ``local'' analogue of \eqref{E:boundg}.

\end{remark}

Recalling the notation in \eqref{eq:fE}
\begin{equation} 
   \label{eq:defFE}
   \begin{split}
      f_{E}(t_{i}^{\perp},t_{i},\zeta_{i}^{\perp},\zeta_{i}) =  |\chi_{E}(t_i+t_i^\perp + \zeta_{i} )  - \chi_{E}(t_i+t_i^\perp)| |\chi_{E}(t_i+t_i^\perp + \zeta^{\perp}_{i}) - \chi_{E}(t_i+t_i^\perp) |,
   \end{split}
\end{equation} 
we can also rewrite the last term on the r.h.s.  of  \eqref{eq:gstr1} as  
\begin{align}
  \label{eq:decomp_double_prod}
      \frac{2}{d}\frac{1}{L^d} \int_{[0,L)^d } \int_{\R^d} f_{E}(t_{i}^{\perp},t_{i},\zeta_{i}^{\perp},\zeta_{i}) K_{\tau}(\zeta)  \d\zeta \dt &=\frac{1}{L^d}  \int_{[0,L)^{d-1}} \sum_{s\in \partial E_{t_{i}^{\perp}}\cap [0,L]} v_{i,\tau}(E,t_{i}^{\perp},s)\dt_{i}^\perp \notag\\
      &+ \frac{1}{L^d}\int_{[0,L)^d} w_{i,\tau}(E,t_i^\perp,t_i) \dt
\end{align}

where
\begin{equation}\label{eq:witau}
      {w}_{i,\tau}(E,t_{i}^{\perp},t_{i}) = \frac{1}{d}\int_{\R^d}  
      f_{E}(t_{i}^{\perp},t_{i},\zeta_{i}^{\perp},\zeta_{i}) K_{\tau}(\zeta)  \d\zeta. 
\end{equation}
and
\begin{equation}\label{eq:vitau}
   v_{i,\tau}(E,t_{i}^{\perp},s) =  \frac{1}{2d}\int_{s^{-}}^{s^{+}} \int_{\R^{d}} f_{E}(t^{\perp}_{i},u,\zeta^{\perp}_{i},\zeta_{i}) K_{\tau}(\zeta) \d\zeta\du
\end{equation}
and $s^+,s^-$ as in \eqref{eq:gstr2}.

Notice that $w_{i,\tau}$ is closely related to $I^i_{\tau,L}$ in \eqref{eq:I}. Indeed, $w_{i,\tau}$ can be seen as a localization or density of $I^i_{\tau,L}$.  More precisely,
\begin{equation*}
   \begin{split}
      I^{i}_{\tau,L}(E) = 2 \int_{[0,L)^d} w_{i,\tau}(E,t_i^\perp,t_i)\dt_i^\perp\dt_i. 
   \end{split}
\end{equation*}

Let us now show the decomposition claimed in \eqref{eq:decomp_double_prod}. By using \eqref{eq:fu}, one  has that
  \begin{equation*}
    \begin{split}
      \sum_{s \in \partial E_{t^\perp_i}}& v_{i,\tau}(E,t^\perp_i ,s ) = 
        \frac{1}{2d} \sum_{s \in \partial E_{t^\perp_i}}\int_{s^{-}}^{s^{+}} \int_{\R^{d}} f_{E}(t^{\perp}_{i},u,\zeta^{\perp}_{i},\zeta_{i}) K_{\tau}(\zeta) \d\zeta\du
      \\ =& 
        \frac{1}{2d}\sum_{s \in \partial E_{t^\perp_i}}\int_{s}^{s^{+}} \int_{\R^{d}} f_{E}(t^{\perp}_{i},u,\zeta^{\perp}_{i},\zeta_{i}) K_{\tau}(\zeta) \d\zeta\du
        +\frac{1}{2d} \sum_{s \in \partial E_{t^\perp_i}}  \int_{s^{-}}^{s} \int_{\R^{d}} f_{E}(t^{\perp}_{i},u,\zeta^{\perp}_{i},\zeta_{i}) K_{\tau}(\zeta) \d\zeta\du \\ 
        & = \frac{1}{d}\int_{0}^{L} \int_{\R^{d}} f_{E}(t^{\perp}_{i},u,\zeta^{\perp}_{i},\zeta_{i}) K_{\tau}(\zeta) \d\zeta\du 
    \end{split}
  \end{equation*}
  Finally, integrating over $t^\perp_i$ one has that

  \begin{equation}
    \label{eq:decomp_double_prod_1}
    \begin{split}
      \frac{1}{L^d}\int_{[0,L)^{d-1}}\sum_{s \in \partial E_{t^\perp_i}}& v_{i,\tau}(E,t^\perp_i ,s ) \dt^\perp_i = \frac{1}{d} \frac{1}{L^d} \int_{[0,L)^d} \int_{\R^d} f_E(t^\perp_i,t_i , \zeta^\perp_i, \zeta_i) \d\zeta \dt.
    \end{split} 
  \end{equation}
  
  To conclude the proof of the decomposition claimed in \eqref{eq:decomp_double_prod}, it is sufficient to combine \eqref{eq:decomp_double_prod_1} and the definition of $w_{i,\tau}$ (see~\eqref{eq:witau}).

The intuition of the role of the  terms $r_{i,\tau}$ and $v_{i,\tau}$ is the following. The term $r_{i,\tau}$ first penalizes oscillations with high frequency  in direction $e_i$, namely sets whose slices in direction $i$ have boundary points at small minimal distance. Indeed, fix $t^{\perp}_{i}$ and consider $s\mapsto \chi_{E_{t^{\perp}_{i}}}(s)$. If this function oscillates with high frequency, there exist $s,s^+\in \partial E_{t^\perp_i}$ such that $|s- s^+|$ is small. Hence by \eqref{eq:stimamax1_eq} the contribution of $r_{i,\tau}$ will be positive and large.

The term $v_{i,\tau}$  penalizes oscillations in direction $e_i$ whenever the neighbourhood of the point $(t_{i}^{\perp}+se_i)$ is close in $L^1$ to a stripe oriented along $e_j$. 
This last statement will be made precise in Lemma~\ref{lemma:stimaContributoVariazionePiccola}.

Finally, for every $Q_{l}(z)$, define
 \begin{equation} 
 \label{eq:fbartau}
   \begin{split}
      \bar{F}_{i,\tau}(E,Q_{l}(z)) &:= \frac{1}{l^d  }\Big[\int_{Q^{\perp}_{l}(z_{i}^{\perp})} \sum_{\substack{s \in \partial E_{t_{i}^{\perp}}\\ t_{i}^{\perp}+se_i\in Q_{l}(z)}} (v_{i,\tau}(E,t_{i}^{\perp},s)+ r_{i,\tau}(E,t_{i}^{\perp},s)) \dt_{i}^{\perp} + \int_{Q_{l}(z)} {w_{i,\tau}(E,t_{i}^{\perp}, t_i) \dt}\Big],\\
      \bar{F}_{\tau}(E,Q_{l}(z)) &:= \sum_{i=1}^d\bar F_{i,\tau}(E,Q_{l}(z)).
   \end{split}
\end{equation} 

The above consists in the ``local contribution'' to the energy in a cube $Q_{l}(z)$ mentioned in the outline. 
More precisely, we will write the \rhs of \eqref{eq:gstr1} in terms of $\bar{F}_{\tau}(E,Q_{l}(\cdot))$ via an averaging process.
In order to do this we will need the following lemma. 

\begin{lemma} 
   \label{lemma:bohx1}
   \label{lemma:average}
   Let $\mu$ be a $[0,L)^d$-periodic locally finite measure, namely a measure invariant under translations in  $L\Z^d$.  Then  one has that
   \begin{equation} 
      \label{eq:gstr18}
      \begin{split}
         \int_{[0,L)^d}\d\mu(x) = \frac{1}{l^d} \int_{[0,L)^d} \int_{Q_{l}(z)} \d \mu(x) \dz.
      \end{split}
   \end{equation} 

\end{lemma} 

\begin{proof} 

   The proof of~\eqref{eq:gstr18} is done by changing order of integration (namely Fubini): first integrating in $z$. 

   Indeed,

   \begin{equation*} 
      \begin{split}
         \int_{[0,L)^d}  \int_{Q_{l}(z)} \d\mu(x)\dz &= \int_{[0,L)^d} \int_{\R^d} \chi_{Q_{l}(z)}(x)\d\mu(x) \dz = \int_{[0,L)^d} \int_{\R^d}\chi_{Q_{l}(x)} (z )  \dz \d\mu(x) \\ &= l^{d} \int_{[0,L)^d}\d\mu(x).
      \end{split}
   \end{equation*} 

\end{proof} 

   By Lemma~\ref{lemma:bohx1}, we have that the \rhs  of \eqref{eq:gstr1} is equal to
   \begin{equation}
      \label{eq:gstr15}
      \frac{1}{L^d}\int_{[0,L)^d}\bar{F}_{\tau}(E,Q_{l}(z))\dz. 
   \end{equation}
   Indeed, since $E$ is $[0,L)^d$-periodic we can see that the \rhs of \eqref{eq:gstr1} as an integration with respect to a $[0,L)^d$-periodic measure. 
   This implies that 
   \begin{equation}
      \label{eq:gstr14}
      \begin{split}
         \Fcal_{\tau,L}(E) \geq \frac{1}{L^d} \int_{[0,L)^d}  
         \bar{F}_{\tau}(E,Q_{l}(z)) \dz. 
      \end{split}
   \end{equation}
   Given that, in the above inequality, equality holds for stripes, if we show that the minimizers of \eqref{eq:gstr15} are periodic optimal stripes, then the same claim holds for $\Fcal_{\tau,L}$.

We will say that a set which is a union of stripes, $S$, is oriented along  the direction $e_i$, if $S$ is invariant with respect to every translation orthogonal to $e_i$. 

In the next definition we define a quantity which measures the distance of a set from being a union of stripes. 
Such a quantity being small means for us to be ($L^1$-)``close'' to stripes in a given cube.

\begin{definition}
   \label{def:defDEta}
   For every $\eta$ we denote by $\Acal^{i}_{\eta}$ the family of all sets $F$ such that  
   \begin{enumerate}[(i)]
      \item they are union of stripes oriented along the direction $e_i$ 
      \item their connected components of the boundary are distant at least $\eta$. 
   \end{enumerate}
   We denote by 
   \begin{equation} 
      \label{eq:defDEta}
      \begin{split}
         D^{i}_{\eta}(E,Q) := \inf\Big\{ \frac{1}{\vol(Q)} \int_{Q} |\chi_{E} -\chi_{F}|:\ F\in \Acal^{i}_{\eta} \Big\} \quad\text{and}\quad D_{\eta}(E,Q) = \inf_{i} D^{i}_{\eta}(E,Q).
      \end{split}
   \end{equation} 
   Finally, we let $\mathcal A_\eta:=\cup_{i}\mathcal A^i_{\eta}$.
\end{definition}

Let $S$ be union of stripes, namely $S =\hat{S}\times \R^{d-1}$, with $\hat{S} = \bigcup_{i\in\Z} (\alpha_i,\beta_{i})$.  Then condition  (ii) in Definition~\ref{def:defDEta}, says that $\inf_{i,j} |\alpha_i -\beta_j| > \eta$. This corresponds to the minimal distance between  the connected components of the boundaries of the stripes in $S$.

In the following  remark, we first notice that the local distance \eqref{eq:defDEta} from a family of stripes in a certain direction is a Lipschitz function w.r.t. the centre of the cube we consider. In particular, what we need in the proof of Theorem~\ref{T:1.3} is that, if at a point we are far from being stripes, then in a neighbourhood of it we also have approximately the same distance (``thickness'' of the set $A_0$ mentioned in the outline and that will be defined in \eqref{a0}).

Moreover, in point (ii) of the next remark we notice that, if we are sufficiently close to stripes in different directions, then either $|E\cap Q_l(z) | \ll l^d$ or $|E^c \cap Q_l(z)| \ll l^d $ (property of the set $A_{-1}$ mentioned in the Outline and that will be defined in \eqref{a1}).

\begin{remark}
   \label{rmk:lip} \ 
   \begin{enumerate}[(i)]
      \item Let $E, F \subset \R^d$.  Then the map $z\mapsto D_{\eta}(E,Q_{l}(z))$ is Lipschitz, with Lipschitz constant $C_d/l$, where $C_d$ is a constant depending only on the dimension $d$. 
         In order to see this for fixed $F$ consider the map
         \begin{equation*}
            \begin{split}
               T_F: z\mapsto \frac{1}{l^d}\int_{Q_{l}(z)} |\chi_{E}(x) - \chi_{F}(x) | \dx. 
            \end{split}
         \end{equation*}
         Then 
         \begin{equation*}
            \begin{split}
               T_F(z') = \frac{1}{l^d}\int_{Q_l(z')} |\chi_{E}(x)- \chi_{F}(x) |\dx &  \leq \frac{1}{l^d}\int_{Q_{l}(z)}  |\chi_{E}(x) - \chi_F(x) |\dx + \frac{1}{l^d} |Q_{l}(z)\Delta Q_{l}(z')|
               \\ &\leq T_{F}(z) +   \frac{C_{d}}l|z-z'|,
            \end{split}
         \end{equation*}
         where $C_{d}$ is a constant depending only on the dimension $d$ and $Q_l(z)\triangle Q_l(z')=(Q_l(z)\setminus Q_l(z'))\cup(Q_l(z')\setminus Q_l(z))$. 
         Finally given that $D_{\eta}^i(E,Q_{l}(\cdot))$ and $D_{\eta}(E,Q_{l}(\cdot))$  are the infima of $T_{F}(\cdot)$ for $F\in \Acal^i_\eta$ and $\Acal_\eta$ respectively,  we have that $D_{\eta}^i(E,Q_{l}(\cdot))$  and $D_{\eta}(E,Q_l(\cdot))$  are Lipschitz with Lipschitz constant $C_d/l$. 

         In particular, whenever $D_{\eta}(E,Q_{l}(z)) > \alpha$ and $D_{\eta}(E,Q_{l}(z')) < \beta$,  then $|z - z'|> l(\alpha - \beta)/C_{d}$.

         \item
            For every $\varepsilon$ there exists ${\delta}_0= \delta_0(\varepsilon)$ such that  for every $\delta \leq \delta_0 $ whenever $D^{j}_{\eta}(E,Q_{l}(z))\leq \delta$ and $D^{i}_{\eta}(E,Q_{l}(z))\leq \delta$ with $i\neq j$ for some $\eta>0$,  it holds 
            \begin{equation}
               \label{eq:gsmstr2}
               \begin{split}
                  \min\big(|Q_l(z)\setminus E|, |E \cap Q_l(z)| \big) \leq\varepsilon. 
               \end{split}
            \end{equation}
            The above claim follows easily by contradiction.  Indeed, suppose that  there exist   $\varepsilon> 0$, a sequence of sets in $\{ E_n \}$,  and  sequences $\delta_n\downarrow 0$ and $\eta_n>0$ such that  
            \begin{equation}
               \label{eq:gsmstr6}
               \begin{split}
                  D^{j}_{\eta_n}(E_n,Q_{l}(z))\leq \delta_n\qquad \text{and} \qquad D^{i}_{\eta_n}(E_n,Q_{l}(z))\leq \delta_n \qquad \text{ (with $i\neq j$) }
               \end{split}
            \end{equation}
            
            and  such that $\min\big(|Q_l(z)\setminus E|, |E \cap Q_l(z)| \big) > \varepsilon$. 
           W.l.o.g. we assume that $z=0$.
            From \eqref{eq:gsmstr6}, we have that there exist two sets $S^i_n$  and $S^j_n$ such that the distance of $E_n$ is $\delta_n$-close in $L^1$ to $S^i_n$ and $S^j_n$. Thus
            \begin{equation}
               \label{eq:gsmstr3}
               \begin{split}
                  \int_{Q_l(0)} |\chi_{S^i_n}(x)  -\chi_{S^i_n}(x) | \dx \leq 2\delta_n. 
               \end{split}
            \end{equation}
            It is not difficult to see that \eqref{eq:gsmstr3} holds if and only if both $S^i_n \cap Q_l(0)$ and $S^j_n\cap Q_l(0)$ are $L^1$-close (depending on $\delta_n$) either to $Q_l(0)$ or to  $\emptyset$. 
            Indeed, without lost of generality we can assume that $d=2$, thus $S^1_n = F^1_n \times \R $ and $S^2_n = \R \times F^2_n$ where $F^1_n,F^2_n\subset \R$. 

            Rewriting \eqref{eq:gsmstr3} and using Fubini, we have that
            \begin{equation*}
               \begin{split}
                  \int_{Q_l(0)} |\chi_{S^i_n}(x)  -\chi_{S^i_n}(x) |\dx    = 
                  \int_{-l/2}^{l/2}\int_{-l/2}^{l/2} | \chi_{F^1_n}(x_1) - \chi_{F^2_n}(x_2) |    \dx_1 \dx_2
                  \leq 2\delta_n. 
               \end{split}
            \end{equation*}
               Noticing  that $\chi_{F^2_n}(x_2)$  does not depend on $x_1$ and that $\chi_{F^2_n}(x_2)\in \{ 0,1\}$, we immediately deduce that $\chi_{F_1}\cap (-l/2,l/2)$  is close (depending on $\delta_n$) in $L^1((-l/2,l/2))$ to either $( -l/2, l/2 )$  or to $\emptyset$, which in turn implies that $E_n$ is close in $L^1(Q_l(0))$ to $Q_l(0)$ or to $\emptyset$.  
            Notice that the above reasoning does not depend on $\eta$, since the only thing used is that $\chi_{S^i_\varepsilon}$ and $\chi_{S^i_\varepsilon}$ are invariant with respect  to two different directions and take values in $\{ 0,1\}$. 

   \end{enumerate}
\end{remark}

The following lemma is a technical lemma that is used in  the Local Rigidity Lemma (Lemma~\ref{lemma:local_rigidity}).

In particular, it says that if a family of sets $E_\tau$ of locally finite perimeter in $\R$ converges in $L^1$ to a set $E_0$ of locally finite perimeter and the local contributions given by $r_\tau(E_\tau,s)$ defined in \eqref{eq:rtau1D} (which for slices $E_{\tau,t_i^\perp}$ of $E_\tau\subset\R^d$ coincides with $r_{i,\tau}(E,t_i^\perp,s)$ in \eqref{eq:ritau}) are uniformly bounded, then 
   \eqref{eq:gstr5} holds. This is one of the preliminary steps used in Lemma \ref{lemma:local_rigidity} to show that $E_0$ is a union of stripes.  

\begin{lemma}
   \label{lemma:technicalBeforeLocalRigidity}
   Let $E_{0}, \{E_{\tau}\}\subset \R$  be a family of sets of locally finite perimeter and $I\subset \R$ be an open interval.   Moreover, assume that $E_{\tau}\to E_0$ in $L^1(I)$.  
   If we denote by $\{k^{0}_{1},\ldots,k^{0}_{m}\} = \partial E_{0}\cap I $, then
   \begin{equation}
      \label{eq:gstr5}
      \liminf_{\tau\downarrow 0}\sum_{\substack{s\in \partial E_{\tau}\\ s\in I}}r_{\tau}(E_{\tau},s) \geq \sum_{i=1}^{m}(-1 + C|k^0_{i} - k^0_{i+1} |^{-\beta}),
   \end{equation}
   where $r_{\tau}$ is defined in \eqref{eq:rtau1D}.
\end{lemma}

\begin{proof}

   Let us denote by $\{k^\tau_{1},\ldots,k^\tau_{m^{\tau}}\} = \partial E_{\tau}\cap I$. 
   We will also denote by
   \begin{equation*}
      k^{\tau}_0 = \sup\{ s\in \partial E_\tau: s < k^\tau_1\} \qquad\text{and}\qquad 
      k^\tau_{m^\tau+1} = \inf\{ s\in \partial E_\tau: s > k^\tau_{m^\tau}\} .
   \end{equation*}
   Denote by $A$ the \rhs of \eqref{eq:gstr5}. 
   From \eqref{eq:stimamax1_eq}, one has that $r_{\tau}(E_\tau,k^{\tau}_{i}) \geq -1  + C \max(|k^{\tau}_{i} - k^{\tau}_{i+1} |^{-\beta},\tau^{-1})$. 
   Thus, there exist $\eta$ and $\bar{\tau}> 0$ such that for every $\tau <\bar{\tau}$, whenever 
   \begin{equation*}
      \begin{split}
         \min_{i\in \{0,\ldots,{m^\tau}\}}|k^{\tau}_{i+1}- k^{\tau}_{i}| < \eta
      \end{split}
   \end{equation*}
   then
   \begin{equation*}
       \sum_{\substack{s\in \partial E_{\tau}\\ s\in I}}r_{\tau}(E_{\tau},s) \gtrsim A. 
   \end{equation*}
   Hence, assume there exists a subsequence $\tau_{k}$ such that $|k^{\tau_k}_{i+1}- k^{\tau_{k}}_{i}| > \eta$ for all $i\leq m^{\tau_k}$. 
   Up to  relabeling, let us assume that it holds true  for the whole sequence of $E_\tau$.  

   Since $\min_{i} | k^{\tau}_{i+1} - k^\tau_{i} | > \eta$ the convergence $E_{\tau}\to E_{0}$ in  $L^1(I)$ can be upgraded to the convergence of the boundaries, namely one has that there exists a $\bar{\tau}$ such that for $\tau<\bar{\tau}$, it holds $\#(\partial E_{\tau}\cap I) = \#(\partial E_{0}\cap I)$ and $k^{\tau}_{i} \to k^0_{i}$. 

   Then because of the convergence of the boundaries, we have that 
   \begin{equation}
      \label{eq:gstr13}
      \begin{split}
         \liminf_{\tau\downarrow 0}\sum_{\substack{s\in \partial E_{\tau}\\ s\in I}}r_{\tau}(E_{\tau},s) &\geq \liminf_{\tau\downarrow 0}\sum_{j=1}^{m} \big( - 1 + C\max(|k^{\tau}_{i} - k^\tau_{i+1}|^{-\beta},\tau^{-1}) \big)
         \\ & \geq\sum_{j=1}^{m} \big( - 1 + C|k^{0}_{i} - k^0_{i+1}|^{-\beta}\big).
      \end{split}
   \end{equation}

\end{proof}

The following lemma contains the local version of {Theorem~\ref{T:gammaconvNew}}, namely the rigidity estimate mentioned in the outline. 

Its content can be summarized as follows.  
Given a sequence of sets $E_\tau\subset\R^d$ of bounded local energy, by Remark \ref{rmk:stimax1} their boundary points on the slices are not too close and then they converge to a set of locally finite perimeter $E_0$. 
Then, using the lower semicontinuity result of Lemma \ref{lemma:technicalBeforeLocalRigidity} and the monotonicity in $\tau$ of the kernel, one gets as $\tau\to0$ a bound similar to $\sum_{i=1}^d\mathcal G^i_{0,L}(E)+I_{0,L}(E)<+\infty$, but with $\mathcal G^i_{0,L}$ and $I_{0,L}$ substituted by their local counterparts ($r_{i,\tau}$ and $w_{i,\tau}$). 
Then, applying Proposition \ref{prop:rigidity}, which has been already proved without using any periodicity assumption on $E$ (see~Remark~\ref{rmk:PeriodicitaNonServe}), one has that the set $E_0$ has to be a union of stripes. 
Therefore, for $\tau>0$ sufficiently small but depending only on $l$, the sets $E_\tau$ will be close to $E_0$ in the sense of Definition \ref{def:defDEta}.

\begin{lemma}[Local Rigidity] 
   \label{lemma:Stima2}
   \label{lemma:local_rigidity}
    For every $M > 1,l,\delta > 0$, there exist $\bar{\tau},\bar{\eta} >0$ 
     such that whenever $\tau< \bar{\tau}$  and $\bar F_{\tau}(E,Q_{l}(z)) < M$ for some $z\in [0,L)^d$ and $E\subset\R^d$ $[0,L)^d$-periodic, with $L>l$, then it holds $D_{\eta}(E,Q_{l}(z))\leq\delta$ for every $\eta < \bar{\eta}$. Moreover $\bar{\eta}$ can be chosen independent on $\delta$.  Notice that $\bar{\tau}$ and $\bar{\eta}$ are independent of $L$.
\end{lemma} 

\begin{proof} 

   The proof will follow by contradiction. 
   Assume that the claim is false. This implies that there exists $M > 1, l$, $\delta>0$ and sequences $\{ \tau_k\}$, $\{ \eta_k\}$, $\{L_k\}$, $\{z_k\}$, $\{E_{\tau_k}\}$ such that:
   \begin{enumerate}[(i)]
   \item one has that   $\tau_k\downarrow 0 $, $L_k > l$, $\eta_k \downarrow 0$, $z_k\in[0,L_k)^d$; 
    \item the family of sets $E_{\tau_k}$ is $[0,L_k)^d$-periodic
    \item one has that $D_{\eta_k}(E_{\tau_k},Q_{l}(z_k)) > \delta$ and $\bar{F}_{\tau_k}(E_{\tau_k},Q_{l}(z_k)) < M$. 
   \end{enumerate}

   Given that $\eta \mapsto D_{\eta}(E,Q_l(z)) $ is monotone increasing, we can fix $\bar{\eta}$  sufficiently small instead of $\eta_{k}$ with $D_{\bar\eta}(E_{\tau_k},Q_l(z_k)) >  \delta$. 
   In particular, $\bar{\eta}$ will be chosen at the end of the proof depending only on $M,l$. 
   
   W.l.o.g. (taking e.g. $E_{\tau_k}-z_k$ instead of $E_{\tau_k}$) we can assume there exists $z\in\R^d$ such that $z_k=z$ for all $k\in\N$. 

   Because of Remark~\ref{rmk:stimax1}, one has that $\sup_{k}\per_1(E_{\tau_k},Q_{l}(z)) < +\infty$. 
   Thus up to subsequences there exists $E_{0}$ such that $E_{\tau_{k}} \to E_{0}$ in $L^1(Q_l(z))$ with $D_{\bar{\eta}}(E_0,Q_{l}(z))> \delta$.

   In order to keep the notation simpler, we will write $\tau \to 0$ instead of $\tau_k \to 0$ and  $E_{\tau}\to E_0$ instead of $E_{\tau_{k}}\to E_0 $.

   In this proof we will denote by $J_i$ the interval $(z_{i}-l/2,z_{i}+l/2)$. 

   By Lebesgue's theorem, there exists a subsequence of $\tau$ such that for almost every $t_{i}^{\perp}\in Q_{l}^{\perp}(z_{i}^{\perp})$ one has that $E_{\tau,t_{i}^{\perp}}\cap J_i$ converges to $E_{0,t_{i}^{\perp}}\cap J_i $ in $L^1(Q_l(z))$.

      By using \eqref{eq:fbartau} and the fact that $v_{i,\tau}\geq0$, we have that
      \begin{equation}
         \label{eq:gstr7}
          M \geq  \bar{F}_{\tau}(E_\tau,Q_l(z)) \geq
         \frac{1}{l^d }\sum_{i =1}^{d}\int_{Q_{l}^{\perp}(z_{i}^{\perp})} \sum_{\substack{s\in \partial E_{\tau, t^\perp_i}\\ s\in J_i}} 
         r_{i,\tau}(E_{\tau},t_{i}^{\perp},s) \dt^\perp_{i} + \int_{Q_{l}(z)} w_{i,\tau}(E_{\tau},t_{i}^{\perp},t_{i})\dt_{i}^{\perp}\dt_{i}.
      \end{equation}
      By the Fatou lemma we have that 
        \begin{equation*}
          \begin{split}
            l^d M \geq  \liminf_{\tau\downarrow 0}\sum_{i =1}^{d}\int_{Q_{l}^{\perp}(z_{i}^{\perp})} \sum_{\substack{s\in \partial E_{\tau,t^\perp_i}\\ s\in J_i}}&r_{i,\tau}(E_{\tau},t_{i}^\perp,s) \dt_{i}^{\perp}         \geq 
            \sum_{i =1}^{d}\int_{Q_{l}^{\perp}(z_{i}^{\perp})} \liminf_{\tau\downarrow 0}\sum_{\substack{s\in \partial E_{\tau,t^\perp_i}\\ s\in J_i}} 
            r_{i,\tau}(E_{\tau},t_{i}^{\perp},s) \dt^{\perp}_{i}       \\ 
            &\geq 
            \sum_{i =1}^{d}\int_{Q_{l}^{\perp}(z_{i}^{\perp})} \sum_{\substack{s\in \partial E_{0, t^\perp_i}\\ s\in J_i}} 
            \big(-1  + C(s^+-s)^{-\beta} + C(s-s^-)^{-\beta}\big) \dt_{i}^\perp,
         \end{split}
      \end{equation*}
      where in the last inequality we have used Lemma \ref{lemma:technicalBeforeLocalRigidity} applied to $r_{i,\tau}(E_\tau,t_i^\perp,s)=r(E_{\tau,t_i^\perp},s)$. 

      The same type of inequality holds also for the last term  in \eqref{eq:gstr7}, namely

      \begin{equation}
         \label{eq:gstr8}
         \begin{split}
            \liminf_{\tau\downarrow 0}&  \int_{Q_{l}(z)} w_{i,\tau}(E_{\tau},t_{i}^{\perp},t_{i})\dt_{i}^{\perp}\dt_{i} 
            \\   &\geq \liminf_{\tau\downarrow 0} \frac{1}{d}\int_{Q_{l}(z)} \int_{Q_{l}(z)} f_{E_\tau}(t^\perp_i, t_{i},t'^{\perp}_{i} - t^\perp_i,t'_{i}- t_i)K_{\tau}(t - t')\dt \dt'\\  
            \\   &\geq \frac{1}{d}\int_{Q_{l}(z)} \int_{Q_{l}(z)} f_{E_0}(t^\perp_i, t_{i},t'^{\perp}_{i} - t^\perp_i,t'_{i}- t_i) K_{0}(t - t')\dt \dt'\\  
         \end{split}
      \end{equation}
      Indeed, in order to prove \eqref{eq:gstr8} we fix $\tau'>0$ and by using initially $E_{\tau}\to E_{0}$ in $L^1(Q_l(z))$ and afterwards the   monotonicity of $\tau \mapsto \widehat{K}_{\tau}(s)$ we have that

   \begin{equation*}
      \begin{split}
         \liminf_{\tau\downarrow 0} \int_{Q_{l}(z)} w_{i,\tau}(E_{\tau},t_{i}^{\perp},t_{i}) \dt_{i}^{\perp}\dt_{i} \geq \sup_{\tau'} 
         \liminf_{\tau\downarrow 0}\int_{Q_{l}(z)} w_{i,\tau'}(E_{\tau},t_{i}^{\perp},t_{i}) \dt_{i}^{\perp}\dt_{i}
         \\ \geq  \sup_{\tau'}\frac{1}{d}\int_{Q_{l}(z)}\int_{Q_{l}(z)} 
         f_{E_{0}} (t^\perp_i, t_i, t'^\perp_i - t'^\perp_{i},t_i - t'_i) K_{\tau'}(t-t')\dt \dt'
         \\ \geq \frac{1}{d}\int_{Q_{l}(z)}\int_{Q_{l}(z)} 
         f_{E_{0}} (t^\perp_i, t_i, t'^\perp_i - t'^\perp_{i},t_i - t'_i) K_{0}(t-t')\dt \dt'.
      \end{split}
   \end{equation*}

   To summarize, we have shown that 
      \begin{align}
          & \sum_{i=1}^{d}\frac{1}{d}\int_{Q_{l}(z)} \int_{Q_{l}(z)}f_{E_{0}} (t^\perp_i, t_i, t'^\perp_i - t'^\perp_{i},t_i - t'_i) K_{0}(t-t')\dt \dt'  \notag
          \\  &+    \sum_{i =1}^{d}\int_{Q_{l}^{\perp}(z_{i}^{\perp})} \sum_{\substack{s\in \partial E_{0,t^\perp_i}\\ s\in J_i}} ((s^+-s)^{-\beta} +(s-s^-)^{-\beta} )          \dt^\perp_{i} \lesssim l^d M + \per_1(E_0,Q_l(z))  
      \end{align}

      Similarly to \eqref{eq:gr1}, one can show that 
      \begin{equation*}
         \begin{split}
            \per_1(E_\tau ,Q_l(z)) \lesssim l^d \max(1,\bar{F}_\tau(E_\tau,Q_l(z))) \lesssim l^d \max(1,M) \lesssim l^d M,
         \end{split}
      \end{equation*}
      thus from the lower semicontinuity of the perimeter we have that
      \begin{equation*}
         \begin{split}
            \per_1(E_0,Q_l(z))) \leq \liminf_{\tau\downarrow 0}\per_1(E_\tau,Q_l(z))) \lesssim l^d M.
         \end{split}
      \end{equation*}
      In particular, 
      \begin{align}
         \label{eq:gstr10}
         &   \sum_{i =1}^{d}\int_{Q_{l}^{\perp}(z_{i}^{\perp})} \sum_{\substack{s\in \partial E_{0,t^\perp_i}\\ s\in J_i}} ((s^+-s)^{-\beta} +(s-s^-)^{-\beta} )          \dt^\perp_{i} \lesssim  l^d M 
            \\  &\sum_{i=1}^{d}\int_{Q_{l}(z)} \int_{Q_{l}(z)}f_{E_{0}} (t^\perp_i, t_i, t'^\perp_i - t'^\perp_{i},t_i - t'_i) K_{0}(t-t')\dt \dt' \lesssim l^d M \label{eq:gstr38}.
      \end{align}


   At this point, the same reasoning as in   Proposition~\ref{prop:rigidity} can be applied in order to obtain that \eqref{eq:gstr10} and \eqref{eq:gstr38}  hold  if and only if $E_{0}\cap Q_{l}(z)$ is a union of stripes. It is indeed sufficient to substitute $[0,L)^d$ with $Q_l(z)$ in the proof.

   Moreover, since the \lhs  of \eqref{eq:gstr10} explodes for stripes with minimal width  tending to zero, one has that there exists $\bar{\eta} = \bar{\eta}(M,l)$ such that $D_{\bar{\eta}}(E_0,Q_l(z))  = 0$.  This contradicts that $D_{\bar{\eta}}(E_{0}, Q_l(z)) >\delta$,  which was assumed at the beginning of the proof. 
   \end{proof} 

The following is a technical lemma needed in the proof  of Lemma~\ref{lemma:stimaLinea}. It says that given a set $E\subset \R$, and $I\subset \R$ an interval,  then the one-dimensional contribution  to the energy, namely $\sum_{s\in \partial E\cap I}r_{\tau}(E,s)$, is comparable to the periodic case up to a constant $C_0$ depending only on the dimension.

More precisely, $C_0$ depends on $\eta_0$, which is the minimal distance between boundary points of $E$ so that $r_\tau(E,s)<0$, and it is what one can lose by extending periodically the set $E$ outside the interval $I$. For periodic sets then, by the reflection positivity Theorem~\ref{T:1d}, the energy contribution in \eqref{eq:gstr40} is bigger than or equal to the contribution of periodic stripes of width $h^*_\tau$, namely $C^*_\tau$ times the length of the interval $I$. 
\begin{lemma}
	\label{lemma:1D-optimization}
   \label{lemma:1D-opitmization}
   There exists $C_0 > 0$ such that the following holds.
	Let $E\subset \R$  be a set of locally finite perimeter and $I\subset \R$ be an open interval. 
	Let $s^-, s$ and $s^+$ be three consecutive points on the boundary of $E$ and $r_\tau(E,s)$ defined as in \eqref{eq:rtau1D}.
	Then for all $\tau< \tau_0$, where $\tau_0$ is given in Remark~\ref{rmk:stimax1}, it holds
	\begin{equation}
      \label{eq:gstr40}
	\sum_{\substack{s \in \partial E\\ s \in I}} r_{\tau}(E,s) \geq C^*_{\tau} |I| - C_0.
	\end{equation}
\end{lemma}

\begin{proof}
	Let us denote by $k_1< \ldots< k_m $ the  points of $\partial E \cap I$, and 
	
	\begin{equation*}
	k_0 = \sup\{ s\in \partial E: s < k_1\} \qquad\text{and}\qquad 
	k_{m+1} = \inf\{ s\in \partial E: s > k_m\} 
	\end{equation*}
	
	W.l.o.g. we may assume that  $r_{\tau}(E,k_1) < 0$ and that $r_{\tau}(E,k_m) < 0$. 

   Indeed, if that is not the case one can consider $I' \subset I$  such that $r_{\tau}(E,k'_1) < 0$ and that $r_{\tau}(E,k'_{m'}) < 0$, where  $k'_1,\cdots,k'_{m'}$ are the points of $\partial E \cap I'$.
   Thus if estimate \eqref{eq:gstr40} holds for $I'$ then it holds also for $I$ by the following chain of inequalities
	\begin{equation*}
	\sum_{\substack{s \in \partial E\\ s \in I}} r_{\tau}(E,s) \geq \sum_{\substack{s \in \partial E\\ s \in I'}} r_{\tau}(E,s) \geq C^*_{\tau} |I'| - C_0 \geq C^*_\tau |I| -C_0,
	\end{equation*}
   where in the last inequality we have used that $C^*_\tau < 0$.

	Because of Remark~\ref{rmk:stimax1}, the fact that $r_{\tau}(E,k_1) < 0$ and $r_{\tau}(E,k_m) < 0$ implies that there exists $\eta_0>0$ (for all  $\tau \leq \tau_0$) such that 
	\begin{equation*}
	\min(|k_1 - k_0 |, |k_2 - k_1 |, |k_{m-1} - k_m |,|k_{m+1} - k_m |) > \eta_0. 
	\end{equation*}
	
	We claim that 
	\begin{equation}
	\label{eq:gstr4}
	\sum_{i =1}^{ m } r_{\tau}(E,k_i) \geq \sum_{i =1}^{ m } r_{\tau}(E',k_i)  - \bar C_0 
	\end{equation}
	where $E'$ is obtained by extending periodically $E$ with the pattern contained in $E \cap (k_1,k_m)$ and $\bar C_0  = \bar C_0(\eta_0) > 0$. 
   The construction of $E'$ can be done as follows: if $m$ is odd we repeat periodically $E\cap (k_1,k_m)$, and if $m$ is even we repeat periodically $(k_1-\eta_{0}, k_m)$. 
   
   Thus we have constructed a set $E'$ which is periodic of period $k_m-k_1$ or $k_m- k_1 + \eta_0$. Therefore

   \begin{equation}
     \label{eq:ggstr1}
     \sum_{i=1}^m r_\tau(E',k_i) \geq C^{*}_{\tau} (k_m - k_1 +\eta_0) \geq  C^*_\tau |I| - \tilde C _0,
   \end{equation}
   where $ \tilde C_0 = \tilde C_0(\eta_0)$. 
   Inequality~\eqref{eq:ggstr1} follows by reflection positivity (Theorem~\ref{T:1d}).
   Indeed, reflection positivity implies that optimal $[0,L)$-periodic set $E_L$ must be a union of periodic stripes and $C^*_\tau$ was the minimal energy density by further optimizing in $L$. 

   Inequality~\eqref{eq:ggstr1} combined with \eqref{eq:gstr4} yields \eqref{eq:gstr40}.

	To show \eqref{eq:gstr4}, notice that the symmetric difference between $E$ and $E'$ satisfies
	\begin{equation*}
	E \Delta E'  \subset (- \infty, k_1 - \eta_0) \cup (k_m+\eta_0, + \infty),
	\end{equation*}
   where $\eta_{0}$ is the constant defined in Remark~\ref{rmk:stimax1}.
   To obtain \eqref{eq:gstr4},  we need to estimate $|\sum_{i=1}^m r_\tau(E,k_i) - \sum_{i=1}^m r_\tau(E',k_i) |$. 
   Let 
   \begin{equation*}
      \begin{split}
          \sum_{i=1}^m r_\tau(E,k_i) - \sum_{i=1}^m r_\tau(E',k_i) =
          A + B,
      \end{split}
   \end{equation*}
   where
   \begin{equation*}
      \begin{split}
        A =  \sum_{i=0}^{m-1}&  \int_{k_{i}}^{k_{i+1}}\int_{0}^{+\infty} (s - |\chi_{E}(s+u) - \chi_{E}(u)|)\widehat{K}_{\tau}(s) \ds \du
         \\  & - \sum_{i=0}^{m-1} \int_{k_{i}}^{k_{i+1}}\int_{0}^{+\infty} (s - |\chi_{E'}(s+u) - \chi_{E'}(u)|)\widehat{K}_\tau(s) \ds \du\\ 
           B =  \sum_{i=1}^{m} & \int_{k_{i}}^{k_{i+1}}\int_{-\infty}^{0} (s - |\chi_{E}(s+u) - \chi_{E}(u)|)\widehat{K}_\tau(s) \ds \du
         \\ &- \sum_{i=1}^{m} \int_{k_{i}}^{k_{i+1}}\int_{-\infty}^{0} (s - |\chi_{E'}(s+u) - \chi_{E'}(u)|)\widehat{K}_\tau (s)\ds \du. 
      \end{split}
   \end{equation*}
   Thus by using the integrability of $\widehat K$ (see~\eqref{eq:hatKInequality}), we have that 
   \begin{equation*}
      \begin{split}
         |A| \leq \int_{k_{0}} ^{k_{m}} \int_{0}^{+\infty} \chi_{E \Delta E'}(u +s  ) \widehat{K}_{\tau}(s) \ds \du \leq\int_{k_0}^{k_m}\int_{k_{m} + \eta_0}^{\infty} \widehat{K}_{\tau}(u-v) \dv\du \leq  \frac{C_{0}}{2},
      \end{split}
   \end{equation*}
   where $C_{0}$ is a constant depending only on $\eta_0$.  Similarly, $|B| \leq C_{0}/2$

	Thus we have that
	\begin{equation*}
	\Big|\sum_{i=1}^m r_\tau(E,k_i) - \sum_{i=1}^m r_\tau(E',k_i)\Big| \leq \int_{k_0}^{k_m}\int_{k_{m} + \eta_0}^{\infty} \widehat{K}_{\tau}(u-v) \du\dv
	+ \int_{k_1}^{k_{m+1}}\int_{-\infty}^{k_{1} - \eta_0} \widehat{K}_{\tau}(u-v) \dv\du.
	\end{equation*}

		Since for every periodic set we have that $C^{*}_{\tau}$ is the infimum of all the energy densities for periodic sets (of any period), we have the desired result. 
\end{proof}

The next lemma is the local analog of the Stability Lemma~\ref{lemma:new_lemma_pre1}.  Informally,  it shows that if  we are in a cube where the set $E\subset\R^d$ is close to a set $E'$ which is a union of stripes in direction $e_i$ (according to Definition \ref{def:defDEta}), then it is not convenient to oscillate in direction $e_j$ with $j\neq i$ (namely, on the slices in direction $e_i$ to have points in $\partial E_{t_i^\perp}$).  
We show that in such a case either the local contribution given by $r_{i,\tau}$ or the one given by $v_{i,\tau}$ are large. 
We recall that the first term penalizes alternation between regions in $E_{t_i^\perp}$ and regions in $E_{t_i^\perp}^c$ with high frequency and the second penalizes oscillations in direction $i$ whenever the neighbourhood of the point $(t_{i}^{\perp},t_i)$ is close in $L^1$ to a stripe in a perpendicular direction $j$. This second fact (about the role of $v_{i,\tau}$) is justified in an analogous way to \eqref{eq:defBarEps}-\eqref{eq:6.13}, with the only difference that there the interaction term $I_{\tau,L}$ was global (on $[0,L)^d$) and here we consider the localized version.

\begin{lemma}[Local Stability]
   \label{lemma:stimaContributoVariazionePiccola}
         Let  $(t^{\perp}_{i}+se_i)\in (\partial E) \cap [0,l)^d$ and $\eta_{0}$, $\tau_0$ as in Remark~\ref{rmk:stimax1}. Then there exist $\tilde{\tau},\tilde\varepsilon > 0$ (independent of $l$) such that for every $\tau < \tilde{\tau}$ and $\varepsilon < \tilde{\varepsilon}$ the following holds: assume that 
         \begin{enumerate}[(a)]
            \item $\min(|s-l|, |s|)> \eta_0$
            \item $D^{j}_{\eta}(E,[0,l)^d)\leq\frac {\varepsilon^d} {16 l^d}$ for some $\eta> 0$ and  with $j\neq i$ (this condition expresses that $E\cap [0,l)^d$ is close to stripes oriented along a direction orthogonal to $e_i$)
         \end{enumerate}
         Then $r_{i,\tau}(E,t^{\perp}_{i},s) + v_{i,\tau}(E,t^{\perp}_{i},s) \geq 0$. 

\end{lemma}
\begin{proof}

   Let $s^-,s,s^+$ be three consecutive points for $\partial E_{t_{i}^{\perp}}$. 
   By Remark~\ref{rmk:stimax1},  there exists $\eta_{0}$, $\tau_0>0$  such that if $\tau<\tau_0$
   \begin{equation*} 
      \begin{split}
         \min(|s-  s^- |, |s^+ -s |) <\eta_{0} \quad\text{then}\quad r_{i,\tau}(E,t_{i}^{\perp},s) > 0. 
      \end{split}
   \end{equation*} 

   Thus without loss of generality we may assume that $\min(|s - s^- |, |s^+ -s |) \geq \eta_{0}$.

   We choose $\tilde{\varepsilon},\tilde{\tau}<\tau_0$ with the following properties: $\tilde{\varepsilon}< \eta_0$, where $\eta_0$ is defined in Remark~\ref{rmk:stimax1} and  $\tilde{\varepsilon} , \tilde{\tau}$ are  such that 
   \begin{equation}\label{eq:epstau}
      \begin{split}
         \frac{7C/16 \tilde\varepsilon^{d+1}}{ \tilde\varepsilon^p + \tau^{p/\beta} } \geq 1 , \qquad \text{where} \qquad  
     K_{\tau}(\zeta)\geq \frac{C}{|\zeta |^p + \tau^{p/\beta}},
      \end{split}
   \end{equation}
   for every $\tau < \tilde{\tau}$ (see \eqref{eq:Kbound} for the estimates on the kernel). 

    Because of condition (a) in the statement of the lemma, there exists a cube $Q^\perp_{\tilde{\varepsilon}}(t'^\perp_i)\subset \R^{d-1}$ of size $\tilde{\varepsilon}$,  such that $t^\perp_i\in Q^{\perp}_{\tilde{\varepsilon}}(t'^\perp_{i})$ and  $(s-\tilde{\varepsilon},s+\tilde{\varepsilon})\times Q^{\perp}_{\tilde{\varepsilon}}(t'^\perp_i) \subset [0,l)^d$. 
  
   By definition, one has that   
   \begin{equation}\label{eq:rv} 
      \begin{split}
         r_{i,\tau}(E,t^{\perp}_{i},s)+ v_{i,\tau}(E,t^{\perp}_{i},s)\geq -1 + \int_{s -\tilde{\varepsilon }}^{s + \tilde{\varepsilon}}\int_{-2\tilde{\varepsilon}}^{2\tilde\varepsilon} \int_{Q^\perp_{\tilde{\varepsilon}}(t'^\perp_{i})} f_{E}(t^{\perp}_{i},t_{i},\zeta^{\perp}_{i},\zeta_{i}) K_{\tau}(\zeta) \d\zeta_i^{\perp}\d\zeta_i\dt_{i} . 
      \end{split}
   \end{equation} 
   
   In order to estimate the r.h.s. of \eqref{eq:rv} from below, one proceeds exactly as in the proof of Lemma \ref{lemma:new_lemma_pre1} (using now that assumption (b) telling that $E$ is $L^1$-close to stripes on the cube $[0,l)^d$ instead of $[0,L)^d$, see Figure \ref{fig:stability}) and obtains

   \begin{equation*} 
      \begin{split}
         \int_{s-\tilde{\varepsilon}}^{s}\int_{-2\tilde \varepsilon} ^{2\tilde{ \varepsilon }}  \int_{Q^{\perp}_{\tilde{\varepsilon}}(t'^{\perp}_{i})} f_{E}(t^{\perp}_{i},u,\zeta^{\perp}_{i},\zeta_{i}) K_{\tau}(\zeta) \d\zeta \du &\geq\frac{7C\slash 16 \tilde{\varepsilon}^{d+1}}{\tilde{\varepsilon}^{p} +\tau^{ p/\beta}}.
         \end{split}
         \end{equation*} 

  Then, by \eqref{eq:rv} and \eqref{eq:epstau}, $   r_{i,\tau}(E,t^{\perp}_{i},s)+ v_{i,\tau}(E,t^{\perp}_{i},s)\geq0$.

   \end{proof}

   In the proof of Theorem~\ref{T:1.3}, we will deal with the following situation: given a segment $J$ on a slice in direction $e_i$, we want to estimate from below the contribution of the energy on $J$, namely $\int_{J} \bar{F}_{i,\tau}(E, Q_{l}(t^{\perp}_{i}+se_i))\ds$.  
   
   Point (i) below aims at estimating the contribution on $J$ assuming that on $l$-cubes around the points in $J$ the set $E$ is close to stripes in another direction $e_j$, $j\neq i$. Considering only points $s\in J$ which are far from the boundary, like in $(a)$ of Lemma \ref{lemma:stimaContributoVariazionePiccola}, then by Lemma \ref{lemma:stimaContributoVariazionePiccola} we would have that the contribution is positive. The possibly negative contribution in \eqref{eq:gstr20} is due to the presence of the boundary terms, and therefore is absent in \eqref{eq:gstr21} when $J=[0,L)$.
   
   Point (ii) wants to estimate the contribution on $J$ without making assumptions on the behaviour around points in $J$. One uses here, far from $\partial J$, the one-dimensional estimate of Lemma \ref{lemma:1D-opitmization}. Close to the boundary, one has eventually some additional negative contribution, which is smaller in absolute value when around the points of $\partial J$ one is close to stripes in another direction (see \eqref{eq:gstr27} and \eqref{eq:gstr36} and Remark \ref{rem:7.10}). 

   \begin{lemma}
      \label{lemma:stimaLinea}
      Let $\tilde{\eps},\tilde{\tau}>0$ as in Lemma \ref{lemma:stimaContributoVariazionePiccola}. Let $\delta=\eps^d/(16l^d)$ with $0<\eps\leq\tilde{\eps}$, $\tau\leq\tilde{\tau}$ and $l>C_0/(-C^*_\tau)$, where $C_0$ is the constant appearing in Lemma \ref{lemma:1D-opitmization}. Let $t_i^\perp\in[0,L)^{d-1}$ and $\eta>0$.

   The following statements hold: there exists $M_0$ constant independent of $l$ (but depending on the dimension) such that  
      \begin{enumerate}[(i)]
         \item Let $J\subset \R$ an interval  such that for every $s\in J$ one has that  $D^{j}_{\eta}(E,Q_{l}(t^{\perp}_{i}+se_i))\leq \delta$ with $j\neq i$. 
            Then
            \begin{equation}
               \label{eq:gstr20}
               \begin{split}
                  \int_{J} \bar{F}_{i,\tau}(E,Q_{l}(t^{\perp}_{i}+se_i))\ds \geq - \frac{M_{0}}{l}.
               \end{split}
            \end{equation}
            Moreover, if $J = [0,L)$, then 
            \begin{equation}
               \label{eq:gstr21}
               \begin{split}
                  \int_{J} \bar{F}_{i,\tau}(E,Q_{l}(t^{\perp}_{i}+se_i))\ds \geq0.
               \end{split}
            \end{equation}
         \item Let $J = (a,b)\subset \R$. 
           If for $s=a$ and $s=b$ it holds $D_\eta^j(E,Q_{l}(t^{\perp}_i+se_i)) \leq \delta$ with $j\neq i$, then 
            \begin{equation}
               \label{eq:gstr27}
               \begin{split}
                  \int_{J} \bar{F}_{i,\tau}(E,Q_{l}(t^{\perp}_{i}+se_i))\ds \geq | J| C^{*}_{\tau} -\frac{M_0} l,
               \end{split}
            \end{equation}
            otherwise
            \begin{equation}
               \label{eq:gstr36}
               \begin{split}
                  \int_{J} \bar{F}_{i,\tau}(E,Q_{l}(t^{\perp}_{i}+se_i))\ds \geq | J| C^{*}_{\tau} - M_0l.
               \end{split}
            \end{equation}
            Moreover, if $J = [0,L)$, then
            \begin{equation}
               \label{eq:gstr28}
               \begin{split}
                  \int_{J} \bar{F}_{i,\tau}(E,Q_{l}(t^{\perp}_{i}+se_i))\ds \geq | J| C^{*}_{\tau}.
               \end{split}
            \end{equation}
   \end{enumerate}
   \end{lemma}

   \begin{remark}\label{rem:7.10}
      Note that since $C^{*}_{\tau} < 0$, one has that \eqref{eq:gstr20} is a stronger inequality when compared to \eqref{eq:gstr27}. This is due to the additional hypothesis $D^{j}_{\eta}(E,Q_{l}(t^\perp_i+se_i)) \leq \delta$, $j\neq i$, for all $s\in J$.
   \end{remark}

   \begin{proof}
   Let us now prove (i). Since the result is valid for a general set $E$, we may assume  without loss of generality that $a=0$ and $b=l'$, namely $J= [0,l')$. We consider the decomposition $Q_l(t^{\perp}_{i}+t_{i}e_i) = Q^{\perp}_{l}(t^{\perp}_{i}) \times Q^i_{l}(t_{i})$, where $Q^{\perp}_{l}(t^{\perp}_{i})\subset \R^{d-1}$ is the cube of size $l$ centered at $t^{\perp}_{i}$ and $Q^{i}_{l}(t_{i})\subset\R$ is the interval of size $l$ centered in $t_{i}$.   
   From the definition of $\bar{F}_{i,\tau}$ (see~\ref{eq:fbartau}) and since $w_{i,\tau}\geq0 $, we have that
  
   \begin{equation}
      \label{eq:gstr25}
      \begin{split}
         \int_{J} \bar{F}_{i,\tau}(E,Q_{l}(t^{\perp}_i+se_i)) \ds &\geq  \frac{1}{l^{d}} 
         \int_{J} \int_{Q^{\perp}_{l}(t^{\perp}_{i})}\sum_{\substack{s' \in \partial E_{t'^{\perp}_{i}} 
               \\ s'\in Q^i_{l}(s)}} \Big(r_{i,\tau }(E,t'^{\perp}_{i},s') + v_{i,\tau }(E,t'^{\perp}_{i},s')\Big) \dt'^{\perp}_{i} \ds \\ 
        & = \frac{1}{l^d}\int_{J} \int_{Q^{\perp}_{l}(t^{\perp}_{i})}  \sum_{\substack{s' \in \partial E_{t'^{\perp}_{i}} 
            }}\chi_{Q^{i}_{l}(s)}(s')\Big( r_{i,\tau }(E,t'^{\perp}_{i},s') + v_{i,\tau }(E,t'^{\perp}_{i},s')\Big) \dt'^{\perp}_{i} \ds 
         \\ 
         &= \frac{1}{l^{d-1}} \int_{Q^{\perp}_{l}(t^{\perp}_{i})} \sum_{\substack{s' \in \partial E_{t'^{\perp}_{i}} 
               \\ s'\in (-\frac l2, l' + \frac l2)}}\frac{|Q^{i}_{l}(s')\cap J|}{l} \Big(r_{i,\tau }(E,t'^{\perp}_{i},s') + v_{i,\tau }(E,t'^{\perp}_{i},s')\Big) \dt'^{\perp}_{i}.
      \end{split}
   \end{equation}
   where in order to obtain the last line, we have used Fubini to integrate first in $s$. 

   Let us  now estimate the last term in \eqref{eq:gstr25}.
   
   For every point $s'\in\partial E_{t_i'^\perp}$ such that $\min(|s'+l/2|,|s'-l' - l/2 |) > \eta_0$, using the fact that $\tau\leq\tilde{\tau}$ and $\delta = \varepsilon^d/(16l^d)$, $0<\varepsilon<\tilde{\eps}$, where $\tilde{\eps}, \tilde{\tau}$ are given in the statement of Lemma~\ref{lemma:stimaContributoVariazionePiccola}, and $\eta>0$,  we can apply such Lemma  and obtain that $r_{i,\tau}(E,t'^{\perp}_{i},s') + v_{i,\tau}(E, t'^{\perp}_{i},s') \geq 0$.

   Let us estimate the contribution of  $s'$  when $s'$ is close to the boundary, namely  $\min(|s'+l/2|,|s'-l' - l /2|) < \eta_0$. 
   Since  $r_{i,\tau}$ is positive when the neighbouring points are closer than $\eta_0$, by using $r_{i,\tau}\geq-1$ and $v_{i,\tau} \geq 0$, for every $t'^\perp_i$ we have that 
   \begin{equation*}
      \begin{split}
         \sum_{\{s':\ \min(|s'+\frac l2|,|s'-l' -\frac l2 |) < \eta_0\}} \Big( r_{i,\tau}(E,t'^{\perp}_{i},s') + v_{i,\tau}(E,t'^{\perp}_{i},s') \Big) \geq -2.
      \end{split}
   \end{equation*}
   Since when $s'$ is such that $\min(|s'+l/2|,|s'-l' - l/2 |) < \eta_0$ we have that $\frac{ |Q^{i}_{l}(s')\cap J|}l \leq \frac{\eta_0}l$, by plugging the above on the \rhs of \eqref{eq:gstr25}, we have the first claim.

Let us now prove  \eqref{eq:gstr21}.
\begin{equation*}
\begin{split}
\int_{0}^L \bar{F}_{i,\tau}(E,Q_{l}(t^{\perp}_i+se_i)) \ds &\geq  \frac{1}{l^{d}} 
\int_{0}^L \int_{Q^{\perp}_{l}(t^{\perp}_{i})}\sum_{\substack{s' \in \partial E_{t'^{\perp}_{i}} 
		\\ s'\in Q^i_{l}(s)}} \Big(r_{i,\tau }(E,t'^{\perp}_{i},s') + v_{i,\tau }(E,t'^{\perp}_{i},s')\Big) \dt'^{\perp}_{i} \ds \\ 
& = \frac{1}{l^d}\int_{0}^L \int_{Q^{\perp}_{l}(t^{\perp}_{i})}  \sum_{\substack{s' \in \partial E_{t'^{\perp}_{i}} 
}}\chi_{Q^{i}_{l}(s)}(s')\Big( r_{i,\tau }(E,t'^{\perp}_{i},s') + v_{i,\tau }(E,t'^{\perp}_{i},s')\Big) \dt'^{\perp}_{i} \ds 
\\ 
&= \frac{1}{l^{d-1}} \int_{Q^{\perp}_{l}(t^{\perp}_{i})} \sum_{\substack{s' \in \partial E_{t'^{\perp}_{i}} 
		\\ s'\in [0,L)}} \Big(r_{i,\tau }(E,t'^{\perp}_{i},s') + v_{i,\tau }(E,t'^{\perp}_{i},s')\Big) \dt'^{\perp}_{i},
\end{split}
\end{equation*}
where the last step follows from using initially Fubini as in \eqref{eq:gstr25} and afterwards the $L$-periodicity.

Hence \eqref{eq:gstr21} holds since when $J=[0,L)$ by periodicity we do not have to care for points $s'\in\partial E_{t_i'^\perp}$ close to $\partial J$ and we can directly apply Lemma \ref{lemma:stimaContributoVariazionePiccola} giving $r_{i,\tau}(E,t'^{\perp}_{i},s') + v_{i,\tau}(E,t'^{\perp}_{i},s') \geq 0$.

Let us now prove (ii).  W.l.o.g. let us assume that $J= (0,l')$.  

As in \eqref{eq:gstr25} one has that 

   \begin{equation}
      \label{eq:gstr32}
      \begin{split}
         \int_{J} \bar{F}_{i,\tau}(E,Q_{l}(t^{\perp}_i+se_i)) \ds &         \geq \frac{1}{l^{d-1}} \int_{Q^{\perp}_{l}(t^{\perp}_{i})} \sum_{\substack{s' \in \partial E_{t'^{\perp}_{i}} 
               \\ s'\in (-\frac l2, l' + \frac l2)}}\frac{|Q^{i}_{l}(s')\cap J|}{l} \Big(r_{i,\tau }(E,t'^{\perp}_{i},s') + v_{i,\tau }(E,t'^{\perp}_{i},s')\Big) \dt'^{\perp}_{i}\\
         & =   \frac{1}{l^{d-1}} \int_{Q^{\perp}_{l}(t^{\perp}_{i})} \sum_{\substack{s' \in \partial E_{t'^{\perp}_{i}} 
               \\ s'\in (\frac l2, l'-\frac l2 )}}\frac{|Q^{i}_{l}(s')\cap J|}{l} \Big(r_{i,\tau }(E,t'^{\perp}_{i},s') + v_{i,\tau }(E,t'^{\perp}_{i},s')\Big) \dt'^{\perp}_{i}\\
         & +   \frac{1}{l^{d-1}} \int_{Q^{\perp}_{l}(t^{\perp}_{i})} \sum_{\substack{s' \in \partial E_{t'^{\perp}_{i}} 
               \\ s'\in (-\frac l2, \frac l2 ]\cup [l'-\frac l2,l'+\frac l2)}}\frac{|Q^{i}_{l}(s')\cap J|}{l} \Big(r_{i,\tau }(E,t'^{\perp}_{i},s') + v_{i,\tau }(E,t'^{\perp}_{i},s')\Big) \dt'^{\perp}_{i}
      \end{split}
   \end{equation}
   where if $l'\leq l$, we have that $(l/2,l'-l/2)$ is empty and then 
   \begin{equation*}
   \begin{split}
   \frac{1}{l^{d-1}} \int_{Q^{\perp}_{l}(t^{\perp}_{i})} \sum_{\substack{s' \in \partial E_{t'^{\perp}_{i}} 
   		\\ s'\in (l/2, l'-l/2)}} \Big(r_{i,\tau }(E,t'^{\perp}_{i},s')  + v_{i,\tau }(E,t'^{\perp}_{i},s')\Big) \dt'^{\perp}_{i} = 0.
   \end{split}
   \end{equation*}


   Fix $t'^\perp_i$. 
   We will now estimate the contributions for $(t_i'^\perp,s')\in Q_{l}(t'^\perp_i)$ 
   and $(t_i'^\perp,s')\in Q_{l}(t'^\perp_i+l'e_i)$. 
   If the condition $D^{j}_{\eta}(E,Q_{l}(t'^\perp_i)) \leq \delta$ or $D^{j}_{\eta}(E,Q_{l}(t'^\perp_i+l'e_i)) \leq \delta$ is missing, then we will estimate $r_{i,\tau}(E,t'^\perp_{i},s') + v_{i,\tau}(E,t'^{\perp}_{i},s')$ from below with $-1$ whenever the neighbouring ``jump'' points are further than $\eta_0$, otherwise $r_{i,\tau} \geq 0$. 
   Hence, the last term in \eqref{eq:gstr32} can be estimated by
   \begin{equation*}
      \begin{split}
         &    \frac{1}{l^{d-1}} \int_{Q^{\perp}_{l}(t^{\perp}_{i})} \sum_{\substack{s' \in \partial E_{t'^{\perp}_{i}} 
               \\ s'\in (-\frac l2, \frac l2 ]}}\frac{|Q^{i}_{l}(s')\cap J|}{l} \Big(r_{i,\tau }(E,t'^{\perp}_{i},s') + v_{i,\tau }(E,t'^{\perp}_{i},s')\Big) \dt'^{\perp}_{i} \gtrsim - M_{0}l
         \\ 
         &\frac{1}{l^{d-1}} \int_{Q^{\perp}_{l}(t^{\perp}_{i})} \sum_{\substack{s' \in \partial E_{t'^{\perp}_{i}} 
               \\ s'\in [l' -\frac l2,l'+  \frac l2 )}}\frac{|Q^{i}_{l}(s')\cap J|}{l} \Big(r_{i,\tau }(E,t'^{\perp}_{i},s') + v_{i,\tau }(E,t'^{\perp}_{i},s')\Big) \dt'^{\perp}_{i} \gtrsim - M_{0}l. 
      \end{split}
   \end{equation*}
   If $l' > l$,  we have that 
   \begin{equation*}
   \begin{split}
   \frac{1}{l^{d-1}} \int_{Q^{\perp}_{l}(t^{\perp}_{i})} \sum_{\substack{s' \in \partial E_{t'^{\perp}_{i}} 
   		\\ s'\in (l/2, l'-l/2)}} r_{i,\tau }(E, t'^{\perp}_{i},s')  \dt'^{\perp}_{i} \geq C^{*}_{\tau} |J|-lC^*_\tau-{C_0},
   \end{split}
   \end{equation*}
   where in the last inequality we have used Lemma~\ref{lemma:1D-optimization} for $E=E_{t'^{\perp}_{i}}$ and $I=(l/2,l'-l/2)$. 
 
   Combining the above with the fact that the same term is $0$ if $l'<l$, we have that
   \begin{equation*}
   \begin{split}
   \frac{1}{l^{d-1}} \int_{Q^{\perp}_{l}(t^{\perp}_{i})} \sum_{\substack{s' \in \partial E_{t'^{\perp}_{i}} 
   		\\ s'\in (l/2, l'-l/2)}} r_{i,\tau }(E,t'^{\perp}_{i},s')  \dt'^{\perp}_{i} \geq \big(C^{*}_{\tau} |J|-lC^*_\tau-{C_0}\big) \chi_{(0, +\infty)}(|J| -l),
   \end{split}
   \end{equation*}
   
   Thanks to the fact that $l>C_0/(-C^*_\tau)$, one has \eqref{eq:gstr36}.


   Let us now turn to the proof of \eqref{eq:gstr27}.
 Given that $D^j_{\eta}(E,Q_{l}(t^{\perp}_{i}))\leq \delta$ and $D^{j}_{\eta}(E,Q_{l}(t^\perp_i+l'e_i)) \leq\delta$ for some $j\neq i$, by Lemma \ref{lemma:stimaContributoVariazionePiccola} with $\delta=\varepsilon^d/(16l^d)$ we have that 
   \begin{equation}
      \label{eq:gstr37}
      \begin{split}
         r_{i,\tau}(E,t'^\perp_i,s') + v_{i,\tau}(E,t'^\perp_i,s') \geq  0 
      \end{split}
   \end{equation}
   whenever $\min(|s' - l'+l/2| ,|s'- l' -l/2 |) \geq \eta_0$  and $(t'^\perp_i+s'e_i)\in Q_{l}(t^\perp_i+l'e_i)$
   or  $\min(|s'+l/2| ,|s'-l/2 |)\geq \eta_0$  and $(t'^\perp_i+s'e_i)\in Q_{l}(t^\perp_i)$. 

   Fix $t'^{\perp}_{i}$. Then
   \begin{equation*}
      \begin{split}
          \sum_{\substack{s' \in \partial E_{t'^{\perp}_{i}} 
               \\ s'\in (-\frac l2, \frac l2 )}}\frac{|Q^{i}_{l}(s')\cap J|}{l} & \Big(r_{i,\tau }(E,t'^{\perp}_{i},s') + v_{i,\tau }(E,t'^{\perp}_{i},s')\Big) 
         \\  \geq & \sum_{\substack{s' \in \partial E_{t'^{\perp}_{i}} 
               \\ s'\in (-\frac l2, \frac l2 )\\
            \min(|s'+l/2| ,|s'-l/2 |)\geq \eta_0
            }}\frac{|Q^{i}_{l}(s')\cap J|}{l} \Big(r_{i,\tau }(E,t'^{\perp}_{i},s') + v_{i,\tau }(E,t'^{\perp}_{i},s')\Big)
         \\  + & \sum_{\substack{s' \in \partial E_{t'^{\perp}_{i}} 
               \\ s'\in (-\frac l2, \frac l2 )\\
            \min(|s'+l/2| ,|s'-l/2 |)<  \eta_0
            }}\frac{|Q^{i}_{l}(s')\cap J|}{l} \Big(r_{i,\tau }(E,t'^{\perp}_{i},s') + v_{i,\tau }(E,t'^{\perp}_{i},s')\Big) 
      \end{split}
   \end{equation*}
   Thus by using \eqref{eq:gstr37}, we have that the first term on the \rhs above is positive. To estimate the last term on the \rhs above we notice that $r_{i,\tau} \geq 0 $ whenever the neighbouring points are closer than $\eta_0$ and otherwise $r_{i,\tau}\geq  -1$. 
   Moreover, given that  $\frac{|Q^i_l{(s')}\cap J |}{l} < \frac{\eta_0}{l}$ for $s' \in (-l/2,l/2)\cup (l'-l/2,l'+l/2)$, we have that the last term on the \rhs above can be bounded from below by $-M_0/l$. Finally integrating over $t'^\perp_i$ we obtain that
   \begin{equation*}
      \begin{split}
         \frac{1}{l^{d-1}} \int_{Q^{\perp}_{l}(t^{\perp}_{i})} \sum_{\substack{s' \in \partial E_{t'^{\perp}_{i}} 
               \\ s'\in (-l/2, l/2 )\cup (l'-l/2,l'+l/2)}}\frac{|Q^{i}_{l}(s')\cap J|}{l} \Big(r_{i,\tau }(E,t'^{\perp}_{i},s') + v_{i,\tau }(E,t'^{\perp}_{i},s')\Big) \dt'^{\perp}_{i} \gtrsim -\frac{M_{0}}{l}.
      \end{split}
   \end{equation*}

   By using the above inequality in \eqref{eq:gstr32} and  the fact  that for every $s'\in (l/2,l'-l/2)$ it holds $\frac{|Q_{l}(s')\cap J |}{l} =1 $, we have that
   \begin{equation*}
      \begin{split}
         \int_{J} \bar{F}_{i,\tau}(E,Q_{l}(t^{\perp}_i+se_i)) \ds 
         & \geq   \frac{1}{l^{d-1}} \int_{Q^{\perp}_{l}(t^{\perp}_{i})} \sum_{\substack{s' \in \partial E_{t'^{\perp}_{i}} 
               \\ s'\in (l/2, l'-l/2 )}} \Big(r_{i,\tau }(E,t'^{\perp}_{i},s') + v_{i,\tau }(E,t'^{\perp}_{i},s')\Big) \dt'^{\perp}_{i}
         - \frac{M_{0}}{l}
      \end{split}
   \end{equation*}

   To conclude the proof of \eqref{eq:gstr27}, as for \eqref{eq:gstr36}, we notice that
   \begin{equation*}
      \begin{split}
         \frac{1}{l^{d-1}} \int_{Q^{\perp}_{l}(t^{\perp}_{i})} \sum_{\substack{s' \in \partial E_{t'^{\perp}_{i}} 
               \\ s'\in (l/2, l'-l/2)}} r_{i,\tau }(E,t'^{\perp}_{i},s')  \dt'^{\perp}_{i} \geq \big(C^{*}_{\tau} |J|-lC^*_\tau-{C_0}\big)\chi_{(0,+\infty)}(|J|-l),
      \end{split}
   \end{equation*}
   where in the last inequality we have used Lemma~\ref{lemma:1D-optimization} for $E=E_{t'^{\perp}_{i}}$, $I=(l/2,l'-l/2)$. Thanks to the fact that $l>C_0/(-C^*_\tau)$, one gets \eqref{eq:gstr27}. 

 If $|J|\leq l$, then the first sum on the \rhs of \eqref{eq:gstr32} is performed on an empty set. Therefore, in both \eqref{eq:gstr27} and \eqref{eq:gstr36} one has only the boundary terms and can conclude in a similar way.    

The proof of \eqref{eq:gstr28} proceeds using the $L$-periodicity of the contributions as done for \eqref{eq:gstr21}. 

\end{proof}

The next lemma gives a simple lower bound on the energy in the case almost all the volume of $Q_l(z)$ is filled by $E$ or $E^c$ (this will be the case on the set $A_{-1}$ defined in \eqref{a1}).
\begin{lemma}
   \label{lemma:stimaQuasiPieno}
   Let $E$ be a set of locally finite perimeter  such that $\min(|Q_{l}(z)\setminus E|, |E\cap Q_{l}(z) |)\leq {\delta} l^d$, for some $\delta>0$. Then 
   \begin{equation*}
      \begin{split}
         \bar F_{\tau} (E,Q_{l}(z)) \geq -\frac {\delta d } {\eta_0 },
      \end{split}
   \end{equation*}
   where $\eta_0$ is defined in Remark \ref{rmk:stimax1}.
\end{lemma}

\begin{proof}
   By assumption,  w.l.o.g. $|Q_{l}(z)\setminus E| \leq\delta l^d$.  

   Fix $t^\perp_{i}$ and consider the slice $E_{t^\perp_i}$. Then one has that $E_{t^{\perp}_{i}} = \bigcup_{j=1}^{n} (u_{j},s_{j})$. Since whenever $u_{j+i}- s_{j} \leq \eta_0$ one has that $r_{i,\tau}(E,t^\perp_{i},s_{j}) \geq 0$, and otherwise $r_{i,\tau}(E,t^{\perp}_{i},s_{j}) \geq -1$, one has that 
   \begin{equation*}
      \begin{split}
         \sum_{j=1}^{n} r_{i,\tau}(E,t^{\perp}_{i},s_{j}) \geq \sum_{\{u_{j+1} - s_{j} > \eta_0\}} -1  \geq -\frac{|Q^{i}_l(z)\setminus E_{t^{\perp}_{i}}|}{\eta_0}.
      \end{split}
   \end{equation*}
   Thus integrating over $t^{\perp}_{i}$ and using the definition of $\bar{F}_{i,\tau}(E,Q_{l}(z))$ \eqref{eq:fbartau} one has that 
   \begin{equation*}
      \begin{split}
         \bar{F}_{i,\tau}(E,Q_{l}(z))         \geq -\frac{1}{l^d} \int_{Q^{\perp}_l(z^{\perp}_{i})} \frac{|Q^{i}_l(z)\setminus E_{t^\perp_i} |}{\eta_0} \dt^{\perp}_{i} \geq -\frac \delta {\eta_0 }. 
      \end{split}
   \end{equation*}
\end{proof}

\subsection{Proof of Theorem~\ref{T:1.3}}\  
	The sets defined in the proof and the main estimates will depend on a set of parameters $l,\delta,\rho,M,\eta$ and $\tau$. If suitably chosen, they lead to the proof of the theorem for any $L>l$ of the form $2kh^*_\tau$ with $k\in\N$.  Recall that $h^{*}_{\tau}$ is the width of the periodic stripes which minimize the energy density $\Fcal_{\tau,L}$ among all periodic stripes as $L$ varies. 
	
	Let us first specify how the parameters are chosen, and their dependence on each other. The reason for such choices will be clarified during the proof.
	
	Let $0<\sigma<-C^*/2$, where $C^*$ is the energy density of optimal periodic stripes for $\Fcal_{0,L}$ (defined in \eqref{def:F0}) over all $L$. 
   Notice that $C^* < 0$ and $\lim_{\tau\downarrow 0} C^*_{\tau} = C^*$.

   \begin{itemize}
      \item  We first fix $l>0$ s.t.
         \begin{equation}\label{eq:lfix}
         l>\max \Big \{ \frac{dC_d}{-C^*-\sigma}, \frac{C_0}{-C^*-\sigma}\Big\}, 
    \end{equation}
    where $C_{d}$ is a constant (depending only on the dimension $d$) that appears in \eqref{eq:toBeShown_integral}, and
    $C_0$ is the constant which appears in the statement of Lemma \ref{lemma:1D-opitmization}.

   \item Let $\eta_0$ and $\tau_0$  as in Remark~\ref{rmk:stimax1}. Then from Lemma~\ref{lemma:stimaContributoVariazionePiccola},  have the parameters  $\tilde{\varepsilon} = \tilde{\varepsilon}(\eta_0,\tau_0)$ and $\tilde{\tau} = \tilde{\tau}(\eta_0, \tau_0)$. 
      
      \item We then fix $\varepsilon < \tilde{\varepsilon}$, $\tau < \tilde{\tau}$ as in Lemma~\ref{lemma:stimaLinea}. Thus we obtain  $\delta$ defined by $\delta  = \frac{\varepsilon^d}{16}$. Moreover,  by choosing $\varepsilon$ sufficiently small we can additionally assume that if for some $\eta>0$
    \begin{equation}\label{eq:deltafix2}
    D^i_{\eta}(E,Q_l(z))\leq\delta\text{ and }D^j_\eta(E,Q_l(z))\leq\delta,\:i\neq j\quad\Rightarrow\quad\min\{|E\cap Q_l(z)|,|E^c\cap Q_l(z)|\}\leq l^{d-1}.  
    \end{equation}
    This follows from Remark~\ref{rmk:lip} (ii). 

         \item Thanks to Remark \ref{rmk:lip} (i), we then fix
    \begin{equation}\label{eq:rhofix}
    \rho\sim\delta l. 
    \end{equation}
    in such a way that  we have that  for any $\eta$ the following holds
    \begin{equation}\label{eq:rhofix2}
    \forall\,z,z'\text{ s.t. }D_{\eta}(E,Q_l(z))\geq\delta,\:|z-z'|_\infty\leq\rho\quad\Rightarrow\quad D_\eta(E,Q_l(z'))\geq\delta/2,
    \end{equation}
    where for $x\in \R^d$,  $|x |_{\infty} = \max_i |x_{i} |$.

    \item Afterwards, we fix $M$ such that
    \begin{equation}
    \label{eq:Mfix}
    \frac{M\rho}{2d}>M_0l. 
    \end{equation}

    \item By applying Lemma~\ref{lemma:local_rigidity}, we obtain $\bar\eta=\bar{\eta}(M,l)$  and $\bar{\tau} = \bar{\tau}(M,l,\delta/2)$.  Thus we fix
   \begin{equation}\label{eq:etafix}
   0<\eta<\bar{\eta},\quad \bar\eta=\bar{\eta}(M,l)
   \end{equation}
   \item Finally, we choose further $\tau>0$ s.t.
   \begin{equation}
      \label{eq:taufix0}
      \begin{split}
         \tau<\tau_0  \qquad\text{as in Remark~\ref{rmk:stimax1},}
      \end{split}
   \end{equation}
   \begin{equation}
   \label{eq:taufix1}
   \tau<\tilde{\tau}, \,\tilde{\tau}\text{ as in Lemma \ref{lemma:stimaContributoVariazionePiccola} and Lemma \ref{lemma:stimaLinea}},
   \end{equation}
   \begin{equation}
   \label{eq:taufix2}
   \tau<\bar{\tau}, \text{ $\bar{\tau}$ as in Lemma \ref{lemma:local_rigidity} depending on $M,l,\delta/2$ and $\eta$}
   \end{equation}
   and
   \begin{equation}\label{eq:taufix3}
   \tau\text{ s.t. }C^*_\tau<C^*+\sigma.
   \end{equation}

   \end{itemize}

   Notice that, by the $\Gamma$-convergence result of Theorem \ref{thm:gammaconv}, $\exists\,\hat{\tau}$ s.t. if $\tau<\hat{\tau}$, then \eqref{eq:taufix3} holds. In particular, \eqref{eq:lfix} is satisfied with $-C^*_\tau$ instead of $-C^*-\sigma$.

Given such parameters, let us prove the theorem for any $L>l$ of the form $L=2kh^*_\tau$, with $k\in\N$.

Let $E$ be a minimizer of $\FtL$. Since $E$ is $L$-periodic, we can consider $E\subset \T^d_L$, where $\T^d_L$ is the $d$-dimensional torus of size $L$. Thus the problem is naturally defined on the torus. Hence
with a slight abuse of notation, we will denote by $[0,L)^d$  the cube of size $L$ with the usual identification of the boundary.

  {\bf Decomposition of $[0,L)^d$: }

We define
\begin{equation*}
   \begin{split}
      \tilde{A}_{0}:= \insieme{ z\in [0,L)^d:\ D_{\eta}(E,Q_{l}(z)) \geq \delta }.
   \end{split}
\end{equation*}
Hence, by Lemma \ref{lemma:local_rigidity}, for every $z\in \tilde{A}_{0}$ one has that $\bar{F}_{\tau}(E,Q_{l}(z)) > M$. 
   
Let us denote by $\tilde{A}_{-1}$ the set of points
\begin{equation*}
   \begin{split}
      \tilde{A}_{-1}: = \insieme{z\in [0,L)^d: \exists\, i,j \text{ with } i\neq j \text{ \st }\, D^{i}_{\eta} (E,Q_{l}(z))\leq\delta , D^{j}_{\eta} (E,Q_{l}(z)) \leq \delta }.
   \end{split}
\end{equation*}
One can easily see that $\tilde A_0$ and $\tilde{A}_{-1}$ are closed. 

By the choice of $\rho$ made in \eqref{eq:rhofix}, \eqref{eq:rhofix2} holds, namely for every $z\in \tilde{A}_{0}$ and $|z- z' |_\infty\leq\rho$ one has that $D_{\eta}(E,Q_{l}(z')) > \delta/2$.

Moreover, since $\delta$ satisfies \eqref{eq:deltafix2}, when $z\in \tilde{A}_{-1}$, then one has that $\min (|E\cap Q_{l}(z)|, |Q_{l}(z)\setminus E|) \leq  l^{d-1} $.

Thus, using  Lemma~\ref{lemma:stimaQuasiPieno} with $\delta=1/l$, one has that
\begin{equation*}
   \begin{split}
      \bar{F}_{\tau}(E, Q_{l}(z)) \gtrsim  -\frac1{l}.
   \end{split}
\end{equation*}
Moreover, let now $z'$ such that $|z- z' |_\infty\leq 1$ with $z\in \tilde{A}_{-1}$. It is not difficult to see that if $|Q_{l}(z)\setminus E | \leq l^{d-1}$ then $|Q_{l}(z')\setminus E| \lesssim l^{d-1}$. Thus from Lemma~\ref{lemma:stimaQuasiPieno}, one has that
\begin{equation}
\label{eq:tildeC}
   \begin{split}
      \bar{F}_{\tau}(E, Q_{l}(z')) \geq -\frac{\tilde C_d}{l}.
   \end{split}
\end{equation}

The above observations motivate the following definitions
   \begin{align}
         A_{0} &:= \insieme{ z' \in [0,L)^d: \exists\, z \in \tilde{A}_{0}\text{ with }|z-z'|_{\infty}  \leq \rho }\label{a0}\\
         A_{-1} &:= \insieme{ z' \in [0,L)^d: \exists\, z \in \tilde{A}_{-1}\text{ with }|z-z' |_{\infty}  \leq 1 },\label{a1}
   \end{align}

%
%

By the choice of the parameters and the observations above, for every $z\in A_{0}$ one has that $\bar{F}_{\tau}(E,Q_{l}(z)) > M$ and for every $z\in A_{-1}$, $\bar{F}_{\tau}(E,Q_{l}(z)) \geq-\tilde C_d/l$.

For simplicity of notation let us denote by $A:= A_{0}\cup A_{-1}$. 

The set $[0,L)^d\setminus A$ has the following property: for every $z\in [0,L)^d\setminus A$, there exists $i\in \{ 1,\ldots,d\}$ such that $D^{i}_{\eta}(E,Q_{l}(z)) \leq \delta$ and for every $k\neq i$ one has that $D^{k}_{\eta}(E,Q_{l}(z)) > \delta$.

   Given  that $A$ is closed, we consider the connected components $\mathcal C_{1},\ldots,\mathcal C_{n}$ of $[0,L)^d\setminus A$.  The sets $\mathcal C_{i}$ are path-wise connected. 

   Let us now show the following claim: given a connected component $\mathcal C_{j}$ one has that there exists  $i$ such that $D^{i}_{\eta}(E,Q_{l}(z)) \leq \delta$ for every $z\in\mathcal  C_{j}$  and for every $k\neq i$ one has that $D^{k}_{\eta}(E,Q_{l}(z)) > \delta$.  Indeed, suppose that there exists $z,z'\in\mathcal  C_{j}$  such that $D^{i}_{\eta}(E,Q_{l}(z)) \leq \delta$ and $D^{k}_{\eta}(E,Q_{l}(z')) \leq \delta$ with $i\neq k$ and take a continuous path $\gamma:[0,1]\to \mathcal C_{j}$ such that $\gamma(0) = z$ and $\gamma(1)= z'$.
   From our hypothesis, we have that $\{ s: D^{i}_{\eta}(E,Q_{l}(\gamma(s)))\leq\delta\} \neq \emptyset$ and there exists $\tilde{s}\in \partial \{s:\, D^{i}_{\eta}(E,Q_{l}(\gamma(s))) \leq \delta \}\cap\partial\{s:\,D^{j}_{\eta}(E,Q_{l}(\gamma(s))) \leq \delta\}$ for some $j\neq i$. Let $\tilde{z} = \gamma(\tilde{s})$. 
   Thus there are points arbitrary close to $\tilde{z}$ in $\mathcal C_{j}$ such that $D^{j}_{\eta}(E,Q_{l}(\cdot)) \leq \delta$ and $D^{i}_\eta(E,Q_{l}(\cdot)) \leq \delta$. 
   From the continuity of the maps $z\mapsto D^{i}_\eta(E,Q_{l}(z))$, $z\mapsto D^{j}_\eta(E,Q_{l}(z))$, we have that $\tilde{z}\in A_{-1} $, which contradicts our assumption. 
   We will say that $\mathcal C_j$ is oriented in direction $e_i$ if there is a point in $z\in \mathcal C_j$ such that $D^{i}_\eta(E,Q_{l}(z)) \leq \delta$. 
 Because of the above being oriented along direction $e_{i}$ is well-defined.

   We will denote by $A_{i}$ the union of the connected  components $\mathcal C_{j}$ such that $\mathcal C_{j}$ is oriented along the direction $e_{i}$. 
   
   An example of such a partition of $[0,L)^d$ for a periodic set $E$ is given in Figure \ref{fig:2bis}.
    
   \begin{figure}
   	\centering
   	\def\svgwidth{19cm}
   	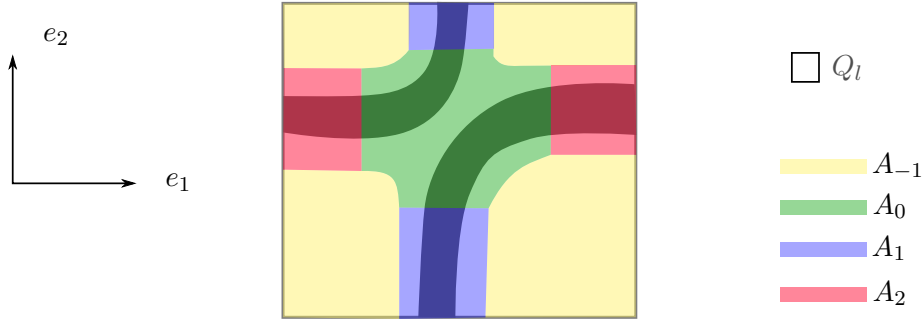
   	\caption{Example of decomposition of $[0,L)^d$ according to the distance of $E$ to stripes on squares $Q_l(z)$. The darker regions within the same color represent the set $E$.}
   	\label{fig:2bis}
   \end{figure}

   Let us now summarize the important properties that will be used in the following

   \begin{enumerate}[(i)]
   \item The sets $A=A_{-1}\cup A_{0}$, $A_{1}$, $A_{2}$, $\ldots, A_d$  form a partition of $[0,L)^d$. 
   \item The sets $A_{-1}, A_{0}$ are closed and $A_{i}$, $i>0$, are open.  
   \item For every $z\in A_{i}$, we have that $D^{i}_{\eta}(E,Q_{l}(z)) \leq \delta$. 
   \item  There exists $\rho$ (independent of $L,\tau$) such that  if $z\in A_{0}$, then $\exists\,z'$ s.t. $Q_{\rho}(z')\subset A_{0}$ and $z \in Q_{\rho}(z')$. If $z\in A_{-1}$ then $\exists\,z'$ s.t. $Q_{1}(z')\subset A_{-1}$ and $z \in Q_{1}(z')$. 
   \item For every $z\in {A}_{i}$ and $z'\in {A}_{j}$ one has that there exists a point $\tilde{z}$ in the segment connecting $z$ to $z'$ lying in ${A}_{0}\cup A_{-1}$. 
\end{enumerate}

   For simplicity we will denote by $B = \bigcup_{i> 0}A_{i}$.

{\bf Main Claim:}
For every $i$, and $\tau$ as in \eqref{eq:taufix1} and \eqref{eq:taufix2}, we will show that the following holds
\begin{equation}
   \label{eq:toBeShown_integral}
   \begin{split}
      \frac{1}{L^d}\int_{B} \bar{F}_{i,\tau}(E,Q_{l}(z))\dz + \frac1{d L^d} \int_{A}\bar{F}_{\tau}(E,Q_{l}(z)) \dz  \geq \frac{C^{*}_{\tau}|A_{i}|}{L^d} - C_d \frac{|A|}{l L^d}
   \end{split}
\end{equation}
for some constant $C_d$ depending on the dimension $d$.
Assuming  \eqref{eq:toBeShown_integral}, we can sum over $i$ and obtain that, since $l$ satisfies \eqref{eq:lfix} and $\tau$ satisfies \eqref{eq:taufix3}, 
\begin{equation*}
   \begin{split}
      \Fcal_{\tau,L}(E) & \geq \sum_{i=1}^{d}\frac{1}{L^d} \int_{[0,L)^d} \bar{F}_{i,\tau}(E,Q_{l}(z))  \dz \geq  \frac{C^{*}_{\tau}}{L^d} \sum_{i=1}^{d} |A_i|  - \frac{dC_d|A|}{lL^d}
      \\ & \geq C^{*}_{\tau} - C^*_{\tau} \frac{|A|} {L^d} - \frac{dC_d}{lL^d} |A| \geq C^{*}_{\tau},
   \end{split}
\end{equation*}
where in the above $C^{*}_{\tau}$ is the energy density of optimal stripes of stripes $h^{*}_{\tau}$ and  we have used that $C^{*}_{\tau} < 0$ and that $|A | + \sum_{i=1}^{d} |A_{i}| =  |[0,L)^d | = L^d$. 

Notice that, in the inequality above, equality holds only if $|A|=0$ and therefore by $(v)$ only if there is just one $A_i$, $i>0$ with $|A_i|>0$.

The rest of the proof is devoted to proving \eqref{eq:toBeShown_integral}.
Our claim, i.e. \eqref{eq:toBeShown_integral}, follows from the analogous statement on the slices: for every $t^{\perp}_{i}\in [0,L)^{d-1}$, it holds

\begin{equation}
   \label{eq:toBeShown_slice}
   \begin{split}
      \frac{1}{L^d} \int_{B_{t^{\perp}_{i}}} \bar{F}_{i,\tau}(E,Q_{l}(t^{\perp}_{i}+se_i))\ds + \frac1{dL^d} \int_{A_{t^{\perp}_{i}}}\bar{F}_{\tau}(E,Q_{l}(t^{\perp}_{i}+se_i)) \ds  \geq \frac{C^{*}_{\tau}|A_{i,t^{\perp}_{i}}|}{L^d} - C_d \frac{|A_{t^\perp_i}|}{l L^d}
   \end{split}
\end{equation}

Indeed by integrating \eqref{eq:toBeShown_slice} over $t^{\perp}_{i}$ we obtain \eqref{eq:toBeShown_integral}.

Notice also that $B_{t^{\perp}_{i}}$ is a finite union of intervals. Indeed, being a union of intervals follows from (ii) and finiteness follows from  condition  (v) on the decomposition.  Indeed, for every point that does not belong to $B_{t^\perp_{i}}$ because of (iv) there is a neighbourhood of fixed positive size that is not included in $B_{t^\perp_i}$. 
Let $\{ I_{1},\ldots,I_{n}\}$ such that $\bigcup_{j=1}^{n} I_{j} = B_{t_i^\perp}$ with $I_j \cap I_{k} = \emptyset$ whenever $j\neq k$. 
We can further assume that $I_{i} \leq I_{i+1}$, namely that for every $s\in I_{i}$ and $s'\in I_{i+1}$ it holds $s \leq s'$. 
By construction there exists $J_{j} \subset A_{t^{\perp}_{i}}$ such that $I_{j}\leq  J_{j} \leq I_{j+1}$.  

Whenever $J_j \cap A_{0,t_i^\perp}\neq\emptyset$,  we have that $|J_j | > \rho$  and whenever $J_{j} \cap A_{-1,t^\perp_i}\neq \emptyset $ then $|J_{i}| > 1$.

Thus we have that
\begin{equation*}
   \begin{split}
      \frac{1}{L^d} \int_{B_{t^{\perp}_{i}}} \bar{F}_{i,\tau}(E,Q_{l}(t^{\perp}_{i}+se_i)) \ds & + 
      \frac{1}{d L^d} \int_{A_{t^{\perp}_{i}}} \bar{F}_{\tau}(E,Q_{l}(t^{\perp}_{i}+se_i)) \ds 
      \\ & \geq \sum_{j=1}^n  \frac{1}{L^d}\int_{I_{j}} \bar{F}_{i,\tau}(E,Q_{l}(t^{\perp}_{i}+se_i)) \ds
       + \frac{1}{dL^d}\sum_{j=1}^n \int_{J_{j}} \bar{F}_{\tau}(E,Q_{l}(t^{\perp}_i+se_i)) \ds 
      \\ & \geq \frac{1}{L^d}\sum_{j=1}^n \Big( \int_{I_{j}} \bar{F}_{i,\tau}(E,Q_{l}(t^{\perp}_{i}+se_i)) \ds
       + \frac{1}{2d} \int_{J_{j-1}\cup J_j} \bar{F}_{\tau}(E,Q_{l}(t^{\perp}_i+se_i)) \ds\Big),
   \end{split}
\end{equation*}
where in order to obtain the third line from the second line, we have used periodicity and $J_0:=J_n$.

Let first $I_{j} \subset A_{i,t_i^\perp}$.  
By construction, we have that $\partial I_{j}\subset A_{t^\perp_i}$. 

If $\partial I_{j}\subset A_{-1,t^\perp_i}$, by using our choice of parameters, namely \eqref{eq:lfix} and \eqref{eq:taufix3}, we can apply \eqref{eq:gstr27} in Lemma~\ref{lemma:stimaLinea} and obtain
\begin{equation*}
   \begin{split}
      \frac{1}{L^d}\int_{I_{j}} \bar{F}_{i,\tau}(E,Q_{l}(t^{\perp}_{i}+se_i))\ds  \geq \frac{1}{L^d}\Big(| I_j| C^{*}_{\tau} -\frac{M_0} l\Big).
   \end{split}
\end{equation*}

If $\partial I_j \cap A_{0, t^\perp_i}\neq \emptyset$, by using our choice of parameters, namely \eqref{eq:lfix} and \eqref{eq:taufix3}, we can apply \eqref{eq:gstr36} in Lemma~\ref{lemma:stimaLinea}, and obtain
\begin{equation*}
   \begin{split}
      \frac{1}{L^d}\int_{I_{j}} \bar{F}_{i,\tau}(E,Q_{l}(t^{\perp}_{i}+se_i))\ds \geq\frac{1}{L^d}\Big(| I_j| C^{*}_{\tau}-M_0 l\Big).
   \end{split}
\end{equation*}

On the other hand, if $\partial I_j \cap A_{0,t^\perp_i}\neq \emptyset$, we have that either $J_{j}\cap A_{0,t^\perp_i}\neq \emptyset$ or $J_{j-1}\cap A_{0,t^\perp_i}\neq\emptyset$. Thus
\begin{equation*}
   \begin{split}
     \frac{1}{2dL^d}\int_{J_{j-1}} \bar{F}_{\tau}(E,Q_{l}(t^{\perp}_{i}+se_i)) \ds & + \frac{1}{2dL^d}\int_{J_{j}} \bar{F}_{\tau}(E,Q_{l}(t^{\perp}_{i}+se_i)) \ds  \\ &\geq  \frac{M\rho}{2dL^d}  - \frac{|J_{j-1}\cap A_{-1,t^\perp_i} |\tilde C_d}{2dl L^d} - \frac{|J_{j}\cap A_{-1,t^\perp_i} |\tilde C_d}{2dl L^d},
   \end{split}
\end{equation*}
where $\tilde C_d$ is the  constant in \eqref{eq:tildeC}.

Since $M$ satisfies \eqref{eq:Mfix}, in both cases $\partial I_{j}\subset A_{-1,t^\perp_i}$ or $\partial I_{j}\cap A_{0,t^\perp_i}\neq \emptyset$, we have that 
\begin{equation*}
   \begin{split}
      \frac{1}{L^d}\int_{I_{j}} \bar{F}_{i,\tau}(E,Q_{l}(t^\perp_i+se_i)) \ds &+ 
      \frac{1}{2dL^d}\int_{J_{j-1}}  \bar{F}_{\tau}(E,Q_{l}(t^{\perp}_{i}+se_i))\ds
      + \frac{1}{2dL^d}\int_{J_{j}}  \bar{F}_{\tau}(E,Q_{l}(t^{\perp}_{i}+se_i))\ds\\
      &\geq \frac{C^{*}_{\tau} |I_{j} |}{L^d} - \frac{|J_{j-1}\cap A_{-1,t^\perp_i}|\tilde C_d}{2dlL^d}
       - \frac{|J_{j}\cap A_{-1,t^\perp_i}|\tilde C_d}{2dlL^d}.
   \end{split}
\end{equation*}

If $I_{j} \subset A_{k,t^\perp_i}$ with $k \neq i$ from  Lemma~\ref{lemma:stimaLinea} Point (i) it holds
\begin{equation*}
   \begin{split}
      \frac 1{L^d}\int_{I_{j}} \bar{F}_{i,\tau}(E,Q_{l}(t^{\perp}_{i}+se_i)) \ds\geq  - \frac{M_0}{lL^d}.
   \end{split}
\end{equation*}

In general for every $J_{j}$  we have that 
\begin{equation*}
   \begin{split}
      \frac{1}{dL^d}\int_{J_{j}} \bar{F}_{\tau}(E,Q_{l}(t^{\perp}_{i}+se_i))\, \ds \geq   \frac{|J_{j}\cap A_{0,t^\perp_i} | M}{dL^d} - \frac{\tilde C_d}{dlL^d }|J_{j}\cap A_{-1,t^\perp_i}|. 
   \end{split}
\end{equation*}

For $I_{j}\subset A_{k,t^\perp_i}$ such that $(J_j \cup J_{j-1})\cap A_{0,t^\perp_i}\neq \emptyset$ with $k\neq i$, we have that 
\begin{equation*}
   \begin{split}
      \frac{1}{L^d}\int_{I_{j}} \bar{F}_{i,\tau}(E,Q_{l}(t^\perp_i+se_i)) \ds &+ 
      \frac{1}{2dL^d}\int_{J_{j-1}}  \bar{F}_{\tau}(E,Q_{l}(t^{\perp}_{i}+se_i))\ds
      + \frac{1}{2dL^d}\int_{J_{j}}  \bar{F}_{\tau}(E,Q_{l}(t^{\perp}_{i}+se_i))\ds\\
      &\geq -\frac{M_{0}}{lL^d} + \frac{M\rho}{2dL^d} - \frac{|J_{j-1}\cap A_{-1,t^\perp_i}|\tilde C_d}{2dlL^d}
       - \frac{|J_{j}\cap A_{-1,t^\perp_i}|\tilde C_d}{2dlL^d}.
       \\ &\geq
        - \frac{|J_{j-1}\cap A_{-1,t^\perp_i}|\tilde C_d}{2dlL^d}
       - \frac{|J_{j}\cap A_{-1,t^\perp_i}|\tilde C_d}{2dlL^d}.
   \end{split}
\end{equation*}
where the last inequality is true due to \eqref{eq:Mfix}.

%

For $I_{j}\subset A_{k,t^\perp_i}$ such that $(J_j \cup J_{j-1})\subset  A_{-1,t^\perp_i}$ with $k\neq i$, we have that 
\begin{equation*}
   \begin{split}
      \frac{1}{L^d}\int_{I_{j}} \bar{F}_{i,\tau}(E,Q_{l}(t^\perp_i+se_i)) \ds &+ 
      \frac{1}{2dL^d}\int_{J_{j-1}}  \bar{F}_{\tau}(E,Q_{l}(t^{\perp}_{i}+se_i))\ds
      + \frac{1}{2dL^d}\int_{J_{j}}  \bar{F}_{\tau}(E,Q_{l}(t^{\perp}_{i}+se_i))\ds\\
      &\geq -\frac{M_{0}}{lL^d}   - \frac{|J_{j-1}\cap A_{-1,t^\perp_i}|\tilde C_d}{2dlL^d}
       - \frac{|J_{j}\cap A_{-1,t^\perp_i}|\tilde C_d}{2dlL^d}.
       \\ &\geq
        - \max\Big(M_{0},\frac{\tilde C_d}{d}\Big)\bigg(\frac{|J_{j-1}\cap A_{-1,t^\perp_i}|}{lL^d}
       + \frac{|J_{j}\cap A_{-1,t^\perp_i}|}{lL^d}\bigg).
   \end{split}
\end{equation*}
where in the last inequality we have used that $|J_j\cap A_{-1,t^\perp_i}|\geq1, \,|J_{j-1}\cap A_{-1,t^\perp_i}|\geq1$.

Summing over $j$, and taking $C_d=\max\Big(M_{0},\frac{\tilde C_d}{d}\Big) $, one obtains \eqref{eq:toBeShown_slice} as desired.

Therefore, it has been proved that there exists $i>0$ with $A_i = [0,L)^d$. 
    Finally, let us consider 
 	\begin{align}
 	\frac{1}{L^d}\int_{[0,L)^d}\bar F_\tau(E,Q_l(z))\dz&=\frac{1}{L^d}\int_{[0,L)^d}\bar F_{i,\tau}(E,Q_l(z))\dz\label{eq:fi}\\
 	&+\frac{1}{L^d}\sum_{j\neq i}\int_{[0,L)^d}\bar F_{j,\tau}(E,Q_l(z))\dz\label{eq:fj}
 	\end{align}

   We will now apply Lemma~\ref{lemma:stimaLinea} with $j =i$ and slice the cube $[0,L)^d$ in direction $e_i$. 
   From \eqref{eq:gstr21}, one has that \eqref{eq:fj} is nonnegative and strictly positive unless the set $E$ is a union of stripes in direction $e_i$. 
   On the other hand, from \eqref{eq:gstr28}, one has the \rhs of \eqref{eq:fi} is minimized by a periodic union of stripes in direction $e_i$ and with width $h^*_{\tau}$. 

 	%

\begin{proof}[Proof of Theorem~\ref{T:1.6}]
   Let us recall the notation $\kappa = \tau^{1/\beta}$, which was introduced in Section~\ref{sec:discrete}. 
   By using the continuous representation $\tilde{E}^{\kappa}$ of a discrete set $E\subset \kappa \Z^d$,  which is described in Section~\ref{sec:discrete}, 
   one can see the rescaled discrete functional $\Fcal^\discrete_{\tau,L}(E)$ as the functional $\Fcal^{\discreteToCont}_{\tau,L}(\tilde{E}^{\kappa})$. 
   Therefore the discrete problem is transformed into the continuous problem into a continuous one. 

   All the statements that have been used to prove Theorem~\ref{T:1.3}, are obtained by using 
   \begin{enumerate}[(i)]
      \item qualitative properties of the kernel, namely \eqref{eq:Kbound}
      \item one-dimensional optimization via the reflection positivity (see assumption \ref{eq:laplpos}).
   \end{enumerate}

   In the discrete,  the kernel is $K^{\discrete}_{\kappa}(\zeta) = \frac{\kappa^{d}}{|\zeta|^p}$. The measure defined by $\mu_\kappa = \sum_{\zeta\in \kappa \Z^d\setminus\{0\}} K^{\discrete}_{\kappa}(\zeta)$ converges to the measure $\frac{1}{|\zeta|^p}\d\zeta$ and the piecewise constant function associated to $\mu_{\kappa}$, namely
\begin{equation*}
   \begin{split}
      x\to \sum_{\zeta \in \kappa \Z^{d}\setminus\{0\}} \frac{1}{|\zeta|^p}\chi_{Q_\kappa(\zeta)}(x)
   \end{split}
\end{equation*}
converges in $L^1_{\lok}(\R^d\setminus\{ 0\})$ to $\frac{1}{|x|^{p}}$. 

Thus in a similar way to the continuous setting, one can define similar quantities to $r_{i,\tau}$, $v_{i,\tau}$, $w_{i,\tau}$ and $\bar{F}_{i,\tau}$ and obtain the same type of estimates. Indeed, the only two lemmas that would depend on the specific form of the kernel are Lemma~\ref{lemma:1D-optimization} and Lemma~\ref{lemma:stimaLinea} in which Lemma~\ref{lemma:1D-optimization} is used. They depend on the specific form of the kernel as the reflection positivity technique is used. Given that reflection positivity holds for the discrete kernel, one can obtain the same type of results for Lemma~\ref{lemma:1D-optimization} and consequently for Lemma~\ref{lemma:stimaLinea}.

   As the proof of Theorem~\ref{T:1.3} is a consequence of the previous lemmas, one has that the proof follows by the same reasoning. 
\end{proof}

\section*{Acknowledgements}
We would like to thank M. Cicalese, A. Giuliani and E. Spadaro for useful comments.


\begin{thebibliography}{10}

   \bibitem{AcFuMo}
   E.~Acerbi, N.~Fusco, and M.~Morini.
   \newblock Minimality via second variation for a nonlocal isoperimetric problem.
   \newblock {\em Comm. Math. Phys.}, 322(2):515--557, 2013.

   \bibitem{ACO}
   G.~Alberti, R.~Choksi, and F.~Otto.
   \newblock Uniform energy distribution for an isoperimetric problem with
   long-range interactions.
   \newblock {\em J. Amer. Math. Soc.}, 22(2):569--605, 2009.

   \bibitem{AFPBV}
   L.~Ambrosio, N.~Fusco, and D.~Pallara.
   \newblock {\em Functions of bounded variation and free discontinuity problems}.
   \newblock Oxford Mathematical Monographs. The Clarendon Press Oxford University Press, New York, 2000.

   \bibitem{AFR}
   L.~Ambrosio, A.~Figalli and E.~Runa
   \newblock On sets of finite perimeter in Wiener spaces: reduced boundary and convergence to halfspaces
   \newblock {\em  Atti Accad. Naz. Lincei Rend. Lincei Mat. Appl.} 24, no. 1, 111--122, 2013.


   \bibitem{MR2096672}
   L.~Ambrosio, B.~Kirchheim, and A.~Pratelli.
   \newblock Existence of optimal transport maps for crystalline norms.
   \newblock {\em Duke Math. J.}, 125(2):207--241, 2004.


   \bibitem{BianchiniDaneri}
   S.~Bianchini and S.~Daneri
   \newblock On Sudakov's type decomposition of transference plans with norm costs
   \newblock {\em Memoirs AMS} 251, n. 1197, 2018.

   \bibitem{MR2883679}
   S.~Bianchini and M.~Gloyer.
   \newblock Transport rays and applications to {H}amilton-{J}acobi equations.
   \newblock In {\em Nonlinear {PDE}'s and applications}, volume 2028 of {\em
      Lecture Notes in Math.},  1--15. Springer, Heidelberg, 2011.


   \bibitem{BlancLewin}
   X.~Blanc and M.~Lewin.
   \newblock The crystallization conjecture: a review.
   \newblock {\em EMS Surv. Math. Sci.}, 2(2):225--306, 2015.

   \bibitem{bourne2012optimality}
   D.~Bourne, M. A.~Peletier, and F. Theil.
   \newblock Optimality of the triangular lattice for a particle system with
   {W}asserstein interaction.
  \newblock {\em Comm. Math. Phys},  329, 1, 117--140, 2014.

   \bibitem{Bre}
   H.~Brezis.
   \newblock How to recognize constant functions. Connections with {S}obolev
   spaces.
   \newblock {\em Russian Mathematical Surveys}, 57(4):693-708, 2002.

   \bibitem{MR2805441}
   L.~Caravenna.
   \newblock A proof of {S}udakov theorem with strictly convex norms.
   \newblock {\em Math. Z.}, 268(1-2):371--407, 2011.

   \bibitem{CarDan}
   L. Caravenna and S. Daneri
   \newblock The disintegration of the Lebesgue measure on the faces of a convex function.
   \newblock {\em J. Funct. Anal.} 258 (11), 3604--3661, 2010.

   \bibitem{chakrabarty2011modulation}
   S. Chakrabarty and Z. Nussinov.
   \newblock Modulation and correlation lengths in systems with competing
   interactions.
   \newblock {\em Physical Review B}, 84(14):144402, 2011.


   \bibitem{chenOshPeriodicity}
   X. Chen and Y. Oshita.
   \newblock  Periodicity and uniqueness of global minimizers of an energy functional containing a long-range interaction
   \newblock {\em SIAM J. Math. Anal.}, 37(4):1299--1332, 2005

   \bibitem{MR2338353}
   X. Chen and Y. Oshita.
   \newblock An application of the modular function in nonlocal variational
   problems.
   \newblock {\em Arch. Ration. Mech. Anal.}, 186(1):109--132, 2007.

   \bibitem{choksi2010small}
   R. Choksi and M. A. Peletier.
   \newblock Small volume fraction limit of the diblock copolymer problem: I.
   sharp-interface functional.
   \newblock {\em SIAM Journal on Mathematical Analysis}, 42(3):1334--1370, 2010.

   \bibitem{CicSpa}
   M.~Cicalese and E.~Spadaro.
   \newblock Droplet minimizers of an isoperimetric problem with long-range
   interactions.
   \newblock {\em Comm. Pure Appl. Math.}, 66(8):1298--1333, 2013.

   \bibitem{de2000dipolar}
   K.~De'Bell, A.B.~MacIsaac, and J.P.~Whitehead.
   \newblock Dipolar effects in magnetic thin films and quasi-two-dimensional
   systems.
   \newblock {\em Reviews of Modern Physics}, 72(1):225, 2000.

   \bibitem{fro}
   J.~Fr\"ohlich, R.~Israel, E.~H. Lieb, and B.~Simon.
   \newblock Phase transitions and reflection positivity. II. Lattice systems with short range and Coulomb interactions.
     \newblock {\em Journ. Stat. Phys.}, 22(3):297--347, 1980. 
   
   \bibitem{FroSim}
   J.~Fr\"ohlich, B.~Simon, T.~Spencer.
   \newblock Infrared bounds, phase transitions and continuous symmetry breaking.
   \newblock{\em Comm. Math. Phys.}, 50(1):79--95, 1976.

   \bibitem{glllRP1}
   A.~Giuliani, J.L.~Lebowitz and E.H.~Lieb 
   \newblock Ising models with long-range dipolar and short range ferromagnetic interactions. 
   \newblock {\em Phys. Rev. B} 74, 064420, 2006.

   \bibitem{glllRP2}
   A.~Giuliani, J.L.~Lebowitz and E.H.~Lieb 
   \newblock Striped phases in two dimensional dipole systems. 
   \newblock {\em Phys. Rev.B} 76, 184426, 2007.

   \bibitem{glllRP3}
   A.~Giuliani, J.L.~Lebowitz and E.H.~Lieb 
   \newblock Periodic minimizers in 1D local mean field theory.
   \newblock {\em  Commun. in Math. Phys.} 286, 163--177, 2009.
   \bibitem{glllRP3bis}
   A.~Giuliani, J.L.~Lebowitz and E.H.~Lieb 
   \newblock Modulated phases of a one-dimensional sharp interface model in a magnetic field. 
   \newblock {\em  Phys. Rev. B} 80, 134420, 2009.

   \bibitem{2011PhRvB..84f4205G}
   A.~Giuliani, J.~L. Lebowitz, and E.~H. Lieb.
   \newblock Checkerboards, stripes, and corner energies in spin models with
   competing interactions.
   \newblock {\em Phys. Rev. B}, 84:064205, 2011.

   \bibitem{2014CMaPh.tmp..127G}
   A.~Giuliani, E.~H. Lieb, and R.~Seiringer.
   \newblock Formation of stripes and slabs near the ferromagnetic transition.
   \newblock {\em Comm. Math. Phys.}, 331(1):333--350, 2014.

   \bibitem{MR2864796}
   A.~Giuliani and S.~M{\"u}ller.
   \newblock Striped periodic minimizers of a two-dimensional model for
   martensitic phase transitions.
   \newblock {\em Comm. Math. Phys.}, 309(2):313--339, 2012.

   \bibitem{GiuSeirGS}
   A.~Giuliani and R.~Seiringer.
   \newblock Periodic striped ground states in {I}sing models with competing
   interactions.
   \newblock {\em Comm. Math. Phys.},  1--25, 2016.

   \bibitem{goldman2013gamma}
   D. Goldman, C.~B. Muratov, and S. Serfaty.
   \newblock The $\gamma$-limit of the two-dimensional Ohta--Kawasaki energy. i.
   droplet density.
   \newblock {\em Arch.  Rat.  Mech.  Anal.}, 210(2):581--613,
   2013.

   \bibitem{GR}
   M.~{Goldman} and E.~{Runa}.
   \newblock {On the optimality of stripes in a variational model with nonlocal
      interactions}.
   \newblock {\em ArXiv:1611.07228}, 2016.

   \bibitem{harrison2000mechanisms}
   C. Harrison, D.~H Adamson, Z. Cheng, J.~M Sebastian,
   S. Sethuraman, D.~A Huse, R.~A Register, and PM~Chaikin.
   \newblock Mechanisms of ordering in striped patterns.
   \newblock {\em Science}, 290(5496):1558--1560, 2000.


      \bibitem{KnMuMicroMagnetics}
      H.~Kn{\"u}pfer and C.~B. Muratov.
      \newblock Domain structure of bulk ferromagnetic crystals in applied fields near saturation
      \newblock {\em J. Nonlinear Sci.}, 21(6):921--962, 2011.


   \bibitem{KnMu}
   H.~Kn{\"u}pfer and C.~B. Muratov.
   \newblock On an isoperimetric problem with a competing nonlocal term {II}:
   {T}he general case.
   \newblock {\em Comm. Pure Appl. Math.}, 67(12):1974--1994, 2014.

   \bibitem{knupfer2016low}
   H. Kn{\"u}pfer, C.~B Muratov, and M. Novaga.
   \newblock Low density phases in a uniformly charged liquid.
   \newblock {\em Comm.  Math.  Phys.}, 345(1):141--183, 2016.

   \bibitem{kohn1992branching}
   R.~V Kohn and S. M{\"u}ller.
   \newblock Branching of twins near an austenite--twinned-martensite interface.
   \newblock {\em Philosophical Magazine A}, 66(5):697--715, 1992.
   
   \bibitem{Maggi}
   F. Maggi. 
   \newblock Sets of finite perimeter and geometric variational problems.
   \newblock{\em Cambridge Studies in Advanced Mathematics} 135, Cambridge university Press, 2012.


   \bibitem{landaulif}
   L.D.~Landau and E.M. Lifshitz. 
   \newblock{Course of Theoretical Physics, vol. 8.}
   \newblock{Pergamon, London, 1984.}
   

   \bibitem{MoriniSternberg}
   M. Morini and P. Sternberg.
   \newblock Cascade of minimizers for a nonlocal isoperimetric problem in thin domains. 
   \newblock {\em SIAM Jour.  Math.  Anal.}, 46:2033--2051, 2014.

   \bibitem{muller1993singular}
   S. M{\"u}ller.
   \newblock Singular perturbations as a selection criterion for periodic
   minimizing sequences.
   \newblock {\em Calc.  Var.  Par.  Diff.  Eq.},
   1(2):169--204, 1993.

   \bibitem{OhtaKawasaki}
   T. Ohta and K. Kawasaki.
   \newblock Equilibrium morphology of block copolymer melts.
   \newblock {\em Macromolecules}, 19(10):2621--2632, 1986.

   \bibitem{pasta}
   M.~Okamoto, T.~Maruyama, K.~Yabana, and T.~Tatsumi.
   \newblock Nuclear ``pasta'' structures in low-density nuclear matter and
   properties of the neutron-star crust.
   \newblock {\em Phys. Rev. C}, 88:025801, 2013.

   \bibitem{osterNatale}
   K.~Osterwalder and R.~Schrader
   \newblock  {A}xioms for {E}uclidean Green's functions.
   \newblock {\em Comm. Math. Phys}, 31, 83-112, 1973.

   \bibitem{hueberschafer}
   A.~Hubert and R.~Schäfer.
   \newblock{Magnetic domains: the analysis of magnetic microstructures.}
   \newblock{\em Springer Science and Business Media}, 1998.


   \bibitem{Seul476}
   M. Seul and D. Andelman.
   \newblock Domain shapes and patterns: {T}he phenomenology of modulated phases.
   \newblock {\em Science}, 267(5197):476--483, 1995.

   \bibitem{theil2006proof}
   F. Theil.
   \newblock A proof of crystallization in two dimensions.
   \newblock {\em Comm. Math. Phys}, 262(1):209--236, 2006.
\end{thebibliography}
\end{document}